\documentstyle[11pt,amssymb,amsmath]{article}

\newtheorem{theorem}{Theorem}
\newtheorem{remark}[theorem]{Remark}
\newtheorem{lemma}[theorem]{Lemma}

\newtheorem{definition}[theorem]{Definition}

\newcommand{\eofproofmark}{\raisebox{-0.3mm}{\large$\Box$}}
\newcommand{\eofproof}{{\nopagebreak \hspace{0.5cm}\\[-\baselineskip]
       \nopagebreak \mbox{\,}\hfill\nopagebreak\eofproofmark\pagebreak[0]}}
\newenvironment{proof}{{\noindent\bf Proof\ \,}}{\eofproof\bigskip\pagebreak[3]

}
\newcommand{\balra}[2]{\makebox[#1][l]{$\displaystyle {#2}\hss$}}
\def\setR{{\mathbb{R}}}
\def\setN{{\mathbb{N}}}
\newcommand{\st}{\!:\,}
\newcommand{\nl}{\nonumber\\ }
\newcommand{\calA}{{\mathcal{A}}}
\newcommand{\calP}{{\mathcal{P}}}
\renewcommand{\phi}{\varphi}
\renewcommand{\epsilon}{\varepsilon}
\newcommand{\dd}[1]{\frac{d}{d{#1}}}
\newcommand{\Ww}{{W^{1,\infty}}}
\newcommand{\Www}{{W^{2,\infty}}}
\def\calL{{\cal L}}
\def\calB{{\cal B}}
\newcommand{\eref}[1]{{\rm(\mbox{\ref{#1}})}}
\newcommand{\IVP}[2]{IVP \eref{#1}-\eref{#2}}
\newcommand{\Lw}{{L^{\infty}}}
\newcommand{\ballp}[3]{{\calB_{#1}\!\left({#2};\,{#3}\right)}}

\newcommand{\esssup}{\mathop{\rm ess\,sup}}
\newcommand{\essinf}{\mathop{\rm ess\,inf}}
\newcommand{\vc}[1]{{\mathbf #1}}
\newcommand{\comment}[1]{}
\newcommand{\Wp}{{W^{1,p}}}

\begin{document}

\title{On second-order differentiability with respect to parameters for differential equations with 
state-dependent delays\thanks{This research was partially supported by the
Hungarian National Foundation for Scientific Research Grant
No. K73274.}}
\author{Ferenc Hartung\\
Department of Mathematics\\
University of Pannonia\\
8201 Veszpr\'{e}m, P.O. Box 158, Hungary}

\maketitle

\begin{abstract}
In this paper we consider a  class of differential equations
with state-dependent delays. 
We show first and second-order differentiability  of the solution
with respect to parameters in a pointwise sense and also using the $C$-norm on the state-space, assuming that the state-dependent 
time lag function is piecewise strictly monotone. 
\end{abstract}

\textbf{AMS(MOS) subject classification:} 34K05

\textbf{keywords:} Delay differential equation, state-dependent delay, differentiability
with respect to parameters.
\bigskip

\section{Introduction}

\label{sec_DDE_intro}

In this paper we study the SD-DDE
\begin{equation}\label{e1}
 \dot x(t) = f(t,x_t,x(t-\tau(t,x_t,\xi)),\theta),\qquad t\in [0,T],
\end{equation}
and the corresponding initial condition
\begin{equation}\label{e2}
x(t) = \phi(t),\qquad t\in[-r,0].
\end{equation}
Let $\Theta$ and $\Xi$ be normed linear spaces with norms $|\cdot|_\Theta$ and
$|\cdot|_\Xi$, respectively, and suppose $\theta\in\Theta$ and $\xi\in\Xi$.
Here we consider the initial function $\phi$, $\theta$ and $\xi$
as parameters in the IVP \eref{e1}-\eref{e2}, and we denote 
the corresponding solution by $x(t,\phi,\theta,\xi)$. The main goal of this paper is
to discuss the differentiability of $x(t,\phi,\theta,\xi)$ wrt
$\phi$, $\theta$ and $\xi$. 
By differentiability we  mean Fr\'echet-differentiability throughout the manuscript.
Differentiability of solutions wrt parameters is
an important qualitative question, but it also has a natural application
in the problem of identification of parameters  (see \cite{Ha3}).
But even for simple constant delay equations this problem leads to technical 
difficulties if the parameter is the delay \cite{hl, La2}. 
Similar difficulty arises in SD-DDEs. 

Theorem~\ref{t1} below yields that, under 
natural assumptions, Lipschitz continuous initial functions generate unique solutions
of \eref{e1}. As it is common for delay equations, as the time increases, the solution
of \eref{e1} gets smoother wrt the time: on the interval $[0,r]$ the solution is $C^1$, on $[r,2r]$
it is a $C^2$ function, etc. But for $t\in[0,r]$ the solution segment function $x_t$ is only
Lipschitz continuous. Therefore the linearization of the composite function 
$x(t-\tau(t,x_t,\xi))$ is not straightforward, which is clearly needed at some point of the proof to
obtain differentiability wrt parameters.

To illustrate the difficulty of this problem in the case when we can't assume continuous differentiability
of $x$, we recall a result of 
Brokate and Colonius \cite{BC}.  They studied  
SD-DDEs of the form
$$
x'(t) = f\Bigl(t,x(t-\tau(t,x(t)))\Bigr),   \qquad t\in[a,b],
$$
and investigated differentiability of the composition 
operator
$$
A\,:\ \Ww([a,b];\setR)\supset \bar X \to 
	L^p([a,b];\setR),
	\qquad A(x)(t) := x(t-\tau(t,x(t))).
$$
They assumed that $\tau$ is twice continuously differentiable
satisfying $a\leq t-\tau(t,v)\leq b$ for all $t\in[a,b]$ 
and $v\in\setR$, and considered as domain of $A$ the set
$$
\bar X:= \Bigl\{ x\in\Ww([a,b];\setR)\st\mbox{There exists }
	\epsilon>0\ \mbox{s.t. } 
	\frac d{dt}\Bigl(t-\tau(t,x(t))\Bigr)\geq\epsilon\
 \mbox{for  a.e. } t\in[a,b]\Bigr\}.
$$
It was shown in \cite{BC} that under these assumptions $A$ is continuously
differentiable with the derivative given by
$$
(DA(x)u)(t) = -\dot x(t-\tau(t,x(t)))D_2 \tau (t,x(t))u(t)+ u(t-\tau(t,x(t))) 
$$
for $u\in\Ww([a,b],\setR)$.
Both the strong $\Ww$-norm on the domain and the weak $L^p$-norm on the range,
together with  the choice  of the domain 
seemed to be necessary to obtain the results in \cite{BC}.
Note that Manitius in \cite{Man}  used a similar domain and norm
when he studied linearization for a 
 class of SD-DDEs.

Differentiability of solutions wrt parameters for SD-DDEs was studied 
in \cite{CHW,Ha2,Ha4, HT,W2,W3}. 
In \cite{Ha2} differentiability of the parameter map
was established at parameter values
where the compatibility condition
\begin{equation}\label{e403b}
\phi\in C^1,\qquad \dot\phi(0-)=f(0,\phi,\phi(-\tau(0,\phi,\xi)),\theta)
\end{equation}
is satisfied. 
It was proved that the parameter map is differentiable in a pointwise sense, i.e.,   the map
\begin{equation}\label{e410}
\Ww\times\Theta\times\Xi\to\setR^n,\qquad (\phi,\theta,\xi)\mapsto x(t,\phi,\theta,\xi)
\end{equation} 
is differentiable for every fixed $t$ from the domain of the solution.  Moreover, it was shown that 
the map
\begin{equation}\label{e411}
\Ww\times\Theta\times\Xi\to C,\qquad (\phi,\theta,\xi)\mapsto x_t(\cdot,\phi,\theta,\xi),
\end{equation} 
and, under a little more smoothness assumptions,  the map 
\begin{equation}\label{e412}
\Ww\times\Theta\times\Xi\to \Ww,\qquad (\phi,\theta,\xi)\mapsto x_t(\cdot,\phi,\theta,\xi)
\end{equation} 
is also differentiable at  fixed parameter values satisfying \eref{e403b}. 
Note that a condition similar to \eref{e403b} was used by Walter in \cite{W2} and \cite{W3}, 
where proved the existence of a $C^1$-smooth  solution semiflow
for  large classes of SD-DDES.

In \cite{HT} differentiability of the parameter map was 
proved without assuming the compatibility condition \eref{e403b}. Instead,  
it was assumed that the time lag function
$t\mapsto t-\tau(t,x_t,\xi)$ corresponding to a fixed solution $x$ 
is strictly monotone increasing, more precisely,
\begin{equation}\label{e405}
\mathop{\rm ess\,inf}_{0\leq t\leq\alpha} \frac{d}{dt} (t-\tau(t,x_t,\xi))>0,
\end{equation}
where $\alpha>0$ is such that the solution exists on $[-r,\alpha]$.
Also, instead of a ``pointwise'' differentiability, the differentiability of
the map
$$
\Ww\times\Theta\times\Xi\to \Wp,\qquad (\phi,\theta,\xi)\mapsto x_t(\cdot,\phi,\theta,\xi)
$$
was proved in a small neighborhood 
of the fixed parameter value. Note that here
 the differentiability was obtained using only a weak norm, the $W^{1,p}$-norm 
($1\leq p<\infty$) on the state-space.

Chen, Hu and Wu in \cite{CHW} extended the above result to proving second ordered
differentiability of the parameter map using the monotonicity condition \eref{e405} 
of the state-dependent time lag function,  the $W^{1,p}$-norm ($1\leq p<\infty$) 
on the state space, and the $W^{2,p}$-norm on the space of initial functions. 
Note that $\tau$  was not given explicitly  in \cite{CHW},
it was defined through a coupled differential equation, but it satisfied the
monotonicity condition \eref{e405}.

In \cite{Ha4} the IVP
\begin{eqnarray}
 \dot x(t) &=& f(t,x_t,x(t-\tau(t,x_t))),\qquad t\in[\sigma,T], \label{e413}\\
x(t) &=& \phi(t-\sigma),\qquad t\in[\sigma-r,\sigma]
\end{eqnarray}
was considered.
In this IVP the parameters $\theta$ and $\xi$ were omitted for simplicity,
but the initial time $\sigma$ was considered together with the initial function
as parameters in the equation.
Combining the techniques of \cite{Ha2} and \cite{HT}, and assuming
the appropriate monotonicity condition \eref{e405}, but without assuming the
compatibility condition \eref{e403b}, the continuous differentiability of the parameter maps
$$
\Ww\to\setR^n,\qquad \phi\mapsto x(t,\sigma,\phi)
$$
and
$$
\Ww\to C,\qquad \phi\mapsto x_t(\cdot,\sigma,\phi)
$$
were proved
for a fixed $t$ and $\sigma$ in a neighborhood of a fixed initial function. 
Note that with this technique similar result can't  be given using the $\Ww$-norm on the
state-space without using the compatibility condition. 

Assuming the compatibility condition \eref{e403b} it was also shown in \cite{Ha4}
that the maps
$$
[0,\alpha)\to\setR^n,\qquad \sigma\mapsto x(t,\sigma,\phi)
$$
and 
$$
[0,\alpha)\to C,\qquad \sigma\mapsto x_t(\cdot,\sigma,\phi)
$$
are differentiable for all $t\in[\sigma-r,\alpha]$ and $t\in[\sigma,\alpha]$, respectively,
and $\sigma$, $\phi$ in a neighborhood 
of a fixed parameter $(\sigma,\phi)$, and where $\alpha>0$ is a certain
constant. 
Assuming that the functions $f$ and $\tau$ have a special form in \eref{e413}, i.e., for equations of the form
\begin{eqnarray*}
\dot x(t)&=&\bar f\Bigl(t,x(t-\lambda^1(t)),\ldots,x(t-\lambda^m(t)),
\int_{-r}^0 A(t,\theta)x(s+\theta)\,ds,\nl
&&\quad x\Bigl(t-\bar\tau\Bigl[t,x(t-\xi^1(t)),\ldots,x(t-\xi^\ell(t)),
\int_{-r}^0 B(t,\theta)x(s+\theta)\,ds\Bigr]\Bigr)\Bigr)
\end{eqnarray*}
the differentiability
of the map 
$$
[0,\alpha)\to\setR^n,\qquad \sigma\mapsto x(t,\sigma,\phi)
$$
was shown in \cite{Ha4} for $t\in[\sigma,\alpha]$  using the monotonicity assumption \eref{e405}, but
without the compatibility condition \eref{e403b}. Note that in this case similar result 
does not hold for the map $\sigma\mapsto x_t(\cdot,\sigma,\phi)$ using the $C$-norm, 
which is not surprising, since it is easy to see \cite{Ha4} that the map 
$\sigma\mapsto x(t,\sigma,\phi)$ is differentiable at the 
point $t=\sigma$ if and only if a compatibility condition similar to \eref{e403b} is satisfied.

We refer the interested reader 
for related works on dependence of the solutions on parameters in SD-DDEs to \cite{Sl1, Sl2},
and for similar works in neutral SD-DDEs to \cite{Ha5, Ha7, W5}.

The organization of this paper is the following. 
In Section~\ref{sec_def} we summarize some notations and preliminary results that will be used 
in the manuscript.
In Section~\ref{sec_DDE_wp} first we list the detailed assumptions on the \IVP{e1}{e2}
we will need 
in our differentiability results later, and formulate a well-posedness result
(Theorem~\ref{t1}) concerning the \IVP{e1}{e2}, and prove some estimates will be
essential later.

In Section~\ref{sec_DDE_first} using and extending the method introduced in \cite{Ha4}, we discuss 
first order differentiability of the parameter maps associated to the \IVP{e1}{e2}.
In the main result of this section (see Theorem~\ref{t2} below) we show
the differentiability of the parameter maps \eref{e410} and \eref{e411}
 without using
the compatibility condition \eref{e403b}, and also relaxing the monotonicity condition
\eref{e405} to the condition that the time lag function $t\mapsto t-\tau(t,x_t,\xi)$
is ``piecewise strictly monotone'' in the sense of  Definition~\ref{d1}.
Note that omitting the compatibility condition is essential in the application 
of this results in \cite{Ha6}, where
we prove the convergence of the quasilinearization method
in the problem of parameter estimation. 
Also, in this application the existence of the derivative is needed in
this strong, pointwise sense, i.e., the differentiability of the map \eref{e410}
is used in \cite{Ha6}. 

In Section~\ref{sec_DDE_second} the main result is Theorem~\ref{t3}, which proves twice 
continuous differentiability of the maps 
$$
\Www\times\Theta\times\Xi\to \setR^n,\qquad (\phi,\theta,\xi)\mapsto x(t,\phi,\theta,\xi)
$$
and
$$
\Www\times\Theta\times\Xi\to C,\qquad (\phi,\theta,\xi)\mapsto x_t(\cdot,\phi,\theta,\xi)
$$
 at a parameter value $(\phi,\theta,\xi)$ satisfying the
compatibility condition \eref{e403b} and such that the corresponding time lag function
$t\mapsto\tau(t,x_t,\xi)$ is piecewise strictly monotone in the sense of Definition~\ref{d1}.
Under some additional condition, the continuity of the second derivative wrt the parameters
is obtained in a certain sense.  The only result known in the literature 
for the existence of a second derivative wrt the parameters in SD-DDEs is the result of Chen, Hu and Wu \cite{CHW}, where
the second order differentiability is proved only using a weak $\Wp$-norm on the state-space. 
Note that our result shows the existence of
the second derivative in a pointwise sense, i.e., at each fixed $t$, moreover, the technique of the proof is simpler.
\bigskip

\bigskip\pagebreak

\section{Notations and preliminaries}
\label{sec_def}
\setcounter{equation}0
\setcounter{theorem}0


Throughout the manuscript $r>0$ is a fixed
constant and $x_t\st [-r,0]\to\setR^n$, $x_t(\theta):=x(t+\theta)$ is the segment 
function. To avoid confusion with the notation of the segment function, sequences of functions are denoted
using the upper index: $x^k$.

$\setN$ and $\setN_0$  denote the set of positive and nonnegative integers, respectively.
A fixed norm on $\setR^n$ and its induced matrix norm on $\setR^{n\times n}$ 
are both denoted by $|\cdot|$. $C$  denotes 
the Banach space of continuous functions   
$\psi\st[-r,0]\to\setR^n$ equipped with the norm $|\psi|_C=\max\{|\psi(\zeta)|\st
\zeta\in[-r,0]\}$. 
$C^1$ is the space of continuously differentiable functions 
$\psi\st[-r,0]\to\setR^n$ where the norm is defined by $|\psi|_{C^1}=\max\{|\psi|_C,|\dot\psi|_C\}$.
$\Lw$ is the space of Lebesgue-measurable functions $\psi\st[-r,0]\to\setR^n$ which are
essentially bounded. The norm on $\Lw$ is denoted by $|\psi|_\Lw= \esssup\{|\psi(\zeta)|\st \zeta\in[-r,0]\}$.
$W^{1,p}$ denotes the Banach-space of absolutely continuous functions $\psi\st[-r,0]\to\setR^n$
of finite norm defined by
$$
|\psi|_{W^{1,p}} :=\left(\int_{-r}^0 |\psi(\zeta)|^p
	+|\dot \psi(\zeta)|^p\,d\zeta\right)^{1/p},\qquad 1\leq p<\infty,
$$
and for $p=\infty$
$$
|\psi|_{\Ww} := \max\left\{|\psi|_C, |\dot \psi|_\Lw\right\}.
$$
We note that $\Ww$ is equal to the space of Lipschitz continuous functions from $[-r,0]$ to $\setR^n$.
The subset of $\Ww$ consisting of those functions which have absolutely continuous first derivative and
essentially bounded second derivative
is denoted by $\Www$, where the norm is defined by 
$$
|\psi|_{\Www} := \max\left\{|\psi|_C,\ |\dot \psi|_C,\  |\ddot \psi|_\Lw\right\}.
$$ 
If the domain or the range of the functions is different from $[-r,0]$ and $\setR^n$, respectively,
we will use a more detailed notation. E.g., $C(X,Y)$ denotes the space of continuous functions mapping from 
$X$ to $Y$.
Finally, $\calL(X,Y)$ denotes the space of bounded linear operators from $X$ to $Y$,
where $X$ and $Y$ are normed linear spaces.
An open ball in the normed linear space $X$ centered at a point $x\in X$ with radius 
$\delta$ is denoted by $\ballp{X}{x}{\delta}:=\{ y\in Y\st |x-y|<\delta\}$.

The derivative of a single variable function $v(t)$ wrt $t$ is 
denoted by $\dot v$. Note that all derivatives we use in this paper are Fr\'echet derivatives.
The partial derivatives of a function $g\st X_1\times X_2\to Y$ wrt the first and second variables will be denoted by
$D_1g$ and $D_2g$, respectively.  
The second-order partial derivative wrt its $i$th and $j$th  variables ($i,j=1,2$)
of the function  $g\st X_1\times X_2\to Y$ at the point $(x_1,x_2)\in X_1\times X_2$ is
the bounded bilinear operator $A\langle \cdot,\cdot\rangle\st X_i\times X_j\to Y$, if 
$$
\lim_{k\to0}\sup_{h\neq 0}\frac{|D_i g(x_1+k\delta_{1j},x_2+k\delta_{2j})h - D_i g(x_1,x_2)h-A\langle h,k\rangle|_Y}{|h|_{X_i}|k|_{X_1}}=0,
\qquad h\in X_i,\ k\in X_j,
$$
where $\delta_{ij}=1$ for $i=j$ and $\delta_{ij}=0$ for $i\neq j$ is the Kronecker-delta.
We will use the notation $D_{ij}g(x_1,x_2)=A$. The norm of the bilinear operator $A\langle \cdot,\cdot\rangle\st X_i\times X_j\to Y$
is defined by
$$
|A|_{\calL^2(X_i\times X_j,Y)}:=\sup
\left\{\frac{|A\langle h,k\rangle|_Y}{|h|_{X_i}|k|_{X_j}}\st h\in X_i, h\neq0,\ k\in X_j, k\neq0 \right\}.
$$
In the case when $X_1=\setR$, we simply write $D_1g(x_1,x_2)$ instead of the more precise
notation $D_1g(x_1,x_2)1$, i.e., here $D_1g$ denotes the value in
$Y$ instead of the linear operator $\calL(\setR,Y)$. In the case when, let say, $X_2=\setR^n=Y$,
then we identify the linear operator $D_2g(x_1,x_2)\in\calL(\setR^n,\setR^n)$ by an $n\times n$
matrix.
\bigskip

Next we formulate a result which is a simple consequence of the Gronwall's lemma.
\medskip

\begin{lemma}[see, e.g., \cite{Ha4}]\label{l0}
Suppose $a>0$, $b\st [0,\alpha]\to [0,\infty)$ and $u\st[-r,\alpha]\to \setR^n$ 
are continuous functions such that $a\geq |u_0|_C$, and
\begin{equation}\label{e301}
|u(t)|\leq a +\int_0^t b(s) |u_s|_C\,ds,\qquad t\in[0,\alpha].
\end{equation}
Then
\begin{equation}\label{e303}
|u(t)|\leq |u_t|_C\leq a e^{\int_0^\alpha b(s)\,ds},\qquad t\in[0,\alpha].
\end{equation}
\end{lemma}
\medskip


\begin{lemma}\label{ll12b}
Suppose $\psi\in\Ww$. Then 
$$
|\psi(b)-\psi(a)|\leq |\dot \psi|_\Lw|b-a|
$$
for every $[a,b]\subset[-r,0]$. 
\end{lemma}
\medskip

We recall the following result from \cite{BC}, which was essential to prove differentiability
wrt parameters in SD-DDEs in \cite{CHW}, \cite{Ha4} and \cite{HT}. We state the result in a simplified form we
need later,  it is formulated in a more general form in \cite{BC}. 
Note that the second part of the lemma was 
stated in \cite{BC} under the assumption $|u^k-u|_{\Ww([0,\alpha],\setR)}\to0$ as $k\to\infty$,
but this stronger assumption on the convergence is  not needed in the proof. 
See also the proof of Lemma 4.26 in \cite{Ha1}.
\medskip

\begin{lemma}[\cite{BC}]\label{l1}
Let  $g\in L^1([c,d],\setR^n)$, $\epsilon>0$, and  $u\in\calA(\epsilon)$,
where
$$
 \calA(\epsilon):= 
\{ v\in W^{1,\infty}([a,b],[c,d])\,:\, 
\dot v(s)\geq \epsilon\ \mbox{for a.e.\ } s\in[a,b]\}.
$$
Then
\begin{equation}\label{e10}
\int_a^b|g(u(s))|\, ds\leq \frac 1\epsilon \int_{c}^d |g(s)|\,ds.
\end{equation}
Moreover, if the sequence $u^k\in\calA(\epsilon)$ is such that $|u^k-u|_{C([a,b],\,\setR)}\to0$
as $k\to\infty$, then
\begin{equation}\label{e68}
\lim_{k\to\infty}\int_a^b\Bigl|g(u^k(s))-g(u(s))\Bigr|\,ds=0.
\end{equation}
\end{lemma}
\bigskip

\begin{remark}\label{r1}\rm
Changing to the new variable $s=-t$ in the integrals in \eref{e10} and \eref{e68} give easily that
 the statements of Lemma~\ref{l1} hold also in the case when conditions $u,u^k \in\calA(\epsilon)$ are replaced
by $-u,-u^k \in\calA(\epsilon)$.
\end{remark}

In the next lemma we relax the condition $u\in\calA(\epsilon)$ of the previous lemma.
\bigskip

\begin{lemma}\label{l2}
 Suppose $g\in\Lw([c,d],\setR)$, and $u\st[a,b]\to[c,d]$ is an absolutely continuous
function, 
and
\begin{equation}\label{e311}
\essinf \{ \dot u(s)\st s\in[a',b']\}>0,\qquad \mbox{for all }\ [a',b']\subset(a,b).
\end{equation}
Then the composite function $g\circ u\in\Lw([a,b],\setR)$, and
$|g\circ u|_{\Lw([a,b],\setR)}\leq |g|_{\Lw([c,d],\setR)}$.
\end{lemma}
\begin{proof}
First note that since $u$ is absolutely continuous, it is a.e.\ differentiable on $[a,b]$,
and condition \eref{e311} yields that $u$ is strictly monotone increasing on $[a,b]$.
Let $G:=\{v\in[c,d]\st g(v)\ \mbox{is not defined or } |g(v)|>|g|_{\Lw([c,d],\setR)}\}$.
Then $\mathop{meas}(G)=0$. Let $A:=\{t\in[a,b]\st g(u(t)) \ 
\mbox{is not defined or } |g(u(t))|>|g|_{\Lw([c,d],\setR)}\}$. Clearly, $A=u^{-1}(G)$.
Let $0<\epsilon<(b-a)/2$ be fixed. Then let $c':=u(a+\epsilon)$,  $d':=u(b-\epsilon)$, and
let $M:=\essinf \{ \dot u(s)\st s\in[a+\epsilon,b-\epsilon]\}$. Then \eref{e311} yields
$M>0$. Since $G$ is of measure 0, there exist open intervals $(c_i,d_i)$, $i\in\setN$ 
such that 
$$
G\subset \bigcup_{i=1}^\infty (c_i,d_i)\quad \mbox{and}\quad \sum_{i=1}^\infty (d_i-c_i)
< \epsilon M.
$$ 
We have 
$$
A=u^{-1}(G)= u^{-1}\Bigl(G\cap[c,c']\Bigr)\cup u^{-1}\Bigl(G\cap[c',d']\Bigr)
\cup u^{-1}\Bigl(G\cap[d',d]\Bigr),
$$
and the monotonicity of $u$ yields
$u^{-1}\Bigl(G\cap[c,c']\Bigr)\subset [a,a+\epsilon]$,\
$u^{-1}\Bigl(G\cap[d',d]\Bigr)\subset [b-\epsilon,b]$,
and
$$
u^{-1}\Bigl(G\cap[c',d']\Bigr)
\subset u^{-1}\Bigl([c',d']\cap\displaystyle\bigcup_{i=1}^\infty [c_i,d_i]\Bigr)
= \bigcup_{i=1}^\infty u^{-1}\Bigl([c',d']\cap [c_i,d_i]\Bigr)
= \bigcup_{i=1}^\infty [a_i,b_i],
$$
where $a_i:=u^{-1}(\max\{c',c_i\})$ and $b_i:=u^{-1}(\min\{d',d_i\})$.
The definition of $M$ yields
$$
d_i-c_i\geq\min\{d',d_i\}-\max\{c',c_i\}=u(b_i)-u(a_i)
=\int_{a_i}^{b_i}\dot u(s)\,ds\geq M(b_i-a_i).
$$
Therefore $A\subset [a,a+\epsilon]\cup[b-\epsilon,b]\cup \bigcup_{i=1}^\infty [a_i,b_i]$,
and the sum of the length of the closed intervals covering $A$ is less than $3\epsilon$.
Since $\epsilon>0$ is arbitrary, we get that $A$ is Lebesgue-measurable and $\mathop{meas}(A)=0$.

We show that $g\circ u$ is Lebesgue-measurable. Let $\kappa\in\setR$, and
define $G_\kappa:=\{v\in[c,d]\st g(v)\ \mbox{is defined and } g(v)<\kappa\}$.
 $G_\kappa$ is a Lebesgue-measurable set,
since $g\in\Lw([c,d],\setR)$. Therefore
there exists a closed set $F_\kappa$ such that $F_\kappa\subset G_\kappa$ and
$\mathop{meas}(G_\kappa\setminus F_\kappa)=0$. Since $u$ is continuous,
$u^{-1}(F_\kappa)$ is a closed set, and therefore, it is Lebesgue-measurable.
Moreover, $u^{-1}(G_\kappa)=u^{-1}(F_\kappa)\cup u^{-1}(G_\kappa\setminus F_\kappa)$,
and as in the first part of the proof, we get that $u^{-1}(G_\kappa\setminus F_\kappa)$
is measurable, and so is $u^{-1}(G_\kappa)$. 
\end{proof}
\medskip

Clearly, the statement of the previous Lemma is also valid if 
\eref{e311} is changed to 
$$
\esssup \{ \dot u(s)\st s\in[a',b']\}<0,\qquad \mbox{for all }\ [a',b']\subset(a,b).
$$
\medskip

We will use the following notation.
\medskip

\begin{definition}\label{d1}
 ${\cal PM}([a,b],[c,d])$ denotes the
set of absolutely continuous functions $u\st[a,b]\to[c,d]$ which are
 piecewise strictly monotone on $[a,b]$ in the sense that there exists a finite mesh
$a=t_0<t_1<\cdots<t_{m-1}<t_m=b$ of $[a,b]$ such that for all $i=0,1,\ldots,m-1$
 either 
$$
\essinf \{ \dot u(s)\st s\in[a',b']\}>0,\qquad \mbox{for all }\ [a',b']\subset(t_i,t_{i+1})
$$
or
$$
\esssup \{ \dot u(s)\st s\in[a',b']\}<0,\qquad \mbox{for all }\ [a',b']\subset(t_i,t_{i+1}).
$$
\end{definition}
\medskip

Lemma~\ref{l2} implies the next result immediately.
\medskip

\begin{lemma}\label{l3}
 Suppose $g\in\Lw([c,d],\setR^n)$, and $u\in{\cal PM}([a,b],[c,d])$. 
Then the composite function $g\circ u\in\Lw([a,b],\setR^n)$ and $|g\circ u|_{\Lw([a,b],\,\setR^n)}
\leq|g|_{\Lw([c,d],\,\setR^n)}$.
\end{lemma}
\medskip

The next lemma generalizes the convergence property \eref{e68} to the class ${\cal PM}$.
We comment that 
to prove the convergence property \eref{e68} for $u,u^k\in{\cal PM}([a,b],[c,d])$, we need the
stronger assumption $|u^k-u|_{\Ww([a,b],\,\setR)}\to0$ instead of
$|u^k-u|_{C([a,b],\,\setR)}\to0$ what is used in Lemma~\ref{l1}.
\medskip

\begin{lemma}\label{l4}
 Suppose $g\in\Lw([c,d],\setR^n)$,  and
 $u, u^k\in{\cal PM}([a,b],[c,d])$ ($k\in\setN$)
satisfying 
\begin{equation}\label{e312}
|u^k-u|_{\Ww([a,b],\,\setR)}\to0,\qquad \mbox{as\ } k\to\infty.
\end{equation}
Then 
\begin{equation}\label{e313}
\int_a^b |g(u^k(s))-g(u(s))|\,ds\to0,\qquad \mbox{as\ } k\to\infty.
\end{equation}
\end{lemma}
\begin{proof}
Clearly, it is enough to show \eref{e313} for the case when $g$ is real valued, i.e., $n=1$.

First note that Lemma~\ref{l3} yields  $g\circ u,\, g\circ u^k\in \Lw([a,b],\setR)$.
We prove \eref{e313} in three steps. 

(i)
First suppose that $g\in\Lw([c,d],\setR)$ is the characteristic function of an
interval $[e,f]\subset[c,d]$, i.e., $g=\chi_{[e,f]}$.
Then $|\chi_{[e,f]}(u^k(s))-\chi_{[e,f]}(u(s))|$ is either 0 or 1,  hence
$$
\mathop{meas}(\{s\in[a,b]\st \chi_{[e,f]}(u^k(s))\neq \chi_{[e,f]}(u(s))\})
\leq 4|u^k-u|_{C([a,b],\setR)},
$$
and so
$$
\int_a^b |\chi_{[e,f]}(u^k(s))-\chi_{[e,f]}(u(s))|\,ds\leq 4|u^k-u|_{C([a,b],\setR)}\to0,
\qquad \mbox{as\ } k\to\infty.
$$

(ii) Suppose $g$ is a step function, i.e., $g=\sum_{i=1}^m c_i\chi_{A_i}$,
where $A_i$ are pairwise disjoint intervals with $\cup_{i=1}^m A_i=[c,d]$.
Then
$$
\int_a^b |g(u^k(s))-g(u(s))|\,ds
\leq \sum_{i=1}^m |c_i|4|u^k-u|_{C([a,b],\setR)}
\to0,\qquad \mbox{as\ } k\to\infty.
$$
 
(iii) Let $a=t_0<t_1<\cdots<t_m=b$ be the mesh points of $u$ from the Definition~\ref{d1}, and 
let $0<\epsilon<\min\{t_{i+1}-t_i\st i=0,\ldots,m-1\}/2$ be fixed, and introduce  $t_i':=t_i+\epsilon$ for
$i=0,\ldots,m-1$ and $t_i'':=t_i-\epsilon$ for $i=1,\ldots,m$, $t_0'':=a$, $t_m':=b$, and let
\begin{equation}\label{e369}
M:=\min_{i=0,\ldots,m-1}\essinf_{t\in[t_i',t_{i+1}'']} |\dot u(t)|.
\end{equation}
We have $M>0$, since $u\in{\cal PM}([a,b],[c,d])$.

 The set of step functions is dense in $L^1([c,d],\setR)$ (see, e.g., \cite{C}), so for a fixed 
$g\in\Lw([c,d],\setR)$ and $0<\delta<\epsilon M/m$ there exists a step function $h\st[c,d]\to\setR$ 
such that
$|g-h|_{L^1([c,d],\setR)}<\delta$. Let $h=\sum_{i=1}^m c_i\chi_{A_i}$, 
where $A_i$ are pairwise disjoint intervals with $\cup_{i=1}^m A_i=[c,d]$,
and define $h^*:=\sum_{i=1}^m c_i^*\chi_{A_i}$, where
$$
c_i^*:=\left\{\begin{array}{ll}
c_i,\quad & \mbox{if } |c_i|\leq |g|_{\Lw([c,d],\setR)}+1,\\
|g|_{\Lw([c,d],\setR)}, & \mbox{if } c_i> |g|_{\Lw([c,d],\setR)}+1,\\
-|g|_{\Lw([c,d],\setR)}, & \mbox{if } c_i< -|g|_{\Lw([c,d],\setR)}-1.\\
\end{array}
\right.
$$
Then it is easy to check that $|g(v)-h^*(v)|\leq  1$
for a.e.\ $v\in[c,d]$, and
$$
\int_c^d |g(v)-h^*(v)|\,dv\leq \int_c^d |g(v)-h(v)|\,dv<\delta.
$$
We have therefore
\begin{eqnarray*}
\balra{1cm}{\int_a^b |g(u(s))-h^*(u(s))|\,ds}\nl
&=&\sum_{i=0}^{m} \int_{t_i''}^{t_i'} |g(u(s))-h^*(u(s))|\,ds+\sum_{i=0}^{m-1} \int_{t_i'}^{t_{i+1}''} |g(u(s))-h^*(u(s))|\,ds\nl
&\leq& 2\epsilon (m+1)+\sum_{i=0}^{m-1} \int_{t_i'}^{t_{i+1}''}  |g(u(s))-h^*(u(s))|\dot u(s)\frac 1{\dot u(s)}\,ds\nl
&\leq& 2\epsilon (m+1)+\frac1M\sum_{i=0}^{m-1}\left| \int_{u(t_i')}^{u(t_{i+1}'')}  |g(v)-h^*(v)|\,dv\right|\nl
&\leq& 2\epsilon (m+1)+\frac{\delta m}M\nl
&\leq& (2m+3)\epsilon.
\end{eqnarray*}
Assumption \eref{e312} yields that there exist $k_0>0$ such that $|u^k-u|_{\Ww([a,b],\,\setR)}<\frac M2$ for $k\geq k_0$. 
Then for $k\geq k_0$ it follows $|\dot u^k(s)|\geq \frac M2$ for a.e.\ $s\in[t_i',t_{i+1}'']$ and $i=0,\ldots,m-1$.
Therefore similarly to the previous estimate 
 we have for $k\geq k_0$
$$
\int_a^b |g(u^k(s))-h^*(u^k(s))|\,ds\leq 2\epsilon (m+1)+\frac{2\delta m} M\leq (2m+4)\epsilon.
$$
Using the above inequalities we get 
\begin{eqnarray*}
\balra{1cm}{\int_a^b |g(u^k(s))-g(u(s))|\,ds }\nl
&\leq& \int_a^b |g(u^k(s))-h^*(u^k(s))|\,ds+\int_a^b |h^*(u^k(s))-h^*(u(s))|\,ds\nl
&&+\int_a^b |g(u(s))-h^*(u(s))|\,ds\nl
&\leq& (4m+7)\epsilon +\int_a^b |h^*(u^k(s))-h^*(u(s))|\,ds,\qquad k\geq k_0,
\end{eqnarray*}
which yields \eref{e313} using part (ii), since $\epsilon>0$ is arbitrary close to 0.
\end{proof}
\pagebreak[3]\medskip

\begin{lemma}\label{l21}
 Suppose $f^{k,h}\in \Lw([c,d],\setR^n)$ for $k\in\setN$ and $h\in H$ for some fixed parameter set $H$,
$$
\lim_{k\to\infty}\sup_{h\in H}\int_c^d |f^{k,h}(s)|\,ds=0,
$$
and there exists $A\geq0$ such that $|f^{k,h}(s)|\leq A$ for $k\in\setN$, $h\in H$ and a.e.\ $s\in[c,d]$.
Let $u, u^k\in{\cal PM}([a,b],[c,d])$ ($k\in\setN$)
be such that \eref{e312} holds.
Then
$$
\lim_{k\to\infty}\sup_{h\in H}\int_a^b |f^{k,h}(u^k(s))|\,ds=0.
$$
\end{lemma}
\begin{proof}
Let $a=t_0<t_1<\cdots<t_m=b$ be the mesh points of $u$ from the Definition~\ref{d1}, and 
let $0<\epsilon<\min\{t_{i+1}-t_i\st i=0,\ldots,m-1\}/2$ be fixed,  let $t_i'$ and $t_i''$
be defined as in the proof of Lemma~\ref{l4},
 and let
$M$ be defined by \eref{e369}. Let $k_0$ be such that $|u^k-u|_{\Ww([a,b],\,\setR)}\leq M/2$ for $k\geq k_0$.
Then for $k\geq k_0$ it follows $|\dot u^k(s)|\geq \frac M2$ for a.e.\ $s\in[t_i',t_{i+1}'']$ and $i=0,\ldots,m-1$.
Since $u^k\in{\cal PM}([a,b],[c,d])$, it follows from Lemma~\ref{l3} that $|f^{k,h}(u^k(s))|\leq A$ for $k\in\setN$, $h\in H$ 
and a.e.\ $s\in[a,b]$.
Therefore for any $k\in\setN$ and $h\in H$ we have
\begin{eqnarray*}
 \int_a^b |f^{k,h}(u^k(s))|\,ds
&=& \sum_{i=0}^{m} \int_{t_i''}^{t_i'} |f^{k,h}(u^k(s))|\,ds+\sum_{i=0}^{m-1} \int_{t_i'}^{t_{i+1}''}|f^{k,h}(u^k(s))|\,ds\nl
&\leq& (m+1)A 2\epsilon+\frac {2m}M\int_c^d |f^{k,h}(s)|\,ds.
\end{eqnarray*}
Then
$$
\sup_{h\in H}\int_a^b |f^{k,h}(u^k(s))|\,ds\leq (m+1)A 2\epsilon+\sup_{h\in H}\frac {2m}M\int_c^d |f^{k,h}(s)|\,ds,
$$
which proves the statement, since $\epsilon$ is arbitrarily close to 0.
\end{proof}

\medskip

\section{Well-posedness and continuous dependence on parameters}\label{sec_DDE_wp}
\setcounter{equation}0
\setcounter{theorem}0

In this section we list all the assumptions we need later on the \IVP{e1}{e2}, and show
some basic results including the well-posedness of the IVP and Lipschitz continuous dependence
of the solutions on the parameters $\phi$, $\theta$ and $\gamma$.

Suppose
 $\Omega_1\subset C$, $\Omega_2\subset\setR^n$, $\Omega_3\subset\Theta$, $\Omega_4\subset\Xi$
are open subsets of the respective spaces. $T>0$ is finite or $T=\infty$,
in which case $[0,T]$ denotes the interval $[0,\infty)$.

We assume
\begin{itemize}
\item[(A1)]  
\begin{itemize}
\item[(i)] $f\, :\, \setR\times C\times\setR^n\times\Theta\supset [0,T]\times \Omega_1\times\Omega_2\times\Omega_3\to\setR^n$ is continuous; 
\item[(ii)] $f(t,\psi,u,\theta)$ is locally Lipschitz continuous in $\psi$, $u$ and $\theta$, i.e.,
	 for every finite 
	$\alpha\in(0,T]$, 
	for every closed subset $M_1\subset \Omega_1$ of $C$ which is also a bounded subset of $\Ww$,	 
	compact subset $M_2\subset \Omega_2$ of $\setR^n$,  
	and closed and bounded subset 
	$M_3\subset \Omega_3$ of $\Theta$ 
	there exists a constant $L_1=L_1(\alpha,M_1,M_2,M_3)$ such that
	$$
	|f(t,\psi,u,\theta)-f(t,\bar \psi,\bar u,\bar\theta)|
	\leq L_1\Bigl(|\psi-\bar \psi|_C+|u-\bar u|+|\theta-\bar\theta|_\Theta\Bigr),
	$$
	for $t\in[0,\alpha]$, $\psi, \bar \psi\in M_1$, $u, \bar u\in M_2$ and $\theta,\bar\theta\in M_3$;
\item[(iii)] 
	$f\,:\,\setR\times C\times\setR^n\times\Theta\supset [0,T]\times\Omega_1\times\Omega_2\times\Omega_3
	\to\setR^n$  
	is continuously differentiable wrt its second, third and fourth	arguments;
\item[(iv)] $f(t,\psi,u,\theta)$ is locally Lipschitz continuous wrt $t$, i.e., 
 for every finite $\alpha\in(0,T]$, 
	for every closed subset $M_1\subset \Omega_1$ of $C$ which is also a bounded subset of $\Ww$,	 
	compact subset $M_2\subset \Omega_2$ of $\setR^n$,  
	and closed and bounded subset 
	$M_3\subset \Omega_3$ of $\Theta$ 
	there exists a constant $L_1=L_1(\alpha,M_1,M_2,M_3)$ such that
	$$
	|f(t,\psi,u,\theta)-f(\bar t, \psi,u,\theta)|
	\leq L_1|t-\bar t|
	$$
	for $t,\bar t\in[0,\alpha]$, $\psi\in M_1$, $u\in M_2$ and $\theta\in M_3$;
\item[(v)] $D_2f$, $D_3f$ and $D_4f$ are locally Lipschitz continuous wrt all of their 
    arguments, i.e., for every finite $\alpha\in(0,T]$, 
	for every closed subset $M_1\subset \Omega_1$ of $C$ which is also a bounded subset of $\Ww$,	 
	compact subset $M_2\subset \Omega_2$ of $\setR^n$,  
	and closed and bounded subset 
	$M_3\subset \Omega_3$ of $\Theta$ there exists $L_3=L_3(\alpha,M_1,M_2,M_3)$ such that
	$$
	|D_if(t,\psi,u,\theta)-D_if(\bar t, \bar \psi,\bar u,\bar \theta)|_{\calL(Y_i,\,\setR^n)} 
	\leq L_3\Bigl(|t-\bar t|+|\psi-\bar \psi|_C+|u-\bar u|+|\theta-\bar\theta|_\Theta\Bigr)
	$$
	for $i=2,3,4$, $t,\bar t\in[0,\alpha]$, $\psi,\bar \psi\in M_1$, $u,\bar u\in M_2$ and $\theta,\bar\theta\in M_3$,
        where $Y_2:=C$, $Y_3:=\setR^n$ and $Y_4:=\Theta$;
\item[(vi)] $D_2f$, $D_3f$ and $D_4f$ are continuously differentiable wrt  their 
    second, third and fourth arguments on $[0,T]\times\Omega_1\times\Omega_2\times\Omega_3$;
\end{itemize}
\item[(A2)] 
\begin{itemize}
\item[(i)] $\tau\, :\, \setR\times C\times\Xi\supset[0,T]\times \Omega_1\times\Omega_4\to[0,r]\subset\setR$ is 
	continuous; 
\item[(ii)] $\tau(t,\psi,\xi)$ is locally Lipschitz continuous in 
	$\psi$ and $\xi$ in the following sense:  
	for every finite $\alpha\in(0,T]$, closed subset $M_1\subset \Omega_1$ of $C$ 
	which is also a bounded subset of $\Ww$, and closed and bounded subset $M_4\subset\Omega_4$
	 of $\Xi$ 
	there exists a constant $L_2=L_2(\alpha,M_1,M_4)$ such that
$$
|\tau(t,\psi,\xi)-\tau(t,\bar\psi,\bar\xi)| \leq L_2\Bigl( |\psi-\bar\psi|_C+|\xi-\bar\xi|_\Xi\Bigr)
$$
	for  $t\in[0,\alpha]$, $\psi,\bar\psi\in M_1$, $\xi,\bar\xi\in M_4$;
\item[(iii)] $\tau\, :\, [0,T]\times C\times\Xi\supset [0,T]\times\Omega_1\times\Omega_4
	\to \setR$ is continuously 
	differentiable wrt its second and third arguments;
\item[(iv)] $\tau(t,\psi,\xi)$ is locally Lipschitz continuous in 
	$t$, i.e.,  
	for every finite $\alpha\in(0,T]$, closed subset $M_1\subset \Omega_1$ of $C$ 
	which is also a bounded subset of $\Ww$, and closed and bounded subset $M_4\subset\Omega_4$
	 of $\Xi$ 
	there exists a constant $L_2=L_2(\alpha,M_1,M_4)$ such that
$$
|\tau(t,\psi,\xi)-\tau(\bar t,\psi,\xi)| \leq L_2|t-\bar t|
$$
	for  $t,\bar t\in[0,\alpha]$, $\psi\in M_1$, $\xi\in M_4$;
\item[(v)]  for every finite $\alpha\in(0,T]$, closed subset $M_1\subset \Omega_1$ of $C$ 
	which is also a bounded subset of $\Ww$, and closed and bounded subset $M_4\subset\Omega_4$
	 of $\Xi$  there exists $L_4=L_4(\alpha,M_1,M_4)\geq0$ such that 
$$
\Bigl| \frac d{dt} \tau(t,y_t,\xi)-\frac d{dt}\tau(t,\bar y_t,\bar \xi)\Bigr|\leq L_4\Bigl(|y_t-\bar y_t|_\Ww+|\xi-\bar \xi|_\Xi\Bigr),
\quad \mbox{a.e. } t\in[0,\alpha],
$$
where $\xi,\bar\xi\in M_4$, and  $y,\bar y\in\Ww([-r,\alpha],\setR^n)$ are such that
$y_t,\bar y_t\in M_1$ for $t\in[0,\alpha]$;
\item[(vi)] $D_2\tau$ and $D_3\tau$ are  locally Lipschitz continuous wrt all arguments, i.e.,  
	for every finite $\alpha\in(0,T]$, closed subset $M_1\subset \Omega_1$ of $C$ 
	which is also a bounded subset of $\Ww$, and closed and bounded subset $M_4\subset\Omega_4$
	 of $\Xi$ 
	there exists a constant $L_5=L_5(\alpha,M_1,M_4)$ such that
$$
|D_i\tau(t,\psi,\xi)-D_i\tau(\bar t,\bar \psi,\bar \xi)|_{\calL(Z_i,\,\setR)} \leq L_5\Bigl(|t-\bar t|+|\psi-\bar\psi|_C+|\xi-\bar\xi|_\Xi\Bigr)
$$
	for  $i=2,3$, $t,\bar t\in[0,\alpha]$, $\psi,\bar\psi\in M_1$, $\xi,\bar\xi\in M_4$, 
      where  $Z_2:=C$ and $Z_3:=\Xi$;
\item[(vii)] $D_2\tau$ and $D_3\tau$ are  continuously 
	differentiable wrt their second and third arguments on $[0,T]\times\Omega_1\times\Omega_4$;
\item[(viii)]  for every finite $\alpha\in(0,T]$, 
	for every closed subset $M_1\subset \Omega_1$ of $C$ which is also a bounded subset of $\Ww$,	 
	compact subset $M_2\subset \Omega_2$ of $\setR^n$,  
	and closed and bounded subsets 
	$M_3\subset \Omega_3$ of $\Theta$ and $M_4\subset\Omega_4$
	 of $\Xi$ there exists $L_6=L_6(\alpha,M_1,M_2,M_3,M_4)$ such that 
\begin{eqnarray*}
\balra{1cm}{\Bigl| \frac d{dt} f(t,y_t,y(t-\tau(t,y_t,\xi)), \theta)
-\frac d{dt}f(t,\bar y_t,\bar y(t-\tau(t,\bar y_t,\bar\xi)),\bar \theta)\Bigr|}\nl
&\leq& L_6\Bigl(|y_t-\bar y_t|_\Ww+|\xi-\bar \xi|_\Xi+|\theta-\bar \theta|_\Xi\Bigr),
\quad \mbox{a.e. } t\in[0,\alpha],
\end{eqnarray*}
where $\theta,\bar\theta\in M_3$, $\xi,\bar\xi\in M_4$, and  $y,\bar y\in\Ww([-r,\alpha],\setR^n)$ are such that
$y_t,\bar y_t\in M_1$ for $t\in[0,\alpha]$.
\end{itemize}
\end{itemize}

We introduce the parameter space
$$
\Gamma:=\Ww\times\Theta\times\Xi
$$
equipped with the product norm $|\gamma|_\Gamma:=|\phi|_\Ww+|\theta|_\Theta+|\xi|_\Xi$ for
$\gamma=(\phi,\theta,\xi)\in\Gamma$, and the set of admissible parameters
$$
\Pi:=\Bigl\{(\phi,\theta,\xi)\in\Gamma\st 
\phi\in\Omega_1,\ \phi(-\tau(0,\phi))\in\Omega_2,\ \theta\in\Omega_3,\ \xi\in\Omega_4\Bigr\}.
$$
The next theorem shows that every admissible parameter $(\hat\phi,\hat\theta,\hat\xi)\in\Pi$ has a neighborhood
$P$ and there exists a constant $\alpha>0$ such that  the \IVP{e1}{e2} has a unique solution on 
$[-r,\alpha]$ corresponding to all parameters $\gamma=(\phi,\theta,\xi)\in P$.
 This solution will be denoted by $x(t,\gamma)$,
and its  segment function at $t$ is denoted by $x_t(\cdot,\gamma)$.

The well-posedness of several classes of SD-DDEs was studied in many papers (see, e.g.,
\cite{D,HKWW,HT,Sl1,W2,W3}. The next result is a variant of a result from \cite{Ha4} where 
 the initial time  is also considered as a parameter, but the parameters $\theta$ and $\xi$ were missing in the equation. 
The proof is similar to that of Theorem 3.1 in \cite{Ha4}, (see also the analogous proof of Theorem~3.2
of the neutral case in \cite{Ha7}), therefore it is omitted here.
The notations and estimates
introduced in the next theorem will be essential in the following sections. 

\medskip
\begin{theorem}\label{t1}
Assume (A1) (i), (ii), (A2) (i), (ii), 
 and let $\hat\gamma\in\Pi$. 
Then there exist $\delta>0$ and  $0<\alpha\leq T$ 
finite numbers such that 
\begin{itemize}
\item[(i)]  for all $\gamma=(\phi,\theta,\xi)\in P:=\ballp{\Gamma}{\hat\gamma}{\delta}$ the
\IVP{e1}{e2}
has a unique solution $x(t,\gamma)$ on $[-r,\alpha]$;  
\item[(ii)] there exist a closed subset $M_1\subset C$ which is also a bounded and convex subset of $\Ww$, 
$M_2\subset\setR^n$  compact and convex subset 
and $M_3\subset\Theta$, $M_4\subset\Xi$  closed, bounded and convex subsets
of the respective spaces such that 
$ x_t(\cdot,\gamma)\in M_1$,  $x(t-\tau(t,x_t(\cdot,\gamma),\xi),\gamma)\in M_2$, $\theta\in M_3$
and $\xi\in M_4$ 
for $\gamma=(\phi,\theta,\xi)\in P$ and $t\in[0,\alpha]$; and
\item[(iii)]  $x_t(\cdot,\gamma)\in\Ww$ 
for $\gamma\in P$ and $t\in[0,\alpha]$, and 
there exist constants $N=N(\alpha,\delta)$ and $L=L(\alpha,\delta)$ such that
\begin{equation}\label{e37}
| x_t(\cdot,\gamma)|_\Ww\leq N,\qquad \gamma\in P,\ t\in[0,\alpha],
\end{equation}
and
\begin{equation}\label{e25}
|x_t(\cdot,\gamma)-x_t(\cdot,\bar\gamma)|_{\Ww}\leq 
	L|\gamma-\bar\gamma|_\Gamma,\qquad \gamma\in P,\ t\in[0,\alpha].
\end{equation}
\end{itemize}
\end{theorem}
\bigskip

The following result is obvious. 

\begin{remark}
Suppose the conditions of Theorem~\ref{t1} hold, 
 $P$ and $\alpha$ are defined by Theorem~\ref{t1}, and let $\calP$ denote 
the subset of $P$ consisting of those parameters which satisfy the compatibility
condition, i.e.,
\begin{equation}\label{e100}
\calP:=\Bigl\{(\phi,\theta,\xi)\in P\st \phi\in C^1,\quad
 \dot\phi(0-)=f(0,\phi,\phi(-\tau(0,\phi,\xi)),\theta)\Bigr\}.
\end{equation}
Then for
all parameter values $\gamma\in\calP$ the corresponding solution
 $x(t,\gamma)$ is continuously differentiable wrt $t$ for $t\in[-r,\alpha]$. 
\end{remark}
\bigskip

Throughout the rest of the paper we will use the following notations.
The parameter  $\hat\gamma\in\Pi$ is fixed, and   the constants $\delta>0$, 
$0<\alpha\leq T$ are defined by Theorem~\ref{t1},
and let $P:= \ballp{\Gamma}{\hat\gamma}{\delta}$. 
The sets $M_1\subset C$,
$M_2\subset\setR^n$, $M_3\subset\Theta$ and $M_4\subset\Xi$ are defined by 
Theorem~\ref{t1} (ii), $L_1=L_1(\alpha,M_1,M_2,M_3)$, 
$L_2=L_2(\alpha,M_1,M_4)$ and  $L_4=L_4(\alpha,M_1,M_4)$ 
denote the corresponding Lipschitz constants from (A1) (ii),  (A2) (ii) and (A2) (iv),
respectively, and the constants $N=N(\alpha,\delta)$ and 
$L=L(\alpha,\delta)$ are defined by Theorem~\ref{t1} (iii).  
We will restrict our attention to the fixed parameter set $P$,
so the sets $M_1, M_2, M_3$ and $M_4$, and the constants $L_1,L_2,L_4,L$ and $N$
can be considered to be fixed throughout this paper. 
\medskip

\begin{lemma}\label{l11}
 Assume (A1) (i), (ii), (A2) (i),(ii),
$\gamma=(\phi,\xi,\theta)\in P$, $h_k=(h^\phi_k,h^\xi_k,h^\theta_k)\in \Gamma$ is a sequence
such that $\gamma+h_k\in P$ for $k\in\setN$ and $|h_k|_\Gamma\to0$ as
$k\to\infty$. Let $x(t):=x(t,\gamma)$, $x^k(t):=x(t,\gamma+h_k)$ be the
corresponding solutions of the \IVP{e1}{e2}, and $u^k(s):=t-\tau(t,x^k_t,\xi+h^\xi_k)$
and $u(t):=t-\tau(t,x_t,\xi)$.
Then there exists $K_0\geq0$ such that
\begin{equation}\label{e88}
|u^k(t)-u(t)|\leq K_0|h_k|_\Gamma,\qquad t\in[0,\alpha],\quad k\in\setN.
\end{equation}
If, in addition, (A2) (iv) holds, 
then $u,u^k\in\Ww([0,\alpha],\setR)$, and  if (A2) (v) is also satisfied, then
there exists $K_1\geq0$ such that
\begin{equation}\label{e332}
|u^k-u|_{\Ww([0,\alpha],\setR)}\leq K_1|h_k|_\Gamma,
\qquad k\in\setN.
\end{equation}
\end{lemma}
\begin{proof}
Assumption (A2) (ii) implies
$$
|u^k(t)-u(t)|=|\tau(t,x^k_t,\xi+h^\xi_k)-\tau(t,x_t,\xi)|
\leq L_2(|x^k_t-x_t|_C+|h^\xi_k|_\Xi),\quad t\in[0,\alpha],
$$
so \eref{e25} yields \eref{e88}  with $K_0:=L_2(L+1)$.

Now assume (A2) (iv) also holds.
For simplicity of the notation let $h_0:=0=(0,0,0)\in\Gamma$, and so $x^0:=x$ and $u^0:=u$. 
Then (A2) (ii), the Mean Value Theorem and \eref{e37} imply for $k\in\setN_0$ and $t,\bar t\in[0,\alpha]$
\begin{equation}\label{e383}
| \tau(t,x^k_t,\xi+h^\xi_k)-\tau(\bar t,x^k_{\bar t},\xi+h^\xi_k)\Bigr|\leq L_2(|t-\bar t|+|x^k_t-x^k_{\bar t}|_C)
\leq L_2(1+N)|t-\bar t|.
\end{equation}
Hence $u^k$ is Lipschitz continuous, and so it is almost everywhere differentiable on
$[0,\alpha]$, and $|\dot u^k|_{\Lw([0,\alpha],\setR)}\leq L_2(1+N)$. Therefore $u^k\in\Ww([0,\alpha],\setR)$ for $k\in\setN_0$.

Let $L_4=L_4(\alpha,M_1,M_4)$ be defined by (A2) (v). 
Assumption (A2) (v) and \eref{e25} give 
$$
|\dot u^k(t)-\dot u(t)|=\Bigl|\frac d{dt}\tau(t,x^k_t,\xi+h^\xi_k)-\frac d{dt}\tau(t,x_t,\xi)\Bigr|
\leq L_4(|x^k_t-x_t|_C+ |h^\xi_k|_\Xi)\leq L_4(L+1)|h_k|_\Gamma
$$
for a.e.\ $t\in[0,\alpha]$.
Therefore \eref{e332} holds with $K_1:=\max\{K_0,L_4(L+1)\}$.
\end{proof}

\bigskip

We note that  (A2) (v) and (viii) hold under natural assumptions for example for  functions of the form
$$
\tau(t,\psi,\xi)=\bar\tau\Bigr(t,\psi(-\eta^1(t)),\ldots,\psi(-\eta^\ell(t)),\int_{-r}^0 A(t,\zeta)\psi(\zeta)\,d\zeta,\xi(t)\Bigr)
$$
and
$$
f(t,\psi,u,\theta)=\bar f\Bigl(t,\psi(-\nu^1(t)),\ldots,\psi(-\nu^m(t)),\int_{-r}^0 B(t,\zeta)\psi(\zeta)\,d\zeta,\theta(t)\Bigr).
$$
Here $\Theta=\Ww([0,T],\setR) $ and $\Xi=\Ww([0,T],\setR)$ can be used, and then we have, e.g.,
for $\tau$ under straightforward assumptions we have for a.e.\ $t\in[0,\alpha]$, $y\in\Ww([-r,\alpha],\setR^n)$
\begin{eqnarray*}
\frac {d}{dt}\tau(t,y_t,\xi)
&=& D_1\bar\tau\Bigr(t,y(t-\eta^1(t)),\ldots,y(t-\eta^\ell(t)),\int_{-r}^0 A(t,\zeta)y(t+\zeta)\,d\zeta,\xi(t)\Bigr)\nl
&& +\sum_{i=1}^\ell D_{i+1}\bar\tau\Bigr(t,y(t-\eta^1(t)),\ldots,y(t-\eta^\ell(t)),\int_{-r}^0 A(t,\zeta)y(t+\zeta)\,d\zeta,\xi(t)\Bigr)\nl
&&\qquad\times \dot y(t-\eta^i(t))(1-\dot\eta^i(t))\nl
&& + D_{i+2}\bar\tau\Bigr(t,y(t-\eta^1(t)),\ldots,y(t-\eta^\ell(t)),\int_{-r}^0 A(t,\zeta)y(t+\zeta)\,d\zeta,\xi(t)\Bigr)\nl
&&\qquad\times \int_{-r}^0 [D_1 A(t,\zeta)y(t+\zeta)+A(t,\zeta)\dot y(t+\zeta)]\,d\zeta\nl
&& + D_{i+3}\bar\tau\Bigr(t,y(t-\eta^1(t)),\ldots,y(t-\eta^\ell(t)),\int_{-r}^0 A(t,\zeta)y(t+\zeta)\,d\zeta,\xi(t)\Bigr)\dot \xi(t).
\end{eqnarray*}
Similar formula holds for $\frac {d}{dt}f(t,y_t,y(t-\tau(t,y_t,\xi)),\theta)$. So if
$\bar\tau$ and $\bar f$ are continuously differentiable, $\eta^i$ are continuously differentiable
and $\esssup_{t\in[0,T]}(1-\dot\eta^i(t))>0$ for $i=1,\ldots,\ell$, then it is easy to argue that
(A2) (v) and (viii) hold. 
\bigskip

\section{First-order differentiability  wrt the parameters}\label{sec_DDE_first}
\medskip
\setcounter{equation}0
\setcounter{theorem}0

In this section we study the differentiability of the solution $x(t,\gamma)$ of the \IVP{e1}{e2} 
wrt  $\gamma$. 
The proof of our differentiability results will be based on the following lemmas.
\medskip

\begin{lemma}\label{l5}
Let  $y\in\Ww([-r,\alpha],\setR^n)$, $\omega_k\in (0,\infty)$ ($k\in\setN$) 
be a sequence satisfying
$\omega_k\to0$ as $k\to\infty$. Let  $u,u^k\in{\cal PM}([0,\alpha],[-r,\alpha])$ ($k\in\setN$) be such that 
\begin{equation}\label{e11}
|u^k- u|_{\Ww([0,\alpha],\,\setR)}\leq \omega_k,\qquad k\in\setN.
\end{equation}
Then
\begin{equation}\label{e12}
\lim_{k\to\infty}\frac 1{\omega_k}\int_0^\alpha |y(u^k(s))-y(u(s))-\dot y(u(s))(u^k(s)-u(s))|\,ds=0.
\end{equation}
\end{lemma}
\begin{proof}
Let $0=t_0<t_1<\cdots<t_{m-1}<t_m=\alpha$ be the  mesh points of $u$ from the Definition~\ref{d1}, and 
let $0<\epsilon<\min\{t_{i+1}-t_i\st i=0,\ldots,m-1\}/2$ be fixed, and introduce  $t_i':=t_i+\epsilon$ for
$i=0,\ldots,m-1$, $t_i'':=t_i-\epsilon$ for $i=1,\ldots,m$, $t_0'':=0$, $t_m':=\alpha$, and 
let
$$
M:=\min_{i=0,\ldots,m-1}\essinf_{t\in[t_i',t_{i+1}'']} |\dot u(t)|.
$$
We have $M>0$, since $u\in{\cal PM}([0,\alpha],[-r,\alpha])$. Assumption \eref{e11} yields 
that there exists $k_0>0$ such that $|u^k-u|_{\Ww([0,\alpha],\,\setR)}<\frac M2$ for $k\geq k_0$.
Then for $k\geq k_0$ it follows $|\dot u^k(s)|\geq \frac M2$ and $|\dot u(s)+\nu(\dot u^k(s)-\dot u(s))|\geq \frac M2$
for a.e. $s\in[t_i',t_{i+1}'']$, $i=0,\ldots,m-1$ and $\nu\in[0,1]$.
Let $A:=|y|_\Ww([-r,\alpha],\setR^n)$.
Then simple manipulations, \eref{e11} and Fubini's theorem yield
\begin{eqnarray*}
\balra{1cm}{\int_0^\alpha |y(u^k(s))-y(u(s))-\dot y(u(s))(u^k(s)-u(s))|\,ds}\\
&\leq& \sum_{i=0}^m \int_{t_i''}^{t_i'} \Bigl(|y(u^k(s))-y(u(s))|+|\dot y(u(s))||u^k(s)-u(s)|\Bigr)\,ds\\
&&\quad+\sum_{i=0}^{m-1} \int_{t_i'}^{t_{i+1}''} \Bigl|\int_{u(s)}^{u^k(s)}\Bigl(\dot y(v)-\dot y(u(s))\Bigr)dv\Bigr|\,ds\\
&\leq& (m+1)2\epsilon 2A|u^k- u|_{C([0,\alpha],\setR)}\nl
&&\quad+\sum_{i=0}^{m-1} \int_{t_i'}^{t_{i+1}''} \Bigl|\int_{0}^{1}\Bigl[\dot y\Bigl(u(s)+\nu(u^k(s)-u(s))\Bigr)
-\dot y(u(s))\Bigr](u^k(s)-u(s))\,d\nu\Bigr|\,ds\\
&\leq& \omega_k\Bigl[(m+1)4A\epsilon
+ \sum_{i=0}^{m-1} \int_{0}^{1}\int_{t_i'}^{t_{i+1}''}  
\Bigl|\dot y\Bigl(u(s)+\nu(u^k(s)-u(s))\Bigr)-\dot y(u(s))\Bigr|\,ds\,d\nu\Bigr].
\end{eqnarray*}
It follows from  Lemma~\ref{l1} and Remark~\ref{r1} that for every $\nu\in[0,1]$
$$
\lim_{k\to\infty}\int_{t_i'}^{t_{i+1}''} 
\Bigl|\dot y\Bigl(u(s)+\nu(u^k(s)-u(s))\Bigr)-\dot y(u(s))\Bigr|\,ds=0,\qquad i=0,\ldots,m-1,
$$
hence we get  by using the Lebesgue's Dominated Convergence Theorem that
\begin{eqnarray*}
 \limsup_{k\to\infty}\frac 1{\omega_k}\int_0^\alpha |y(u^k(s))-y(u(s))-\dot y(u(s))(u^k(s)-u(s))|\,ds
\leq (m+1)4A\epsilon.
\end{eqnarray*}
This concludes the proof of \eref{e12}, since $\epsilon>0$ can be arbitrary close to 0.
\end{proof}
\bigskip

We introduce the notations
\begin{eqnarray}
\omega_f(t,\bar\psi,\bar u,\bar\theta,\psi,u,\theta)
&:=&f(t,\psi,u,\theta)-f(t,\bar\psi,\bar u,\bar\theta)-D_2f(t,\bar\psi,\bar u,\bar\theta)(\psi-\bar\psi)\nl
&&-D_3f(t,\bar\psi,\bar u,\bar\theta)(u-\bar u)
-D_4f(t,\bar\psi,\bar u,\bar\theta)(\theta-\bar \theta),\qquad\label{e17}\\
\omega_\tau(t,\bar\psi,\bar \xi,\psi,\xi)
&:=&\tau(t,\psi,\xi)-\tau(t,\bar\psi,\bar\xi)-D_2\tau(t,\bar\psi,\bar\xi)(\psi-\bar\psi)\nl
&&-D_3\tau(t,\bar\psi,\bar\xi)(\xi-\bar\xi)\label{e22}
 \end{eqnarray}
for $ t\in[0,T]$, $\bar \psi,\psi\in\Omega_1$, $\bar u,u\in\Omega_2$, $\bar\theta,\theta\in\Omega_3$,
$\bar\xi,\xi\in\Omega_4$, and
\begin{eqnarray}
\Omega_f(\epsilon)%
&:=&\max_{i=2,3,4}\sup\Bigl\{|D_i f(t,\psi,u,\theta)-
D_if(t,\tilde\psi,\tilde u,\tilde\theta)|_{\calL(Y_i,\,\setR^n)}\st\nl
&&\qquad\qquad
|\psi-\tilde \psi|_C+|u-\tilde u| +|\theta-\tilde\theta|_\Theta\leq 
   \epsilon,\  \ t\in[0,\alpha],\ \psi,\tilde \psi\in M_1,\nl
&&\qquad\qquad u,\tilde u \in M_2,\ \theta,\tilde \theta \in M_3\Bigr\},
\quad\label{e29a}\\
\Omega_\tau(\epsilon)\!\!
&:=&\max_{i=2,3}\sup \Bigl\{|D_i\tau(t,\psi,\xi)-D_i\tau(t,\bar\psi,\bar\xi)|_{\calL(Z_i,\,\setR)}\st 
|\psi-\bar\psi|_C+|\xi-\bar\xi|_\Xi\leq\epsilon,\nl
&&\qquad\qquad \ t\in[0,\alpha],\ 
\psi,\bar\psi\in M_1,\ \xi,\bar\xi\in M_4\Bigr\},\qquad \label{e29}
\end{eqnarray}
where $Y_2:=C$, $Y_3:=\setR^n$, $Y_4:=\Theta$, $Z_2:=C$ and $Z_3:=\Xi$.
\bigskip

The following result is an easy generalization of Lemma 4.2 of \cite{Ha4}
for the \IVP{e1}{e2}, therefore we omit its proof here. (See also
the related proof of Lemma~\ref{l17} below.)
\medskip

\begin{lemma}[see \cite{Ha4}]\label{l6}
Suppose (A1) (i)--(iii), (A2) (i)--(iii). Let $P$ and $\alpha>0$ be defined by Theorem~\ref{t1},
let $\gamma=(\phi,\theta,\xi)\in P$ 
be fixed, and   $h_k=(h^\phi_k,h^\theta_k,h^\xi_k)\in \Gamma$ ($k\in\setN$) be a sequence satisfying
$|h_k|_{\Gamma} \to 0$ as $k\to\infty$, and $\gamma+h_k\in P$ for $k\in\setN$. 
Let $x(t):=x(t,\gamma)$,  $x^k(t):=x(t,\gamma+h_k)$,
$u(t):=t-\tau(t,x_t,\xi)$ and $u^k(t):=t-\tau(t,x^k,\xi+h^\xi_k)$.
Then
\begin{equation}\label{e18}
\lim_{k\to\infty}\frac 1{|h_k|_{\Gamma}}\int_0^\alpha 
|\omega_f(s,x_s,x(u(s)),\theta,x^k_s,x^k(u^k(s)),\theta+h^\theta_k)|\,ds=0
\end{equation}
and
\begin{equation}\label{e21}
\lim_{k\to\infty}\frac 1{|h_k|_{\Gamma}}\int_0^\alpha 
|\omega_\tau(s,x_s,\xi,x^k_s,\xi+h^\xi_k)|\,ds=0.
\end{equation}
\end{lemma}
\bigskip

A solution $x(\cdot,\gamma)$  of the \IVP{e1}{e2} for $\gamma\in P$ is, 
 in general, only a $\Ww$-function on the interval $[-r,0]$, but it is 
continuously differentiable for $t\geq 0$. In \cite{HT} (see also \cite{Ha4})
a parameter set 
$$
P_1:=\{ \gamma=(\phi,\theta,\xi)\in P\st x(\cdot,\gamma)\in X(\alpha,\xi)\}
$$
 was considered, where
\begin{eqnarray*}
X(\alpha,\xi)
&:=&\Bigl\{x\in \Ww([-r,\alpha],\setR^n)\st
x_t\in\Omega_1,\ x(t-\tau(t,x_t,\xi))\in\Omega_2\ \mbox{for } t\in[0,\alpha],\ \nl
&&\quad  \mbox{and}\ \ \essinf\Bigl\{ \dd{t} (t-\tau(t,x_t,\xi))\st\  \mbox{a.e. }
t\in[0,\alpha^*]\Bigr\}>0\Bigr\}\qquad\
\label{e24}
\end{eqnarray*}
and $\alpha^*:=\min\{r,\alpha\}$.
Then Lemma~\ref{l1} yields that the function $t\mapsto \dot x(t-\tau(t,x_t,\xi))$ is 
well-defined for a.e. $t\in[0,\alpha^*]$ and it is integrable on $[0,\alpha^*]$, 
and it is well-defined and continuous on
$[\alpha^*,\alpha]$. Note that it was shown in \cite{HT} (see also \cite{Ha4}) that $P_1$ 
is an open subset of the parameter set $P$.
In this section we relax this condition. We define the parameter set
\begin{eqnarray}
P_2&:=&\{ \gamma=(\phi,\theta,\xi)\in P\st \mbox{the map }\ [0,\alpha^*]\to\setR,\ t\mapsto t-\tau(t,x_t(\cdot,\gamma),\xi)\ 
\nl
&&\qquad\qquad \mbox{belongs to } {\cal PM}([0,\alpha^*],[-r,\alpha^*])\}.\label{e317}
\end{eqnarray}
Then we have $P_1\subset P_2\subset P$, and Lemma~\ref{l3} yields that for a solution $x$ corresponding to
parameter $\gamma\in P_2$
the function $t\mapsto \dot x(t-\tau(t,x_t,\xi))$ is 
well-defined for a.e. $t\in[0,\alpha^*]$ and it is integrable on $[0,\alpha^*]$.
Therefore, as the next discussion will show, the parameter set where the variational equation
is defined,
and correspondingly the differentiability of the solution wrt the parameters can be obtained is
larger than in the previous papers \cite{Ha2, Ha4, HT}.

Let $\gamma=(\phi,\theta,\xi)\in P_2$ be fixed, and let $x(t):=x(t,\gamma)$.
Consider the space $C\times\Theta\times\Xi$
equipped with the product norm
$|( h^\phi,h^\theta,h^\xi)|_{C\times\Theta\times\Xi}:=|h^\phi|_C+|h^\theta|_\Theta+|h^\xi|_\Xi$.
Then for a.e.\ $t\in[0,\alpha]$ we introduce the linear operator 
$L(t,x)\st C\times\Theta\times\Xi\to\setR^n$ by
\begin{eqnarray}
\balra{0.5cm}{L(t,x)(h^\phi,h^\theta,h^\xi)}\nl
&:=&D_2f(t,x_t,x(t-\tau(t,x_t,\xi)),\theta)h^\phi 
+ D_3f(t,x_t,x(t-\tau(t,x_t,\xi)),\theta)\nl
&&\quad\times\Bigl[-\dot x(t-\tau(t,x_t,\xi))\Bigl(D_2\tau(t,x_t,\xi)h^\phi+D_3\tau(t,x_t,\xi)h^\xi\Bigr)
+h^\phi(-\tau(t,x_t,\xi))\Bigr]\nl
&&+D_4f(t,x_t,x(t-\tau(t,x_t,\xi)),\theta)h^\theta
\qquad\ \  \label{e26}
\end{eqnarray}
for $(h^\phi,h^\theta,h^\xi)\in C\times\Theta\times\Xi$. 
We have by (A1) (ii), (A2) (ii) and \eref{e37}
\begin{eqnarray}
|L(t,x)(h^\phi,h^\theta,h^\xi)|
&\leq&L_1|h^\phi|_C 
+ L_1\Bigl[N(L_2|h^\phi|_C+L_2|h^\xi|_\Xi)+|h^\phi|_C\Bigr]+L_1|h^\theta|_\Theta\nl
&\leq& L_1N_0|(h^\phi,h^\theta,h^\xi)|_{C\times\Theta\times\Xi},\qquad \mbox{a.e. } t\in[0,\alpha],\label{e368}
\end{eqnarray}
where
\begin{equation}\label{e47}
N_0:=NL_2+3.
\end{equation}
Therefore
\begin{equation*}
|L(t,x)|_{\calL(C\times\Theta\times\Xi,\setR^n)}
\leq  L_1N_0,\qquad \mbox{a.e. } t\in[0,\alpha].\qquad\ \  
\end{equation*}
Hence $L(t,x)$ is a bounded linear operator for all $t$ for which $\dot x(t-\tau(t,x_t,\xi))$ exists,
i.e., for a.e.\ $t\in[0,\alpha]$.

For $\gamma\in P_2$ we define the variational equation associated to  $x=x(\cdot,\gamma)$ 
as
\begin{eqnarray}
\dot z(t)&=&L(t,x)(z_t,h^\theta,h^\xi) \qquad \mbox{a.e.}\ t\in [0,\alpha],\label{e3} \\
z(t)\!&=& h^\phi(t),\qquad t\in[-r,0],\label{e4}
\end{eqnarray}
where $h=(h^\phi,h^\theta,h^\xi)\in C\times\Theta\times\Xi$ is fixed.
The \IVP{e3}{e4} is a Carath\'eodory type linear delay equation. By its solution we mean a continuous function
$z\st[-r,\alpha]\to\setR^n$, which is absolutely continuous on $[0,\alpha]$, and it satisfies
\eref{e3} for a.e. $t\in[0,\alpha]$ and \eref{e4} for all $t\in[-r,0]$.
Standard argument (\cite{CL}, \cite{HLu}) shows
that the \IVP{e3}{e4} has a unique solution $z(t)=z(t,\gamma, h)$ for 
$t\in[-r,\alpha]$, $\gamma\in P_2$ and $ h=(h^\phi,h^\theta,h^\xi)\in C\times\Theta\times\Xi$.

The following  result was proved in \cite{Ha4} for the parameter set $P_1$ (see Lemma 4.4 in
\cite{Ha4}),
but the proof is identical for the parameter set $P_2$, as well. \medskip

\begin{lemma}[see \cite{Ha4}]\label{l7}
Assume (A1) (i)--(iii), (A2) (i)--(iii). Let 
 $\gamma\in P_2$, and 
$x(t):=x(t,\gamma)$  for $t\in[-r,\alpha]$. 
Let $ h\in C\times\Theta\times\Xi$ and let $z(t,\gamma, h)$ be the corresponding solution of the
\IVP{e3}{e4} on $[-r,\alpha]$. Then 
\begin{itemize}
\item[(i)] $z(t,\gamma,\cdot)\in\calL(C\times\Theta\times\Xi,\setR^n)$, 
the map $C\times\Theta\times\Xi\to C$, $ h\mapsto z_t(\cdot,\gamma, h)$ is in $\calL(C\times\Theta\times\Xi,C)$, and 
\begin{equation}\label{e45}
|z(t,\gamma,h)|\leq|z_t(\cdot,\gamma,h)|_C\leq N_1|h|_{C\times\Theta\times\Xi},
\qquad t\in[0,\alpha],\ \gamma\in P_2,\ h\in C\times\Theta\times\Xi,
\end{equation}
where $N_1:=e^{L_1N_0\alpha}$;
\item[(ii)]
there exists $N_2\geq 0$ such that
\begin{equation}\label{e94}
|z_t(\cdot,\gamma,h)|_\Ww\leq N_2|h|_\Gamma,
\qquad t\in[0,\alpha],\ \gamma\in P_2,\ h\in\Gamma.
\end{equation}
\end{itemize}
\end{lemma}
\bigskip

Next we show that the linear operators $z(t,\gamma,\cdot)$
and $z_t(\cdot,\gamma,\cdot)$ are continuous in $t$  and $\gamma$,
assuming that $\gamma$ belongs to $P_2$. First we need
the following result.
\medskip

\begin{lemma}\label{l33}
Assume (A1) (i)--(iii), (A2) (i)--(iii). 
Let $\gamma\in P_2$, $h=(h^\phi,h^\theta,h^\xi)\in\Gamma$,  
$h_k=(h^\phi_k,h^\theta_k,h^\xi_k)\in \Gamma$ ($k\in\setN$) be a sequence such that
$|h_k|_\Gamma\to0$ as $k\to\infty$, and $\gamma+h_k\in P_2$ for $k\in\setN$.
Let $x(s):=x(s,\gamma)$, $x^k(s):=x(s,\gamma+h_k)$, $u(s):=s-\tau(s,x_s,\xi)$,
and $u^k(s):=s-\tau(s,x^k_s,\xi+h^\xi_k)$. 
Then there exists a nonnegative sequence $c_{0,k}$ such that 
$c_{0,k}\to0$ as $k\to\infty$, and
\begin{equation}\label{e395}
|L(s,x^k)h-L(s,x)h|
\leq c_{0,k}|h|_\Gamma+L_1L_2|\dot x(u^k(s))-\dot x(u(s))||h|_\Gamma 
\end{equation}
for a.e.\ $s\in[0,\alpha]$, $k\in\setN$ and $h\in\Gamma$.
\end{lemma}
\begin{proof}
We have
 \begin{eqnarray*}
\balra{1cm}{ L(s,x^k)(h^\phi,h^\theta,h^\xi)-L(s,x)(h^\phi,h^\theta,h^\xi)}\nl
&=& \Bigl(D_2f(s,x^k_s,x^k(u^k(s)),\theta+h^\theta_k)-D_2f(s,x_s,x(u(s)),\theta)\Bigr)h^\phi\nl
&&+\Bigl(D_3f(s,x^k_s,x^k(u^k(s)),\theta+h^\theta_k)-D_3f(s,x_s,x(u(s)),\theta)\Bigr)\nl
&&\quad\times\Bigl(-\dot x^k(u^k(s))\Bigr)\Bigl(D_2\tau(s,x^k_s,\xi+h^\xi_k)h^\phi+D_3\tau(s,x^k_s,\xi+h^\xi_k)h^\xi\Bigr)\nl
&&+D_3f(s,x_s,x(u(s)),\theta)\Bigl(-\dot x^k(u^k(s))+\dot x(u^k(s)))\Bigr)\nl
&&\quad\times\Bigl(D_2\tau(s,x^k_s,\xi+h^\xi_k)h^\phi+D_3\tau(s,x^k_s,\xi+h^\xi_k)h^\xi\Bigr)\nl
&&+D_3f(s,x_s,x(u(s)),\theta)\Bigl(-\dot x(u^k(s))+\dot x(u(s)))\Bigr)\nl
&&\quad\times\Bigl(D_2\tau(s,x^k_s,\xi+h^\xi_k)h^\phi+D_3\tau(s,x^k_s,\xi+h^\xi_k)h^\xi\Bigr)\nl
&&+D_3f(s,x_s,x(u(s)),\theta)\Bigl(-\dot x(u(s))\Bigr)\nl
&&\quad\times\Bigl[\Bigr(D_2\tau(s,x^k_s,\xi+h^\xi_k)-D_2\tau(s,x_s,\xi)\Bigr)h^\phi\nl
&&\qquad\quad +\Bigl(D_3\tau(s,x^k_s,\xi+h^\xi_k)-D_3\tau(s,x_s,\xi)\Bigr)h^\xi\Bigr]\nl
&&+\Bigl(D_3f(s,x^k_s,x^k(u^k(s)),\theta+h^\theta_k)-D_3f(s,x_s,x(u(s)),\theta)\Bigr)h^\phi(-\tau(s,x^k_s,\xi+h^\xi_k))\nl
&&+D_3f(s,x_s,x(u(s)),\theta)\Bigl(h^\phi(-\tau(s,x^k_s,\xi+h^\xi_k))-h^\phi(-\tau(s,x_s,\xi))\Bigr)\nl
&&+\Bigl(D_4f(s,x^k_s,x^k(u^k(s)),\theta+h^\theta_k)-D_4f(s,x_s,x(u(s)),\theta)\Bigr)h^\theta,\qquad s\in[0,\alpha].
\end{eqnarray*}
Relations \eref{e37}, \eref{e25},  \eref{e88} and the Mean Value Theorem give
\begin{eqnarray}\label{e345}
|x^k(u^k(s))-x(u(s))|
&\leq& |x^k(u^k(s))-x(u^k(s))|+|x(u^k(s))-x(u(s))|\nl
&\leq& L|h_k|_\Gamma+N| u^k(s)-u(s)|\nl
&\leq& K_2|h_k|_\Gamma,\
\end{eqnarray}
with $K_2:=L+ NK_0$, 
\begin{equation}\label{e355}
|x^k_s-x_s|_C+|x^k(u^k(s))-x(u(s))|+|h^\theta_k|_\Theta
\leq K_3|h_k|_\Gamma,
\end{equation}
with $K_3:=L+K_2+1$, and
\begin{equation}\label{e396}
|x^k_s-x_s|_C+|h^\xi_k|_\Xi\leq (L+1)|h_k|_\Gamma.
\end{equation}
Combining the above estimates with (A1) (ii),  (A2) (ii), \eref{e37}, \eref{e25}, \eref{e88} 
 and the definition of $\Omega_f$ and $\Omega_\tau$ we get
\begin{eqnarray*}
\balra{1cm}{ |L(s,x^k)(h^\phi,h^\theta,h^\xi)-L(s,x)(h^\phi,h^\theta,h^\xi)|}\nl
&\leq& \Omega_f\Bigl(K_3|h_k|_\Gamma\Bigr)|h^\phi|_C
+\Omega_f\Bigl(K_3|h_k|_\Gamma\Bigr)NL_2(|h^\phi|_C+|h^\xi|_\Xi)\nl
&&+L_1L|h_k|_\Gamma L_2(|h^\phi|_C+|h^\xi|_\Xi)
+L_1\Bigl|\dot x(u^k(s))-\dot x(u(s))\Bigr| L_2(|h^\phi|_C+|h^\xi|_\Xi)\nl
&&+L_1N\Omega_\tau\Bigl((L+1)|h_k|_\Gamma)\Bigr)(|h^\phi|_C+|h^\xi|_\Xi)
+\Omega_f\Bigl(K_3|h_k|_\Gamma\Bigr)|h^\phi|_C\nl
&&+L_1|\dot h^\phi|_\Lw K_0|h_k|_\Gamma
+\Omega_f\Bigl(K_3|h_k|_\Gamma\Bigr)|h^\theta|_\Theta,
\qquad s\in[0,\alpha],\qquad\quad 
\end{eqnarray*}
which yields \eref{e395} with $c_{0,k}:=N_0\Omega_f\Bigl(K_3|h_k|_\Gamma\Bigr)
+L_1L_2L|h_k|_\Gamma+L_1N\Omega_\tau\Bigl((L+1)|h_k|_\Gamma\Bigr)+L_1K_0|h_k|_\Gamma$,
where $N_0$ is defined by \eref{e47}.
\end{proof}
\medskip

\begin{lemma}\label{l8}
Assume (A1) (i)--(iii), (A2) (i)--(v). Let  
 $\gamma\in P_2$, and $x(t):=x(t,\gamma)$  for $t\in[-r,\alpha]$. 
Let $ h\in C\times\Omega\times\Xi$ and let $z(t,\gamma, h)$ be the corresponding solution of the
\IVP{e3}{e4} on $[-r,\alpha]$. Then 
 the maps 
$$
\setR\times\Gamma\supset [0,\alpha]\times P_2 \to \calL(\Gamma,\setR^n),
\quad (t,\gamma)\mapsto z(t,\gamma,\cdot)
$$ 
and
$$
\setR\times\Gamma\supset [0,\alpha]\times P_2 \to \calL(\Gamma,C),
\quad (t,\gamma)\mapsto z_t(\cdot,\gamma,\cdot)
$$
are continuous.
\end{lemma}
\begin{proof}
Let $\gamma\in P_2$ be fixed, 
and let
$h_k=(h^\phi_k,h^\theta_k,h^\xi_k)\in \Gamma$ ($k\in\setN$) be a sequence such that
$|h_k|_\Gamma\to0$ as $k\to\infty$ and $\gamma+h_k\in P_2$ for $k\in\setN$.
For a fixed  $ h=(h^\phi,h^\theta,h^\xi)\in\Gamma$ we define the short notations
 $x^k(t):=x(t,\gamma+h_k)$, $x(t):=x(t,\gamma)$,
$u^k(t):=t-\tau(t,x^k_t,\xi+h^\xi_k)$, $u(t):=t-\tau(t,x_t,\xi)$,
$z^{k,h}(t):=z(t,\gamma+h_k, h)$ and $z^h(t):=z(t,\gamma, h)$. 
The functions $z^{k,h}$ and $z^{h}$ satisfy
\begin{eqnarray*}
z^{k,h}(t)&=& h^\phi(0)+\int_{0}^t L(s,x^k)(z^{k,h}_s,h^\theta,h^\xi)\,ds,\qquad t\in[0,\alpha],\\
z^{h}(t) &=& h^\phi(0)+\int_{0}^t L(s,x)(z^{h}_s,h^\theta,h^\xi)\,ds,\qquad t\in[0,\alpha],
\end{eqnarray*}
and therefore  
for $t\in[0,\alpha]$
\begin{equation}\label{e320}
|z^{k,h}(t)-z^{h}(t)|\leq \int_{0}^t \Bigl|\Bigl(L(s,x^k)-L(s,x)\Bigr)(z^{h}_s,h^\theta,h^\xi)
+L(s,x^k)(z^{k,h}_s-z^{h}_s,0,0)\Bigr|\,ds.
\end{equation}
We have by \eref{e94} and $N_2\geq1$
\begin{equation}\label{e397}
|(z^h_s,h^\theta,h^\xi)|_\Gamma \leq N_2 |h|_\Gamma+|h^\theta|_\Theta+|h^\xi|_\Xi
\leq (N_2+1) |h|_\Gamma.
\end{equation}
Then   \eref{e368},   \eref{e395},  \eref{e320} and \eref{e397} imply
\begin{equation}\label{e52}
|z^{k,h}(t)-z^{h}(t)|\leq c_{1,k}| h|_\Gamma+\int_{0}^t L_1N_0|z^{k,h}_s-z^{h}_s|_C\,ds,
\qquad t\in[0,\alpha],
\end{equation}
where $c_{1,k}$ is defined by
\begin{eqnarray*}
 c_{1,k} &:=& \alpha c_{0,k}(N_2+1)+ L_1L_2(N_2+1)\int_{0}^\alpha |\dot x(u^k(s))-\dot {x}(u(s))|\,ds.
\end{eqnarray*}
Relation \ref{e332} and Lemma~\ref{l4}  yield that 
\begin{equation}\label{e400}
\lim_{k\to\infty} \int_{\alpha}^\alpha |\dot x(u^k(s))-\dot {x}(u(s))|\,ds=0.
\end{equation}
Hence  $c_{1,k}\to0$ as $k\to\infty$. 

Lemma~\ref{l0} is applicable for \eref{e52} since $|z^{k,h}_{0}-z^{h}_{0}|_C=0$,  and it gives
\begin{equation}\label{e344}
|z^{k,h}(t)-z^{h}(t)|\leq |z^{k,h}_t-z^{h}_t|_C\leq c_{1,k}N_1| h|_\Gamma,\qquad  t\in[0,\alpha],
\end{equation}
where $N_1:=e^{L_1N_0\alpha}$.
Therefore we get for $t\in[0,\alpha]$
\begin{equation}\label{e53}
 |z(t,\gamma+h_k,\cdot)-z(t,\gamma,\cdot)|_{\calL(\Ww,\setR^n)}
\leq |z_{t}(\cdot,\gamma+h_k,\cdot)-z_{t}(\cdot,\gamma,\cdot)|_{\calL(\Ww,C)}
\leq c_{1,k}N_1
\end{equation}
for all $k\in\setN$. 

Let $t\in[0,\alpha]$ be fixed, and let $\nu_k$ be a sequence of real numbers such that
$t+\nu_k\in[0,\alpha]$ for $k\in\setN$ and $\nu_k\to 0$ as $k\to\infty$. 
  Then \eref{e94} and the Mean Value Theorem yield
$$
|z_{t+\nu_k}(\cdot,\gamma+h_k,\cdot)-z_{t}(\cdot,\gamma+h_k,\cdot)|_{\calL(\Gamma,C)}
\leq N_2|\nu_k|, \qquad k\geq k_0.
$$
Combining this relation  with \eref{e53} and $c_{1,k}\to0$ we get
\begin{eqnarray*}
\balra{0.5cm}{|z(t+\nu_k,\gamma+h_k,\cdot)-z(t,\gamma,\cdot)|_{\calL(\Gamma,\setR^n)}}\nl
&\leq& |z_{t+\nu_k}(\cdot,\gamma+h_k,\cdot)-z_{t}(\cdot,\gamma,\cdot)|_{\calL(\Gamma,C)}\nl
&\leq& |z_{t+\nu_k}(\cdot,\gamma+h_k,\cdot)-z_t(\cdot,\gamma+h_k,\cdot)|_{\calL(\Gamma,C)}
+|z_t(\cdot,\gamma+h_k,\cdot)-z_t(\cdot,\gamma,\cdot)|_{\calL(\Gamma,C)}\nl
&\leq& N_2|\nu_k|+c_{1,k}N_1\nl
&\to& 0,\qquad \mbox{as } k\to\infty.
\end{eqnarray*}
This completes the proof.
\end{proof}
\bigskip

\begin{remark}\label{r66}\rm
 Note that if in the statement of Lemma~\ref{l8} the parameter set $P_2$ is replaced by the smaller set 
$P_1$, then assumptions (A2) (iv) and (v) are not needed to prove the statement, since
in this case \eref{e88} and Lemma~\ref{l1} can be used to show that $c_{1,k}\to0$ as
$k\to\infty$.
\end{remark}
\medskip

Now we are ready to prove the Fr\'echet-differentiability of the function
 $x(t,\gamma)$  wrt $\gamma$. We will denote this derivative by 
$D_2 x(t,\gamma)$.
\medskip

\begin{theorem}\label{t2}
Assume (A1) (i)--(iii), (A2) (i)--(v), and let $P_2$ be defined by \eref{e317}.
Then the functions
$$
\setR\times \Gamma\supset [0,\alpha]\times P\to\setR^n,\qquad 
(t,\gamma)\mapsto x(t,\gamma)
$$
and
$$
\setR\times \Gamma\supset [0,\alpha]\times P \to C,\qquad 
(t,\gamma)\mapsto x_t(\cdot,\gamma)
$$
 are both differentiable wrt $\gamma$ for every $\gamma\in P_2$, and
\begin{equation}\label{e33}
D_2 x(t,\gamma) h=z(t,\gamma, h),\qquad  h\in\Gamma,\ 
t\in[0,\alpha],\ \gamma\in P_2,
\end{equation}
and
\begin{equation}\label{e60}
D_2 x_t(\cdot,\gamma) h=z_t(\cdot,\gamma, h),\qquad  h\in\Gamma,\ 
t\in[0,\alpha],\ \gamma\in P_2,
\end{equation}
where $z(t,\gamma, h)$ is the solution of the \IVP{e3}{e4} for 
$t\in[0,\alpha]$,\ $\gamma\in P_2$ and $ h\in\Gamma$.
Moreover, the functions 
$$
\setR\times\Gamma\supset [0,\alpha]\times P_2\to\calL(\Gamma,\setR^n),\qquad (t,\gamma)\mapsto D_2 x(t,\gamma) 
$$
and
$$
\setR\times\Gamma\supset [0,\alpha]\times P_2\to \calL(\Gamma,C),\qquad (t,\gamma)\mapsto D_2 x_t(\cdot,\gamma) 
$$
are continuous.
\end{theorem}
\begin{proof}
Let $\gamma=(\phi,\theta,\xi)\in P_2$ be fixed, and
let $h_k=(h^\phi_k,h^\theta_k,h^\xi_k)\in\Gamma$ ($k\in\setN$) be a sequence with $|h_k|_\Gamma\to0$ as $k\to\infty$
and $\gamma+h_k\in P$ for $k\in\setN$. To simplify notation, 
let $x^k(t):=x(t,\gamma+h_k)$, $x(t):=x(t,\gamma)$, $u(s):=s-\tau(s,x_s,\xi)$, 
$u^k(s):=s-\tau(s,x^k_s,\xi+h^\xi_k)$ and $z^{h_k}(t):=z(t,\gamma,h_k)$.
Then
\begin{eqnarray*}
x^k(t)&=&\phi(0)+h^\phi_k(0)+\int_0^t f(s,x^k_s,x^k(u^k(s)),\theta+h^\theta_k)\,ds,\qquad t\in[0,\alpha],\\
x(t)&=&\phi(0)+\int_0^t f(s,x_s,x(u(s)),\theta)\,ds,\qquad t\in[0,\alpha],
\end{eqnarray*}
and
$$
z^{h_k}(t)=h^\phi_k(0)+\int_0^t L(s,x)(z^{h_k}_s,h^\theta_k,h^\xi_k)\,ds,\qquad t\in[0,\alpha].
$$
We have 
\begin{eqnarray}\label{e101}
x^k(t)-x(t)-z^{h_k}(t)
&=&\int_0^t \Bigl(f(s,x^k_s,x^k(u^k(s)),\theta+h^\theta_k)-
f(s,x_s,x(u(s)),\theta)\nl
&&\qquad -\ L(s,x)(z^{h_k}_s,h^\theta_k,h^\xi_k)\Bigr)ds.
 \end{eqnarray}
The definitions of $\omega_f$ and $L(s,x)$ (see \eref{e17} and \eref{e26}, 
respectively)  yield for $s\in[0,\alpha]$
\begin{eqnarray}
\balra{0.cm}{f(s,x^k_s,x^k(u^k(s)),\theta+h^\theta_k)-f(s,x_s,x(u(s)),\theta)-L(s,x)(z^{h_k}_s,h^\theta_k,h^\xi_k)}\nl
&=&D_2f(s,x_s,x(u(s)),\theta)(x^k_s-x_s-z^{h_k}_s)
+ D_3f(s,x_s,x(u(s)),\theta)\Bigl(x^k(u^k(s))-x(u(s))\Bigr)\nl
&& +\ D_3f(s,x_s,x(u(s)),\theta)\Bigl(\dot x(u(s))\Bigl(D_2\tau(s,x_s,\xi)z^{h_k}_s+D_3\tau(s,x_s,\xi)h^\xi_k\Bigr)-z^{h_k}(u(s))\Bigr)\nl
&&+\ \omega_f(s,x_s,x(u(s),\theta,x^k_s,x^k(u^k(s)),\theta+h^\theta_k).
\label{e55}
\end{eqnarray}
Relation \eref{e22} and simple manipulations give
\begin{eqnarray}
 \balra{0cm}{x^k(u^k(s))-x(u(s))+\dot x(u(s))\Bigl(D_2\tau(s,x_s,\xi)z^{h_k}_s+D_3\tau(s,x_s,\xi)h^\xi_k\Bigr)-z^{h_k}(u(s))}\nl
&=& x^k(u^k(s))-x(u^k(s))-z^{h_k}(u^k(s))+x(u^k(s))-x(u(s))-\dot x(u(s))(u^k(s)-u(s))\nl
&& -\dot x(u(s))\omega_\tau(s,x_s,\xi,x^k_s,\xi+h^\xi_k)-\dot x(u(s))D_2\tau(s,x_s,\xi)(x^k_s-x_s-z^{h_k}_s)\nl
&& +z^{h_k}(u^k(s))-z^{h_k}(u(s)).\label{e55b}
\end{eqnarray}
Relation \eref{e88} and \eref{e94} imply
\begin{equation}\label{e89}
|z^{h_k}(u^k(s))-z^{h_k}(u(s))|\leq N_2| h_k|_\Gamma|u^k(s)-u(s)|
\leq N_2K_0|h_k|^2_\Gamma.
\end{equation}
Using \eref{e37}, (A1) (ii), (A2) (ii),  and combining  \eref{e101}, \eref{e55}, \eref{e55b} and \eref{e89}
we get 
\begin{eqnarray}
\balra{0.5cm}{|x^k(t)-x(t)-z^{h_k}(t)|}\nl
&\leq&\int_0^t \Bigl[L_1\Bigl(|x^k_s-x_s-z^{h_k}_s|_C+|x^k(u^k(s))-x(u^k(s))-z^{h_k}(u^k(s))|\nl
&&\qquad +\ |x(u^k(s))-x(u(s))-\dot x(u(s))(u^k(s)-u(s))|\nl
 &&\qquad +\ N|\omega_\tau(s,x_s,\xi,x^k_s,\xi+h^\xi_k)|
  + NL_2|x^k_s-x_s-z^{h_k}_s|_C + N_2K_0|h_k|_\Gamma^2\Bigr)\nl
&&\qquad+\ |\omega_f(s,x_s,x(u(s)),\theta,x^k_s,x^k(u^k(s)),\theta+h^\theta_k)|\Bigr]ds,\qquad t\in[0,\alpha].\label{e201}
\end{eqnarray}
Let $N_0$  be defined by \eref{e47}. Then
\begin{equation}\label{e32}
|x^k(t)-x(t)-z^{h_k}(t)| \leq a_k+b_k+c_k+d_k+L_1N_0\int_0^t |x^k_s-x_s-z^{h_k}_s|_C\,ds,\qquad 
t\in[0,\alpha],
\end{equation}
where
\begin{eqnarray}
a_k&:=&\int_0^\alpha|\omega_f(s,x_s,x(u(s)),\theta,x^k_s,x^k(u^k(s)),\theta+h^\theta_k)|\,ds,\label{e57}\\
b_k&:=&L_1N\int_0^\alpha|\omega_\tau(s,x_s,\xi,s,x^k_s,\xi+h^\xi_k)|\,ds,\label{e57b}\\
c_k&:=&L_1\int_0^\alpha\! |x(u^k(s))-x(u(s))-\dot x(u(s))(u^k(s)-u(s))|\,ds,\label{e58}
\end{eqnarray}
and
\begin{equation}\label{e58b}
d_k:=\alpha N_2K_0|h_k|_\Gamma^2.
\end{equation}
Since $|x^k_0-x_0-z_0|_C=0$, 
Lemma~\ref{l0} is applicable for \eref{e32}, and it yields 
\begin{equation}\label{e58c}
|x^k(t)-x(t)-z^{h_k}(t)|\leq |x^k_t-x_t-z_t|_C\leq(a_k+b_k+c_k+d_k)N_1,\qquad t\in[0,\alpha],
\end{equation}
where $N_1:=e^{L_1N_0\alpha}$,
and hence
\begin{equation}\label{e58e}
\frac{|x^k(t)-x(t)-z^{h_k}(t)|}{|h_k|_\Gamma}
\leq\frac{|x^k_t-x_t-z^{h_k}_t|_C}{|h_k|_\Gamma}
\leq\frac{a_k+b_k+c_k+d_k}{|h_k|_\Gamma}N_1,\quad t\in[0,\alpha],\quad
\end{equation}
which proves both \eref{e33} and \eref{e60}, since Lemmas \ref{l5}, \ref{l6} and \eref{e58b} show that
\begin{equation}\label{e58d}
\lim_{k\to\infty}\frac{a_k+b_k+c_k+d_k}{|h_k|_\Gamma}=0.
\end{equation}
The continuity of $D_2 x(t,\gamma)$  follows from 
Lemma~\ref{l8}.
\end{proof}
\bigskip

\begin{remark}\label{r2}\rm
We comment that if in the statement of Theorem~\ref{t2} the set  $P_2$ is replaced by 
$P_1$, the statements are valid without assumptions (A2) (iv) and (v). To see this we refer to
Remark~\ref{r66}, and in the proof of Theorem~\ref{t2} we use Lemma~4.1 of \cite{Ha4}
to show that $c_k/|h_k|_\Gamma\to0$ as $k\to\infty$. We also note that continuous differentiability of $x$ wrt the
parameters holds in a neighborhood of $\gamma$, since $P_1$ is open in $P$. 
See Theorem 4.7 in \cite{Ha4} for a related result.
\end{remark}
\bigskip

\section{Second-order differentiability wrt the parameters}\label{sec_DDE_second}
\setcounter{equation}0
\setcounter{theorem}0

To obtain second-order differentiability wrt the parameters we need more smoothness of the initial
functions. Therefore we introduce the parameter set
$$
\Gamma_2:=\Www\times\Theta\times\Xi
$$
equipped with the norm $|h|_{\Gamma_2}:=|h^\phi|_\Www+|h^\theta|_\Theta+|h^\xi|_\Xi$.
We will show in Theorem~\ref{t3} below that the parameter map 
$$
\Gamma_2\supset (P_2\cap\Gamma_2) \to\setR^n,\qquad \gamma\to x(t,\gamma)
$$
is twice differentiable at every point $\gamma\in P_2\cap\Gamma_2\cap\calP$. The proof will be based on a sequence of lemmas. 
\medskip

We assume throughout this section
\begin{itemize}
 \item[(H)] $\gamma=(\phi,\theta,\xi)\in P_2\cap\Gamma_2$, $ h=(h^\phi,h^\theta,h^\xi)\in\Gamma$, 
$h_k=(h^\phi_k,h^\theta_k,h^\xi_k)\in\Gamma$ ($k\in\setN$) are so that 
$|h_k|_\Gamma\to0$ as $k\to\infty$, 
$\gamma+h_k\in P_2$ for $k\in\setN$, and $|h_k|_\Gamma\neq0$ for $k\in\setN$.  Let $x^k(t):=x(t,\gamma+h_k)$ 
and $x(t):=x(t,\gamma)$ be the solutions of the \IVP{e1}{e2}, 
 $z^{k,h}(t):=D_2x(t,\gamma+h_k) h$ and  $z^h(t):=D_2x(t,\gamma) h$ be the solutions of the \IVP{e3}{e4}.
\end{itemize}
\medskip

The simplifying notations for $t\in[0,\alpha]$ and $k\in \setN$
\begin{eqnarray*}
u(t)&:=&t-\tau(t,x_t,\xi),\nl
u^k(t)&:=&t-\tau(t,x^k_t,\xi+h^\xi_k),\nl
\vc v(t) &:=&(t,x_t,x(u(t)),\theta),\nl
\vc v^k(t) &:=&(t,x^k_t,x^k(u^k(t)),\theta),\nl
A(t,h^\phi,h^\xi) &:=& D_2\tau(t,x_t,\xi)h^\phi+D_3\tau(t,x_t,\xi)h^\xi,\nl
A^k(t, h^\phi,h^\xi) &:=& D_2\tau(t,x^k_t,\xi+h^\xi_k)h^\phi
+D_3\tau(t,x^k_t,\xi+h^\xi_k)h^\xi,\nl
E(t,h^\phi,h^\xi)&:=&-\dot x(u(t))A(t,h^\phi,h^\xi)+h^\phi(-\tau(t,x_t,\xi)),\quad\mbox{a.e. } t\in[0,\alpha],\nl
E^k(t,h^\phi,h^\xi)&:=&-\dot x^k(u^k(t))A^k(t,h^\phi,h^\xi)+h^\phi(-\tau(t,x^k_t,\xi+h^\xi_k)),
\quad\mbox{a.e. } t\in[0,\alpha],\nl
F(t,h^\phi,h^\xi)&:=&-\ddot x(u(t))A(t,h^\phi,h^\xi)+\dot h^\phi(-\tau(t,x_t,\xi)),\quad\mbox{a.e. } t\in[0,\alpha],\nl
F^k(t,h^\phi,h^\xi)&:=&-\ddot x^k(u^k(t))A^k(t,h^\phi,h^\xi)+\dot h^\phi(-\tau(t,x^k_t,\xi+h^\xi_k)),
\quad\mbox{a.e. } t\in[0,\alpha]
\end{eqnarray*}
will be used throughout this section.
For simplicity of the notation we define $h_0:=0=(0,0,0)\in\Gamma$, and accordingly, 
$x^0:=x$, $u^0:=u$, $z^{0,h}:=z^h$, $A^0:=A$, $E^0:=E$.
Note that in all the above abbreviations the dependence on $\gamma$ is omitted from the notation
but it should be kept in mind.  
With these notations  the operator $L(t,x)$ defined by \eref{e26} can be written shortly as
$$
L(t,x) h=D_2f(\vc v(t))h^\phi+D_3f(\vc v(t))E(t, h^\phi,h^\xi)+D_4f(\vc v(t))h^\theta.
$$
\smallskip

\begin{lemma}\label{l9}
 Assume (A1) (i)--(iv), (A2) (i)--(iv) and  $\gamma=(\phi,\theta,\xi)\in P$ is such that $\phi\in\Www$. 
Then there exists $K_4=K_4(\gamma)\geq0$ such that the solution $x(t)=x(t,\gamma)$ of the \IVP{e1}{e2}
satisfies
\begin{equation}\label{e382}
 |\dot x(t)-\dot x(\bar t)| \leq K_4|t-\bar t|\qquad \mbox{for } t,\bar t\in[-r,0)\quad\mbox{and}\quad t,\bar t\in(0,\alpha].
\end{equation}
Moreover, if in addition $\gamma\in \calP$, then  $x\in \Www([-r,\alpha],\setR^n)$, and
\end{lemma}
\begin{equation}\label{e382b}
 |\dot x(t)-\dot x(\bar t)| \leq K_4|t-\bar t|\qquad \mbox{for } t,\bar t\in[-r,\alpha].
\end{equation}
\begin{proof}
The Mean Value Theorem and the definition of the $\Www$-norm 
yield
$$
|\dot x(t)-\dot x(\bar t)|=|\dot \phi(t)-\dot \phi(\bar t)|\leq |\phi|_\Www|t-\bar t|,\qquad t,\bar t\in[-r,0).
$$ 
For $t,\bar t\in(0,\alpha]$ it follows from (A1) (ii), (iv), (A2) (ii), (iv), \eref{e37} and \eref{e383} with $k=0$
\begin{eqnarray*}
|\dot x(t)-\dot x(\bar t)|
&=& |f(t,x_t,x(u(t)),\theta)-f(\bar t,x_{\bar t},x(u(\bar t)),\theta)|\nl
&\leq& L_1\Bigl(|t-\bar t|+|x_t-x_{\bar t}|_C+|x(u(t))-x(u(\bar t))|\Bigr)\nl
&\leq& L_1\Bigl(1+N+NL_2(1+N)\Bigr)|t-\bar t|.
\end{eqnarray*}
Hence \eref{e382} is satisfied with $K_4:=\max\{|\phi|_\Www,L_1[1+N+NL_2(1+N)]\}$.

If $\gamma\in\calP$, then $\dot x$ is continuous, and \eref{e382} yields that it is Lipschitz continuous
on $[-r,\alpha]$ with the Lipschitz constant $K_4$, so, in particular, $x\in \Www([-r,\alpha],\setR^n)$.
\end{proof}
\bigskip

\begin{lemma}\label{l10}
 Assume (A1) (i)--(iii), (A2) (i)--(v), and  (H). 
Then 
\begin{equation}\label{e381}
\lim_{k\to\infty} \frac1{|h_k|_\Gamma}\int_0^\alpha |\dot x^k(s)-\dot x(s)-\dot z^{h_k}(s)|\,ds=0,
\end{equation}
and
\begin{equation}\label{e380}
\lim_{k\to\infty} \frac1{|h_k|_\Gamma}\int_0^\alpha |\dot x^k(u^k(s))-\dot x(u^k(s))-\dot z^{h_k}(u^k(s))|\,ds=0.
\end{equation}
\end{lemma}
\begin{proof}
 Using \eref{e101}, \eref{e201}, \eref{e32} and \eref{e58c} we get
\begin{eqnarray*}
\balra{0.5cm}{\int_0^\alpha|\dot x^k(s)-\dot x(s)-\dot z^{h_k}(s)|\,ds}\nl
&\leq&\int_0^\alpha \Bigl[L_1\Bigl(|x^k_s-x_s-z^{h_k}_s|_C+|x^k(u^k(s))-x(u^k(s))-z^{h_k}(u^k(s))|\nl
&&\qquad +\ |x(u^k(s))-x(u(s))-\dot x(u(s))(u^k(s)-u(s))|\nl
 &&\qquad +\ N|\omega_\tau(s,x_s,\xi,x^k_s,\xi+h^\xi_k)|
  + NL_2|x^k_s-x_s-z^{h_k}_s|_C + N_2K_0|h_k|_\Gamma^2\Bigr)\nl
&&\qquad+\ |\omega_f(s,x_s,x(u(s)),\theta,x^k_s,x^k(u^k(s)),\theta+h^\theta_k)|\Bigr]ds\nl
&\leq& a_k+b_k+c_k+d_k +L_1N_0\int_0^\alpha |x^k_s-x_s-z^{h_k}_s|_C\,ds\nl
&\leq& (a_k+b_k+c_k+d_k)(1 +L_1N_0N_1\alpha),
\end{eqnarray*}
where $a_k$, $b_k$, $c_k$ and $d_k$ are defined by \eref{e57}--\eref{e58b}, respectively.
Then \eref{e381} is obtained from \eref{e58d}.

Relation \eref{e380} follows from \eref{e381}, $x^k(s)-x(s)-z^{h_k}(s)=0$ for $s\in[-r,0]$, 
$|\dot x^k(s)-\dot x(s)-\dot z^{h_k}(s)|\leq (L+N_2)|h_k|_\Gamma$ for $s\in[-r,0]$,  and Lemmas~\ref{l21} and \ref{l11}.
\end{proof}
\medskip

\begin{lemma}\label{l25}
Assume (A1) (i)--(v), (A2) (i)--(vi),  (H) and $\gamma\in\calP$.
Then there exists $N_4=N_4(\gamma)\geq0$ such that
\begin{equation}\label{e386}
|\dot z^{h}(s)-\dot z^{h}(\bar s)|\leq N_4|h|_{\Gamma_2}|s-\bar s|,\quad\mbox{for}\quad  s,\bar s\in[-r,0)\quad\mbox{and}\quad s,\bar s\in(0,\alpha],
\quad h\in\Gamma_2.
\end{equation}
\end{lemma}
\begin{proof} 
For $h\in\Gamma_2$, i.e., $h^\phi\in\Www$, the function 
 $\dot h^\phi$ is continuous, and for $s,\bar s\in[-r,0)$
$$
|\dot z^h(s)-\dot z^h(\bar s)|=|\dot h^\phi(s)-\dot h^\phi(\bar s)|\leq |h^\phi|_\Www|s-\bar s|\leq |h|_{\Gamma_2}|s-\bar s|.
$$

Since $\gamma\in\calP$, $L(s,x)$ is defined and continuous for all $s\in[0,\alpha]$,
so $\dot z^h$ is continuous on $(0,\alpha]$. 
For $s,\bar s\in(0,\alpha]$ \eref{e368} and \eref{e3}  imply
{\arraycolsep=0.6\arraycolsep
\begin{eqnarray}
|\dot z^h(s)-\dot z^h(\bar s)|
&=&|L(s,x)(z^h_s,h^\theta,h^\xi)-L(\bar s,x)(z^h_{\bar s},h^\theta,h^\xi)|\nl
&\leq& |[L(s,x)-L(\bar s,x)](z^h_s,h^\theta,h^\xi)|+|L(\bar s,x)(z^h_s-z^h_{\bar s},0,0)|\nl
&\leq& |[D_2 f(\vc v(s))-D_2 f(\vc v(\bar s))]z^h_s|
+|[D_3 f(\vc v(s))-D_3 f(\vc v(\bar s))]E(s,z^h_s,h^\xi)|\nl
&&+|D_3 f(\vc v(\bar s))[E(s,z^h_s,h^\xi)-E(\bar s,z^h_{\bar s},h^\xi)]|\nl
&& +|[D_4 f(\vc v(s))-D_4 f(\vc v(\bar s))]h^\theta|+L_1N_0|z^h_s-z^h_{\bar s}|_C.
\label{e375}
\end{eqnarray}}
We have by \eref{e37} and \eref{e383} with $k=0$ for $s,\bar s\in[0,\alpha]$
\begin{equation}\label{e385}
|\vc v(s)-\vc v(\bar s)|\leq |s-\bar s|+|x_s-x_{\bar s}|_C+|x(u(s))-x(u(\bar s))|
\leq K_5|s-\bar s|
\end{equation}
and 
\begin{equation}\label{e385b}
|(s,x_s,\xi)-(\bar s,x_{\bar s},\xi)|\leq (1+N)|s-\bar s|
\end{equation}
with $K_5:=(1+N+NL_2(1+N))$ and $(1+N):=1+N$.
Let $L_3:=L_3(\alpha,M_1,M_2,M_3)$ and $L_5:=L_5(\alpha,M_1,M_2,M_3)$ be defined by
(A1) (v) and (A2) (vi), respectively. 

The definition of $A$, (A2) (ii) and \eref{e45} give
\begin{equation}\label{e384}
|A(s,z^h_s,h^\xi)|\leq |D_2\tau(s,x_s,\xi)z^h_s|+|D_3\tau(s,x_s,\xi)h^\xi|
\leq K_6|h|_\Gamma,\quad s\in [0,\alpha],\ h\in\Gamma,\ \gamma\in P_2
\end{equation}
with $K_6:=L_2(N_1+1)$,  and by using (A2) (ii), (vi), \eref{e45}, \eref{e94}, \eref{e385b} 
{\arraycolsep=0.6\arraycolsep
\begin{eqnarray}
|A(s,z^{h}_s,h^\xi)-A(\bar s,z^{h}_{\bar s},h^\xi)|
&\leq& |[D_2\tau(s,x_s,\xi)-D_2\tau(\bar s,x_{\bar s},\xi)]z^h_s|+|D_2\tau(\bar s,x_{\bar s},\xi)[z^h_s-z^{h}_{\bar s}]|\nl
&&+|[D_3\tau(s,x_s,\xi)-D_3\tau(\bar s,x_{\bar s},\xi)]h^\xi|\nl
&\leq& K_7|s-\bar s||h|_\Gamma,\qquad s,\bar s\in[0,\alpha] \label{e388}
\end{eqnarray}}%
with $K_7:=L_5(1+N)N_1+L_2N_2+L_5(1+N)$.
Relations \eref{e37},  \eref{e45} and \eref{e384} yield
\begin{eqnarray}
|E(s,z^{h}_s,h^\xi)|
&\leq& |\dot x(u(s))||A(s,z^h_s,h^\xi)|+|z^{h}(u(s))|\nl
&\leq& K_8|h|_\Gamma,\qquad  s\in[0,\alpha],\ h\in\Gamma,\ \gamma\in P_2\label{e389a}
\end{eqnarray} 
with $K_8:=NK_6+N_1$, and using \eref{e37}, \eref{e383} with $k=0$, \eref{e94}, \eref{e382b}, \eref{e384} 
and  \eref{e388}  
\begin{eqnarray}
\balra{1cm}{|E(s,z^h_s,h^\xi)-E(\bar s,z^h_{\bar s},h^\xi)|}\nl
&\leq& |[\dot x(u(s))-\dot x(u(\bar s))]A(s,z^h_s,h^\xi)|+|\dot x(u(\bar s))[A(s,z^h_s,h^\xi)-A(\bar s,z^h_{\bar s},h^\xi)]|\nl
&&+|z^h(u(s))-z^h(u(\bar s))|\nl
&\leq& K_9|s-\bar s||h|_\Gamma,\qquad s,\bar s\in[0,\alpha]\label{e389}
\end{eqnarray}
with $K_9=K_9(\gamma):=K_4L_2(1+N)K_6+NK_7+N_2L_2(1+N)$.
Then combining \eref{e375} with \eref{e385}, \eref{e389a} and \eref{e389} yields
$$
|\dot z^h(s)-\dot z^h(\bar s)|
\leq (L_3K_5N_1+L_3K_5K_8+L_1K_9+L_3K_5+L_1N_0N_2)|s-\bar s||h|_\Gamma
$$
for $s,\bar s\in[0,\alpha]$ and $h \in\Gamma$.
Hence  $N_4:=\max\{1,L_3K_5N_1+L_3K_5K_8+L_1K_9+L_3K_5+L_1N_0N_2\}$ satisfies \eref{e386}.
\end{proof}
\medskip

\begin{lemma}\label{l25b}
Assume (A1) (i)--(v), (A2) (i)--(vi), (H) and $\gamma\in\calP$.
 Then 
\begin{equation}\label{e379}
\lim_{k\to\infty}\sup_{h\neq0\atop h\in\Gamma_2}\frac1{|h|_{\Gamma_2}}
\int_0^\alpha |\dot z^{h}(u^k(s))-\dot z^{h}(u(s))|\,ds=0.
\end{equation}
\end{lemma}
\begin{proof}
Since $\gamma\in  P_2$ and $u(0)\leq0$, it follows that  $u$ has finitely many zeros on $[0,\alpha]$. Let
 $0\leq s_1<s_2<\cdots<s_\ell\leq \alpha$ be the mesh points where $u(s_i)=0$,  
$0<\epsilon<\min\{s_{i+1}-s_i\st i=1,\ldots,\ell-1\}/2$ be fixed, and introduce  $s_i':=\min\{s_i+\epsilon,\alpha\}$ 
and $s_i'':=\max\{s_i-\epsilon,0\}$ for $i=1,\ldots,\ell$, $s_0':=0$, $s_{\ell+1}'':=\alpha$, and 
let
$$
M:=\min_{i=1,\ldots,\ell-1}\min_{s\in[s_i',s_{i+1}'']} |u(s)|.
$$
We have $M>0$. Relation \eref{e88} yields that there exist $k_0>0$ such that 
$|u^k-u|_{C([0,\alpha],\,\setR)}<\frac M2$ for $k\geq k_0$. 
Then for $k\geq k_0$ it follows $|u^k(s)|\geq \frac M2$ for $s\in[s_i',s_{i+1}'']$ and $i=0,\ldots,\ell$.
Note that $h\in \Gamma_2$ and $\gamma\in\calP$ yield $\dot z^h$ is continuous on $[-r,0)$ and $(0,\alpha]$, 
and \eref{e94} implies $|\dot z^h(s)|\leq N_2|h|_\Gamma\leq N_2|h|_{\Gamma_2}$ for $s\neq0$.
Therefore $|\dot z^h(u^k(s))|\leq N_2|h|_{\Gamma_2}$ for a.e. $s\in[0,\alpha]$, since,  by  assumption (H),
$\gamma+h_k\in P_2$, hence
 $u^k\in{\cal PM}([0,\alpha],[-r,\alpha])$.  
Then  \eref{e88}, \eref{e94} and \eref{e386} yield
\begin{eqnarray*}
\balra{1cm}{\int_0^\alpha |\dot z^{h}(u^k(s))-\dot z^{h}(u(s))|\,ds}\\
&\leq& \sum_{i=1}^\ell \int_{s_i''}^{s_i'} [|\dot z^{h}(u^k(s))|+|\dot z^{h}(u(s))|]\,ds
+\sum_{i=0}^{\ell} \int_{s_i'}^{s_{i+1}''} |\dot z^{h}(u^k(s))-\dot z^{h}(u(s))|\,ds\\
&\leq& 4\ell \epsilon N_2|h|_{\Gamma_2}+(\ell+1) \alpha N_4 K_0|h|_{\Gamma_2}|h_k|_\Gamma.
\end{eqnarray*}
This concludes the proof of \eref{e379}, since $\epsilon>0$ can be arbitrary close to 0.
\end{proof}
\bigskip

\begin{lemma}\label{l23}
Assume (A1) (i)--(v), (A2) (i)--(vi), (H) and $\gamma\in\calP$. Then 
\begin{equation}\label{e372}
\lim_{k\to\infty}\sup_{h\neq0\atop
h\in\Gamma_2}\frac{1}{|h|_{\Gamma_2}|h_k|_\Gamma}\int_0^\alpha
|z^{h}(u^k(s))-z^{h}(u(s))-\dot z^{h}(u(s))(u^k(s)-u(s))|\,ds=0.
\end{equation}
\end{lemma}
\begin{proof}
Let $s_i,s_i',s_i''$,  $\ell$, $\epsilon$, $M$  and $k_0$ be defined as in the proof of Lemma~\ref{l25b}.
Then $|u(s)+\nu(u^k(s)-u(s))|>\frac M2$, and $u(s)$ and $u(s)+\nu(u^k(s)-u(s))$  are both either positive or negative 
for $s\in[s_i',s_{i+1}'']$, $\nu\in[0,1]$ and $i=0,\ldots,\ell$. 
Therefore \eref{e88} and \eref{e386} yield 
$$
|\dot z^h(u(s)+\nu(u^k(s)-u(s)))-\dot z^h(u(s))|\leq N_4|h|_{\Gamma_2}|u^k(s)-u(s)|\leq N_4K_0|h|_{\Gamma_2}|h_k|_\Gamma.
$$
Hence, using Fubini's Theorem, \eref{e88}  and \eref{e94} 
we have
\begin{eqnarray*}
\balra{1cm}{ \int_0^\alpha
|z^{h}(u^k(s))-z^{h}(u(s))-\dot z^{h}(u(s))(u^k(s)-u(s))|\,ds}\nl
&\leq&\sum_{i=1}^\ell \int_{s_i''}^{s_i'} \Bigl(|z^{h}(u^k(s))-z^{h}(u(s))|+|\dot z^{h}(u(s))||u^k(s)-u(s))|\Bigr)ds\nl
&&\quad+\sum_{i=0}^{\ell} \int_{s_i'}^{s_{i+1}''}|z^{h}(u^k(s))-z^{h}(u(s))-\dot z^{h}(u(s))(u^k(s)-u(s))|\,ds\nl
&\leq&4\epsilon \ell N_2 K_0|h|_\Gamma|h_k|_\Gamma\nl
&&\quad+ \sum_{i=0}^{\ell} \int_{s_i'}^{s_{i+1}''}\left|\int_0^1 [\dot z^{h}(u(s)+\nu(u^k(s)-u(s)))-
\dot z^{h}(u(s))][u^k(s)-u(s)]\,d\nu\right|ds\nl
&\leq&4\epsilon\ell N_2 K_0|h|_\Gamma|h_k|_\Gamma\nl
&&\quad+ K_0|h_k|_\Gamma\sum_{i=0}^{\ell}\int_0^1 \int_{s_i'}^{s_{i+1}''}|\dot z^{h}(u(s)+\nu(u^k(s)-u(s)))-
\dot z^{h}(u(s))|ds\,d\nu\nl
&\leq&4\epsilon\ell N_2 K_0|h|_{\Gamma_2}|h_k|_\Gamma
+ K_0^2(\ell+1)\alpha N_4|h|_{\Gamma_2}|h_k|_\Gamma^2.
\end{eqnarray*}
This completes the proof of \eref{e372}, since $\epsilon>0$ is arbitrary close to 0.
\end{proof}
\medskip

\begin{lemma}\label{l20}
 Assume (A1) (i)--(iii), (A2) (i)--(v), (H). Then 
\begin{equation}\label{e364}
\lim_{k\to\infty}\sup_{h\neq0\atop h\in\Gamma }\frac{1}{|h|_\Gamma}\int_0^\alpha |\dot z^{k,h}(s)-\dot z^{h}(s)|\,ds=0, 
\end{equation}
and 
\begin{equation}\label{e376}
\lim_{k\to\infty}\sup_{|h|_\Gamma\neq0}\frac{1}{|h|_\Gamma|h_k|_\Gamma} \int_0^\alpha |z^{k,h}(u^k(s))-z^{h}(u^k(s))-[z^{k,h}(u(s))-z^{h}(u(s))]|\,ds
=0.
\end{equation}
\end{lemma}
\begin{proof}
 For $s\in[0,\alpha]$ combining \eref{e368}, \eref{e3},  \eref{e395}, \eref{e397} and \eref{e344} we get
\begin{eqnarray*}
\balra{1cm}{|\dot z^{k,h}(s)-\dot z^{h}(s)|}\nl
&\leq& |L(s,x^k)(z^{k,h}_s-z^{h}_s,0,0)|+|(L(s,x^k)-L(s,x))(z^{h}_s,h^\theta,h^\xi)|\nl
&\leq& L_1N_0c_{1,k}N_1|h|_\Gamma+c_{0,k}(N_2+1)|h|_\Gamma+L_1L_2(N_2+1)|\dot x(u^k(s))-\dot x(u(s))||h|_\Gamma.
\end{eqnarray*}
Hence Lemmas~\ref{l4} and \ref{l11} yield \eref{e364}.

Define the functions
$$
f^{k,h}(s):=\frac{|\dot z^{k,h}(s)-\dot z^{h}(s)|}{|h|_\Gamma},
$$
and the set $H:=\{h\in\Gamma\st h\neq0\}$. 
 Note that \eref{e368}, \eref{e3} and \eref{e45} yield 
$|\dot z^{k,h}(s)|=|L(s,x^k)z^{k,h}_s|\leq L_1N_0N_1|h|_\Gamma$ for $k\in\setN_0$ and $s\in[0,\alpha]$, so
$|f^{k,h}(s)|\leq 2L_1N_0N_1$ for a.e.\ $s\in[-r,\alpha]$,
$k\in\setN$ and $h\in H$.
Then it follows from \eref{e364}, $z^{k,h}(s)-z^h(s)=0$ for $s\in[-r,0]$, and Lemmas~\ref{l21} and \ref{l11}
that for any fixed $\nu\in[0,1]$ 
\begin{equation}\label{e365}
\lim_{k\to\infty}\sup_{h\neq0\atop h\in\Gamma }\frac{1}{|h|_\Gamma}\int_0^\alpha \Bigl|\dot z^{k,h}\Bigl(u(s)+\nu(u^k(s)-u(s))\Bigr)
-\dot z^{h}\Bigl(u(s)+\nu(u^k(s)-u(s))\Bigr)\Bigr|\,ds=0.\  \
\end{equation}

\eref{e88} and Fubini's Theorem  yield
\begin{eqnarray*}
\balra{0.cm}{\int_0^\alpha |z^{k,h}(u^k(s))-z^{h}(u^k(s))-[z^{k,h}(u(s))-z^{h}(u(s))]|\,ds}\nl
&=& \int_0^\alpha \Bigl|\int_0^1 \Bigl[\dot z^{k,h}\Bigl(u(s)+\nu(u^k(s)-u(s))\Bigr)
-\dot z^{h}\Bigl(u(s)+\nu(u^k(s)-u(s))\Bigr)\Bigr]\nl
&&\qquad\times[u^k(s)-u(s)]\,d\nu\Bigr|\,ds\nl
&\leq& K_0|h_k|_\Gamma \int_0^1\int_0^\alpha\Bigl|\dot z^{k,h}\Bigl(u(s)+\nu(u^k(s)-u(s))\Bigr)
-\dot z^{h}\Bigl(u(s)+\nu(u^k(s)-u(s))\Bigr)\Bigr|\,ds\,d\nu.
\end{eqnarray*}
Therefore \eref{e365} and the Dominated Convergence Theorem imply \eref{e376}.
\end{proof}
\bigskip

Introduce the notation 
$$p^k(t):=x^k(t)-x(t)-z^{h_k}(t).
$$
Then, under the assumptions of Theorem~\ref{t2}, \eref{e58e} and \eref{e58d} give 
\begin{equation}\label{e335}
\lim_{k\to\infty}\max_{s\in[-r,\alpha]}\frac{|p^k(s)|}{|h_k|_\Gamma}=0.
\end{equation}
To linearize equation \eref{e3} around a fixed solution $z$ we will need the following results.
\medskip

\begin{lemma}\label{l13}
 Assume (A1) (i)--(v), (A2) (i)--(vi), (H) and $\gamma\in\calP$.
Then
\begin{itemize}
\item[(i)]
\begin{equation}\label{e373}
u^k(s)-u(s)+A(s,z^{h_k}_s,h^\xi_k)
= g^{k}_{0}(s),\qquad s\in[0,\alpha],
\end{equation}
where 
$$
g^{k}_{0}(s):=  -\omega_\tau(s,x_s,\xi,x^k_s,\xi+h^\xi_k)-D_2\tau(s,x_s,\xi)p^k_s\nl
$$
 satisfies
\begin{equation}\label{e374}
\lim_{k\to\infty} \frac1{|h_k|_\Gamma} \int_0^\alpha |g^{k}_{0}(s)|\,ds=0;
\end{equation}
\item[(ii)]
\begin{equation}\label{e300}
x^k(u^k(s))-x(u(s))-E(s,z^{h_k}_s,h^\xi_k)
= g^{k}_{1}(s),\qquad s\in[0,\alpha],
\end{equation}
where 
\begin{eqnarray*}
g^{k}_{1}(s)&:=&p^k(u^k(s))+x(u^k(s))-x(u(s))-\dot x(u(s))(u^k(s)-u(s))+\dot x(u(s))g^{k}_{0}(s)\nl
&&+z^{h_k}(u^k(s))-z^{h_k}(u(s))
\end{eqnarray*}
 satisfies
\begin{equation}\label{e304}
\lim_{k\to\infty} \frac1{|h_k|_\Gamma} \int_0^\alpha |g^{k}_{1}(s)|\,ds=0;
\end{equation}
and 
\item[(iii)] if $h_k\in\Gamma_2$ for $k\in\setN$, then
\begin{equation}\label{e305}
\dot x^k(u^k(s))-\dot x(u(s))-F(s,z^{h_k}_s,h^\xi_k)=g^{k}_{2}(s),\qquad s\in[0,\alpha],
\end{equation}
where
\begin{eqnarray*}
g^{k}_{2}(s)&:=& \dot x^k(u^k(s))-\dot x(u^k(s))-\dot z^{h_k}(u^k(s))+\dot z^{h_k}(u^k(s))-\dot z^{h_k}(u(s))\nl
&&+\dot x(u^k(s))-\dot x(u(s))-\ddot x(u(s))(u^k(s)-u(s))\nl
&& -\ddot x(u(s))\omega_\tau(s,x_s,\xi,x^k_s,\xi+h^\xi_k)-\ddot x(u(s))D_2\tau(s,x_s,\xi)p^k_s
\end{eqnarray*}
 satisfies
\begin{equation}\label{e306}
\lim_{k\to\infty} \frac1{|h_k|_{\Gamma_2}} \int_0^\alpha |g^{k}_{2}(s)|\,ds=0.
\end{equation}
\end{itemize}
\end{lemma}
\begin{proof}
The definition of $\omega_\tau$ and $A$ imply
\begin{eqnarray*}
 \balra{1cm}{u^k(s)-u(s)+A(s,z^{h_k}_s,h^\xi_k)}\nl
&=& -[\tau(s,x^k_s,\xi+h^\xi_k)-\tau(s,x_s,\xi)-D_2\tau(s,x_s,\xi)(x^k_s-x_s)-D_2\tau(s,x_s,\xi)h^\xi_k]\nl
&&-D_2\tau(s,x_s,\xi)(x^k_s-x_s-z^{h_k}_s),\qquad s\in[0,\alpha],
\end{eqnarray*}
which shows \eref{e373}. \eref{e374} follows from $|D_2\tau(s,x_s,\xi)|_{\calL(C,\,\setR)}\leq L_2$ 
for $s\in[0,\alpha]$, \eref{e21} and \eref{e335}.

 Relation \eref{e55b} and the definition of $g^{k}_{1}$ yield \eref{e300}.
We have by \eref{e37} and \eref{e89}
\begin{eqnarray*}
\int_0^\alpha\! |g^{k}_{1}(s)|\,ds\!%
&\leq& \alpha\max_{s\in[-r,\alpha]}|p^k(s)|+\int_0^\alpha |x(u^k(s))-x(u(s))-\dot x(u(s))(u^k(s)-u(s))|\,ds\nl
&&+N\int_0^\alpha |g^{k}_{0}(s)|\,ds+\alpha N_2K_0| h_k|^2_\Gamma.
\end{eqnarray*}
Therefore  \eref{e335}, \eref{e374}, and Lemmas~\ref{l5} and \ref{l11} yield \eref{e304}.

Simple computation and the definition of $g^{k}_{2}$ imply \eref{e305} immediately. 
Note that $\gamma\in\calP$ yields that $\dot x$ is continuous on $[-r,\alpha]$, and $\phi\in\Www$ and
Lemma~\ref{l9} imply that $x\in\Www([-r,\alpha],\setR^n)$. 
Then \eref{e332} and Lemma~\ref{l5} with $y=\dot x$ yield  
\begin{equation}\label{e12b}
\lim_{k\to\infty}\frac 1{|h_k|_\Gamma}\int_0^\alpha |\dot x(u^k(s))-\dot x(u(s))-\ddot x(u(s))(u^k(s)-u(s))|\,ds=0.
\end{equation}
We have
by \eref{e382} and Lemma~\ref{l3} that $|\ddot x(u(s))|\leq K_4$ for a.e.\ $s\in[0,\alpha]$, therefore 
\begin{eqnarray*}
\int_0^\alpha\! |g^{k}_{2}(s)|\,ds\!%
&\leq&\! \int_0^\alpha\! |\dot x^k(u^k(s))-\dot x(u^k(s))-\dot z^{h_k}(u^k(s))|\,ds\nl
&&+ \int_0^\alpha\ |\dot z^{h_k}(u^k(s))-\dot z^{h_k}(u(s))|\,ds\nl
&&+\int_0^\alpha |\dot x(u^k(s))-\dot x(u(s))-\ddot x(u(s))(u^k(s)-u(s))|\,ds\nl
&&+K_4\int_0^\alpha\! |\omega_\tau(s,x_s,\xi,x^k_s,\xi+h^\xi_k)|\,ds+\alpha K_4L_2\max_{s\in[0,\alpha]}|p^k_s|_C.
\end{eqnarray*}
Hence \eref{e21}, \eref{e380},  \eref{e379}, \eref{e335} and \eref{e12b} imply \eref{e306}.
\end{proof}
\medskip

We define the notations
\begin{eqnarray*}
\balra{0.5cm}{\omega_{D_2\tau}(s,\bar\phi,\bar\xi,\phi,\xi,\psi)}\nl
&:=& D_2\tau(s,\phi,\xi)\psi-D_2\tau(s,\bar\phi,\bar\xi)\psi
-D_{22}\tau(s,\bar\phi,\bar\xi)\langle \psi,\phi-\bar\phi\rangle 
-D_{23}\tau(s,\bar\phi,\bar\xi)\langle \psi,\xi-\bar \xi\rangle \nl
\balra{0.5cm}{\omega_{D_3\tau}(s,\bar\phi,\bar\xi,\phi,\xi,\chi)}\nl
&:=& D_3\tau(s,\phi,\xi)\chi-D_3\tau(s,\bar\phi,\bar\xi)\chi
-D_{32}\tau(s,\bar\phi,\bar\xi)\langle \chi,\phi-\bar\phi\rangle 
-D_{33}\tau(s,\bar\phi,\bar\xi)\langle \chi,\xi-\bar \xi\rangle 
\end{eqnarray*}
for $s\in[0,\alpha]$, $\bar\phi,\phi\in\Omega_1$, $\bar\xi,\xi\in\Omega_4$, $\psi\in C$ and $\chi\in\Xi$.
\medskip

\begin{lemma}\label{l17}
 Assume (A2) (i)--(vii) and (H).
Then
\begin{equation}\label{e356}
\lim_{k\to\infty}\sup_{h\neq0\atop h\in\Gamma}\frac1{|h|_\Gamma|h_k|_\Gamma}
\int_0^\alpha |\omega_{D_2\tau}(s,x_s,\xi,x^k_s,\xi+h^\xi_k,z^{k,h}_s)|\,ds=0,
\end{equation}
and
\begin{equation}\label{e357}
\lim_{k\to\infty}\sup_{h\neq0\atop h\in\Gamma}\frac1{|h|_\Gamma|h_k|_\Gamma}
\int_0^\alpha |\omega_{D_3\tau}(s,x_s,\xi,x^k_s,\xi+h^\xi_k,h^\xi)|\,ds=0.
\end{equation}
\end{lemma}
\begin{proof}
Let $L_5=L_5(\alpha,M_1,M_3)$ be defined by (A2) (vi).
Then (A2) (vi), \eref{e25}, \eref{e45} and \eref{e396} yield for $s\in[0,\alpha]$
\begin{eqnarray*}
|D_2\tau(s,x^k_s,\xi+h^\xi_k)z^{k,h}_s-D_2\tau(s,x_s,\xi)z^{k,h}_s|
&\leq& L_5(L+1)N_1|h_k|_\Gamma|h|_\Gamma,\nl
|D_{22}\tau(s,x_s,\xi)\langle z^{k,h}_s,x^k_s-x_s\rangle 
&\leq& L_5N_1L|h|_\Gamma|h_k|_\Gamma,\nl
|D_{23}\tau(s,x_s,\xi)\langle z^{k,h}_s,h^\xi_k\rangle 
&\leq& L_5N_1|h|_\Gamma|h_k|_\Gamma,
\end{eqnarray*}
and hence, 
$$
|\omega_{D_2\tau}(s,x_s,\xi,x^k_s,\xi+h^\xi_k,z^{k,h}_s)|\leq 2L_5(L+1)N_1|h_k|_\Gamma|h|_\Gamma,\qquad s\in[0,\alpha].
$$
On the other hand,  for  $s \in[0,\alpha],\ k\in\setN$ and $0\neq h\in\Gamma$ such that
$|x^k_s-x_s|_C+|h^\xi_k|_\Gamma\neq0$ and $|z^{k,h}_s|_C\neq0$, assumption (A2) (vii), \eref{e25} and \eref{e45} yield
\begin{eqnarray*}
 \balra{1cm}{\sup_{|h|_\Gamma\neq0}\frac{|\omega_{D_2\tau}(s,x_s,\xi,x^k_s,\xi+h^\xi_k,z^{k,h}_s)|}{|h|_\Gamma|h_k|_\Gamma}}\nl
&=&\sup_{|h|_\Gamma\neq0}\frac{|\omega_{D_2\tau}(s,x_s,\xi,x^k_s,\xi+h^\xi_k,z^{k,h}_s)|}
{(|x^k_s-x_s|_C+|h^\xi_k|_\Gamma)|z^{k,h}_s|_C}\cdot\frac{(|x^k_s-x_s|_C+|h^\xi_k|_\Gamma)|z^{k,h}_s|_C}{|h|_\Gamma|h_k|_\Gamma}\nl
&\leq& (L+1)N_1\sup_{|h|_\Gamma\neq0}\frac{|\omega_{D_2\tau}(s,x_s,\xi,x^k_s,\xi+h^\xi_k,z^{k,h}_s)|}
{(|x^k_s-x_s|_C+|h^\xi_k|_\Gamma)|z^{k,h}_s|_C}\nl
&\to&0,\qquad k\to\infty.
\end{eqnarray*}
Note that for $s, k$ and $h$ such that
$|x^k_s-x_s|_C+|h^\xi_k|_\Gamma=0$ or $|z^{k,h}_s|_C=0$, $|\omega_{D_2\tau}(s,x_s,\xi,x^k_s,\xi+h^\xi_k,z^{k,h}_s)|=0$.
Therefore the Dominated Convergence Theorem implies \eref{e356}.

The proof of \eref{e357} is similar.
\end{proof}
\medskip

For a.e.\ $s\in[0,\alpha]$, $h,y\in\Gamma$  we introduce
the bilinear operators by
\begin{eqnarray*}
G(s)\langle (h^\phi,h^\xi),(y^\phi,y^\xi)\rangle
&:=&
D_{22}\tau(s,x_s,\xi)\langle h^\phi, y^\phi\rangle+D_{23}\tau(s,x_s,\xi)\langle h^\phi, y^\xi\rangle \nl
&&+D_{32}\tau(s,x_s,\xi)\langle h^\xi, y^\phi\rangle+D_{33}\tau(s,x_s,\xi)\langle h^\xi, y^\xi\rangle,\nl
H(s)\langle(h^\phi,h^\xi),(y^\phi,y^\xi)\rangle
&:=&-A(s,h^\phi,h^\xi)F(s,y^\phi,y^\xi)
-\dot x(u(s))G(s)\langle(h^\phi,h^\xi),(y^\phi,y^\xi)\rangle\nl
&&-\dot h^\phi(-\tau(s,x_s,\xi))A(s,y^\phi,y^\xi),
\end{eqnarray*}
and 
\begin{eqnarray*}
B(s)\langle h,y\rangle\!\!
&:=& D_{22}f(\vc v(s))\langle h^\phi,y^\phi\rangle 
+D_{23}f(\vc v(s))\langle h^\phi,E(s,y^\phi,y^\xi)\rangle 
+D_{24}f(\vc v(s))\langle h^\phi,y^\theta\rangle \nl
&&+D_{32}f(\vc v(s))\langle E(s,h^\phi,h^\xi),y^\phi\rangle 
+D_{33}f(\vc v(s))\langle E(s,h^\phi,h^\xi),E(s,y^\phi,y^\xi)\rangle \nl
&&+D_{34}f(\vc v(s))\langle E(s,h^\phi,h^\xi),y^\theta\rangle +D_{42}f(\vc v(s))\langle h^\theta,y^\phi\rangle\nl
&& 
+D_{43}f(\vc v(s))\langle h^\theta,E(s,y^\phi,y^\xi)\rangle 
+D_{44}f(\vc v(s))\langle h^\theta,y^\theta\rangle\nl
&&+D_3f(\vc v(s))H(s)\langle(h^\phi,h^\xi),(y^\phi,y^\xi)\rangle.
\end{eqnarray*}
Note that $G$, $H$ and $B$ correspond to $\gamma$, but this dependence is omitted for simplicity 
in the notation. 
\medskip

For $\gamma\in P_2$ consider the corresponding solution $x$ of the \IVP{e1}{e2}, and let $z^h$ and $z^y$
be the solutions of the \IVP{e3}{e4} corresponding to a fixed $h,y\in\Gamma$.
We 
consider the IVP
\begin{eqnarray}
 \dot w(t)\!\! &=&\! L(t,x)(w_t,0,0)+B(t)\langle(z^h_t,h^\theta,h^\xi),
(z^{y}_t,y^\theta,y^\xi)\rangle,\quad \mbox{a.e. } t\in[0,\alpha],\qquad\quad \label{e330}\\
w(t)\!\! &=&\! 0,\qquad t\in[-r,0].\label{e331}
\end{eqnarray}
The \IVP{e330}{e331} is a Carath\'eodory type inhomogeneous linear delay system with time-dependent
but state-independent delays. It is easy to see that under assumptions (A1) (i)--(vi), (A2) (i)--(vii) the
\IVP{e330}{e331} has a unique solution on $[-r,\alpha]$, which will be denoted by $w^{h,y}(t):=w(t,\gamma,h,y)$.
It is easy to see that $\Gamma\times\Gamma\to\setR^n,\ (h,y)\mapsto w(t,\gamma,h,y)$ is a bilinear map for a fixed $t\in[0,\alpha]$ and $\gamma\in P_2$.
In Lemma~\ref{l30} below we will show that this bilinear map is bounded. 
\medskip

We need the further notation
\begin{eqnarray*}
q^{k,h}(s)&:=&z^{k,h}(s)-z^{h}(s)-w^{h,{h_k}}(s),\quad s\in[-r,\alpha].
\end{eqnarray*}

\begin{lemma}\label{l19}
Assume (A2) (i)--(vi) and (H).
Then there exists  $K_{10}\geq0$ such that
\begin{equation}\label{e371}
|A^k(s,z^{j,h}_s,h^\xi)-A(s,z^{j,h}_s,h^\xi)|\leq K_{10}|h|_\Gamma|h_k|_\Gamma,\qquad s\in[0,\alpha],\ k\in\setN,\ j\in\setN_0,
\end{equation} 
and there exists  a sequence $c_{2,k}\geq0$ satisfying $c_{2,k}\to0$ as $k\to\infty$
such that
\begin{equation}\label{e363}
|A^k(s,z^{k,h}_s,h^\xi)-A(s,z^{h}_s,h^\xi)|\leq c_{2,k}|h|_\Gamma,\qquad s\in[0,\alpha],\ k\in\setN.
\end{equation} 
\end{lemma}
\begin{proof}
Let $L_5=L_5(\alpha,M_1,M_3)$ be defined by (A2) (vi).
 To show \eref{e363} we use \eref{e25}, \eref{e45}, \eref{e396} and  (A2) (vi)
to get
\begin{eqnarray*}
 \balra{0cm}{|A^k(s,z^{j,h}_s,h^\xi)-A(s,z^{j,h}_s,h^\xi)|}\nl
&\leq& |D_2\tau(s,x^k_s,\xi+h^\xi_k)z^{j,h}_s-D_2\tau(s,x_s,\xi)z^{j,h}_s|
+|D_3\tau(s,x^k_s,\xi+h^\xi_k)h^\xi-D_3\tau(s,x_s,\xi)h^\xi|\nl
&\leq& L_5(L+1)|h_k|_\Gamma N_1|h|_\Gamma +L_5(L+1)|h_k|_\Gamma |h|_\Gamma,\qquad s\in[0,\alpha],\ k\in\setN,\ j\in\setN_0,
\end{eqnarray*}
which yields \eref{e371}.
Using \eref{e344}, \eref{e363} and (A2) (ii) we get
\begin{eqnarray*}
 \balra{1cm}{|A^k(s,z^{k,h}_s,h^\xi)-A(s,z^{h}_s,h^\xi)|}\nl
&\leq& |A^k(s,z^{k,h}_s,h^\xi)-A(s,z^{k,h}_s,h^\xi)|+|A(s,z^{k,h}_s,h^\xi)-A(s,z^{h}_s,h^\xi)|\nl
&\leq& K_{10}|h|_\Gamma|h_k|_\Gamma+|D_2\tau(s,x_s,\xi)(z^{k,h}_s-z^{h}_s)|\nl
&\leq& K_{10}|h_k|_\Gamma |h|_\Gamma + L_2c_{1,k}N_1|h|_\Gamma,\qquad s\in[0,\alpha],\ k\in\setN,
\end{eqnarray*}
therefore \eref{e363} holds.
\end{proof}
\medskip

\begin{lemma}\label{l14}
 Assume (A1) (i)--(v), (A2) (i)--(vii), (H) and $\gamma\in\calP$.
Then 
\begin{eqnarray}
\balra{1cm}{A^k(s,z^{k,h}_s,h^\xi)-A(s,z^{h}_s,h^\xi)-G(s)\langle(z^{h}_s,h^\xi),(z^{h_k}_s,h^\xi_k)\rangle
-A(s,w^{h,{h_k}}_s,0)}\nl
&=& A(s,q^{k,h}_s,0)+g^{k,h}_{3}(s),\qquad s\in[0,\alpha],\ h\in\Gamma,\ k\in\setN,\hspace{2cm}\label{e336}
\end{eqnarray}
where
\begin{eqnarray*}
g^{k,h}_{3}(s)
&:=&D_{22}\tau(s,x_s,\xi)\langle z^{k,h}_s-z^{h}_s,x^k_s-x_s\rangle
+D_{22}\tau(s,x_s,\xi)\langle z^{h}_s,p^k_s\rangle\nl
&&+D_{23}\tau(s,x_s,\xi)\langle z^{k,h}_s-z^{h}_s,h^\xi_k\rangle
+D_{32}\tau(s,x_s,\xi)\langle h^\xi,p^k_s\rangle\nl
&&+\omega_{D_2\tau}(s,x_s,\xi,x^k_s,\xi+h^\xi_k,z^{k,h}_s)
+\omega_{D_3\tau}(s,x_s,\xi,x^k_s,\xi+h^\xi_k,h^\xi)
\end{eqnarray*}
satisfies
\begin{equation}\label{e337}
\lim_{k\to\infty}\, \sup_{h\neq0\atop h\in\Gamma} \frac1{|h|_\Gamma|h_k|_\Gamma} \int_0^\alpha |g^{k,h}_{3}(s)|\,ds=0;
\end{equation}
and if $h_k\in\Gamma_2$ for $k\in\setN$, then
\begin{eqnarray}
\balra{1cm}{ E^k(s,z^{k,h}_s,h^\xi)-E(s,z^{h}_s,h^\xi)-H(s)\langle(z^{h}_s,h^\xi),(z^{h_k}_s,h^\xi_k)\rangle
-E(s,w^{h,{h_k}}_s,0)}\nl
&=& E(s,q^{k,h}_s,0)+g^{k,h}_{4}(s),\qquad \mbox{a.e. } s\in[0,\alpha],\ h\in\Gamma,\ k\in\setN\qquad\qquad\label{e338}
\end{eqnarray}
with
\begin{eqnarray*}
g^{k,h}_{4}(s)
&:=& -[\dot x^k(u^k(s))-\dot x(u(s))][A^k(s,z^{k,h}_s,h^\xi)-A(s,z^{k,h}_s,h^\xi)]-g^{k}_{2}(s)A(s,z^{k,h}_s,h^\xi)\nl
&&-\dot x(u(s))g^{k,h}_{3}(s)+z^{k,h}(u^k(s))-z^{h}(u^k(s))-[z^{k,h}(u(s))-z^{h}(u(s))]\nl
&&+z^{h}(u^k(s))-z^{h}(u(s))-\dot z^{h}(u(s))(u^k(s)-u(s))\nl
&&+\dot z^{h}(u(s))\Bigl(u^k(s)-u(s)+A(s,z^{h_k}_s,h^\xi_k)\Bigl)
\end{eqnarray*}
satisfying
\begin{equation}\label{e339}
\lim_{k\to\infty}\, \sup_{h\neq0\atop h\in\Gamma_2}\frac1{|h|_{\Gamma_2}|h_k|_{\Gamma_2}} \int_0^\alpha |g^{k,h}_{4}(s)|\,ds=0.
\end{equation}
\end{lemma}
\begin{proof}
The definitions of $A^k,A,G,g^{k,h}_{3},\omega_{D_2\tau}$, $\omega_{D_3\tau}$ and relation
$$
A(s,z^{k,h}_s,h^\xi)-A(s,z^{h}_s,h^\xi)-A(s,w^{h,{h_k}}_s,0)=A(s,z^{k,h}_s-z^{h}_s-w^{h,{h_k}}_s,0)
$$ 
yield
\begin{eqnarray*}
\balra{1cm}{ A^k(s,z^{k,h}_s,h^\xi)-A(s,z^{h}_s,h^\xi)-G(s)\langle(z^{h}_s,h^\xi),(z^{h_k}_s,h^\xi_k)\rangle
-A(s,w^{h,{h_k}}_s,0)}\nl
&=&A^k(s,z^{k,h}_s,h^\xi)-A(s,z^{k,h}_s,h^\xi)-G(s)\langle(z^{h}_s,h^\xi),(z^{h_k}_s,h^\xi_k)\rangle
+A(s,q^{k,h}_s,0)\nl
&=& D_2\tau(s,x^k_s,\xi+h^\xi_k)z^{k,h}_s-D_2\tau(s,x_s,\xi)z^{k,h}_s-D_{22}\tau(s,x_s,\xi)\langle z^{k,h}_s,x^k_s-x_s\rangle\nl
&&-D_{23}\tau(s,x_s,\xi)\langle z^{k,h}_s,h^\xi_k\rangle+D_{22}\tau(s,x_s,\xi)\langle z^{k,h}_s-z^{h}_s,x^k_s-x_s\rangle\nl
&&+D_{22}\tau(s,x_s,\xi)\langle z^{h}_s,p^k_s\rangle+D_{23}\tau(s,x_s,\xi)\langle z^{k,h}_s-z^{h}_s,h^\xi_k\rangle\nl
&&+D_3\tau(t,x^k_s,\xi+h^\xi_k)h^\xi-D_3\tau(s,x_s,\xi)h^\xi-D_{32}\tau(s,x_s,\xi)\langle h^\xi,x^k_s-x_s\rangle\nl
&&-D_{33}\tau(s,x_s,\xi)\langle h^\xi,h^\xi_k\rangle+D_{32}\tau(s,x_s,\xi)\langle h^\xi,p^k_s\rangle
+A(s,q^{k,h}_s,0)\nl
&=& A(s,q^{k,h}_s,0)+g^{k,h}_{3}(s).
\end{eqnarray*}
Let $L_5=L_5(\alpha,M_1,M_3)$ be defined by (A2) (vi). Then
we have by \eref{e25}, \eref{e45} and \eref{e344}
\begin{eqnarray*}
 \int_0^\alpha |g^{k,h}_{3}(s)|\,ds
&\leq& \alpha L_5c_{1,k}N_1| h|_\Gamma L|h_k|_\Gamma+\alpha L_5N_1|h|_\Gamma\max_{s\in[0,\alpha]}|p^k_s|_C
+\alpha L_5c_{1,k}N_1| h|_\Gamma|h_k|_\Gamma\nl
&&+\alpha L_5| h|_\Gamma\max_{s\in[0,\alpha]}|p^k_s|_C
+\int_0^\alpha |\omega_{D_2\tau}(s,x_s,\xi,x^k_s,\xi+h^\xi_k,z^{k,h}_s)|\,ds\nl
&&+\int_0^\alpha |\omega_{D_3\tau}(s,x_s,\xi,x^k_s,\xi+h^\xi_k,h^\xi)|\,ds.
\end{eqnarray*}
Hence  $c_{1,k}\to0$ as $k\to\infty$, \eref{e335}, \eref{e356} and \eref{e357} imply \eref{e337}.

Relation
$$
E(s,z^{k,h}_s,h^\xi)-E(s,z^{h}_s,h^\xi)-E(s,w^{h,{h_k}}_s,0)=E(s,z^{k,h}_s-z^{h}_s-w^{h,{h_k}}_s,0)
$$
and the definition of $E,E^k$ and $H$ give 
\begin{eqnarray*}
 \balra{0.5cm}{ E^k(s,z^{k,h}_s,h^\xi)-E(s,z^{h}_s,h^\xi)-H(s)\langle(z^{h}_s,h^\xi),(z^{h_k}_s,h^\xi_k)\rangle
-E(s,w^{h,{h_k}}_s,0)}\nl
&=& E^k(s,z^{k,h}_s,h^\xi)-E(s,z^{k,h}_s,h^\xi)-H(s)\langle(z^{h}_s,h^\xi),(z^{h_k}_s,h^\xi_k)\rangle
+E(s,q^{k,h}_s,0)\nl
&=& -\dot x^k(u^k(s))A^k(s,z^{k,h}_s,h^\xi)+ \dot x(u(s))A(s,z^{k,h}_s,h^\xi)+z^{k,h}(u^k(s))-z^{k,h}(u(s))\nl
&&+A(s,z^{h}_s,h^\xi)F(s,z^{h_k}_s,h^\xi_k)+\dot x(u(s))
G(s)\langle(z^{h}_s,h^\xi),(z^{h_k}_s,h^\xi_k)\rangle\nl
&&+\dot z^{h}(u(s))A(s,z^{h_k}_s,h^\xi_k)-E(s,q^{k,h}_s,0)\nl
&=& -[\dot x^k(u^k(s))-\dot x(u(s))][A^k(s,z^{k,h}_s,h^\xi)-A(s,z^{k,h}_s,h^\xi)]\nl
&&-[\dot x^k(u^k(s))-\dot x(u(s))-F(s,z^{h_k}_s,h^\xi_k)]A(s,z^{k,h}_s,h^\xi)\nl
&&-\dot x(u(s))\Bigl[A^k(s,z^{k,h}_s,h^\xi)-A(s,z^{k,h}_s,h^\xi)-G(s)\langle(z^{h}_s,h^\xi),(z^{h_k}_s,h^\xi_k)\rangle\Bigr]\nl
&&+z^{k,h}(u^k(s))-z^{h}(u^k(s))-[z^{k,h}(u(s))-z^{h}(u(s))]\nl
&&+z^{h}(u^k(s))-z^{h}(u(s))-\dot z^{h}(u(s))(u^k(s)-u(s))\nl
&&+\dot z^{h}(u(s))\Bigl(u^k(s)-u(s)+A(s,z^{h_k}_s,h^\xi_k)\Bigl)+E(s,q^{k,h}_s,0),
\end{eqnarray*}
which, together with \eref{e305} and \eref{e336}, yields \eref{e338}.

To prove \eref{e339} first note that by \eref{e25}, \eref{e88} and  \eref{e382b}
\begin{eqnarray}
|\dot x^k(u^k(s))-\dot x(u(s))|
&\leq& |\dot x^k(u^k(s))-\dot x(u^k(s))|+|\dot x(u^k(s))-\dot x(u(s))|\nl
&\leq&  L|h_k|_\Gamma+ K_4K_0|h_k|_\Gamma.\label{e409}
\end{eqnarray}
Hence \eref{e371} and \eref{e409} give
$$
\lim_{k\to\infty}\sup_{h\neq0\atop h\in\Gamma} \frac{1}{|h|_\Gamma|h_k|_\Gamma}\int_0^\alpha|\dot x^k(u^k(s))-\dot x(u(s))||A^k(s,z^{k,h}_s,h^\xi)-A(s,z^{k,h}_s,h^\xi)|\,ds
=0.
$$
Relations\eref{e37},  \eref{e384}, \eref{e306} and \eref{e337} imply for $h_k\in\Gamma_2$ for $k\in\setN$
$$
\lim_{k\to\infty}\sup_{h\neq0\atop h\in\Gamma} \frac{1}{|h|_\Gamma|h_k|_{\Gamma_2}}\int_0^\alpha
|g^{k}_{2}(s)A(s,z^{k,h}_s,h^\xi)|\,ds
\leq \lim_{k\to\infty}\frac{K_6}{|h_k|_{\Gamma_2}}\int_0^\alpha |g^{k}_{2}(s)|\,ds
=0
$$
and
$$
\lim_{k\to\infty}\sup_{h\neq0\atop h\in\Gamma} \frac{1}{|h|_\Gamma|h_k|_\Gamma}\int_0^\alpha
|\dot x(u(s))g^{k,h}_{3}(s)|\,ds
\leq \lim_{k\to\infty}\frac{N}{|h|_\Gamma|h_k|_\Gamma}\int_0^\alpha |g^{k,h}_{3}(s)|\,ds
=0.
$$
The above limits and \eref{e372}, \eref{e376}, $|\dot z^{h}(u(s))|\leq N_2|h|_{\Gamma_2}$ and
\eref{e374}
yield \eref{e339}.

\end{proof}
\medskip

\begin{lemma}\label{l16}
Assume (A2) (i)--(vii), (H) and $\gamma\in\calP$.
Then there exist $K_{11}=K_{11}(\gamma)\geq0$ and a nonnegative sequence $c_{3,k}=c_{3,k}(\gamma)$   satisfying 
$c_{3,k}\to0$   as $k\to\infty$
such that
\begin{eqnarray}
|F(s,z^{h}_s,h^\xi)|&\leq& K_{11}|h|_\Gamma,\qquad  
\mbox{a.e. } s\in[0,\alpha],\ h\in\Gamma,\qquad\label{e398}\\
|E^k(s,z^{k,h}_s,h^\xi)-E(s,z^{h}_s,h^\xi)|
&\leq& c_{3,k}|h|_\Gamma,\qquad  
\mbox{a.e. } s\in[0,\alpha],\ k\in\setN,\label{e347}
\end{eqnarray} 
and, if in addition, (A2) (viii) holds, there exists a nonnegative sequence  $c_{4,k}=c_{4,k}(\gamma)$  satisfying 
 $c_{4,k}\to0$  as $k\to\infty$
such that
\begin{equation}\label{e399}
\int_0^\alpha |F^k(s,z^{k,h}_s,h^\xi)-F(s,z^{h}_s,h^\xi)|\,ds\ \leq\ c_{4,k}|h|_{\Gamma_2},\qquad  
\mbox{a.e. } s\in[0,\alpha],\ k\in\setN,\ h\in\Gamma_2.
\end{equation} 
\end{lemma}
\begin{proof}
The definition of $F$, \eref{e382} and \eref{e384} imply immediately \eref{e398} with $K_{11}:=K_4K_6+1$.

Relations \eref{e37}, \eref{e25}, \eref{e88}, \eref{e45}, \eref{e94}, \eref{e344}, 
 \eref{e384}, \eref{e363}, \eref{e409} and (H2) (ii) yield for a.e.\ $s\in[0,\alpha]$
\begin{eqnarray*}
 \balra{1.cm}{|E^k(s,z^{k,h}_s,h^\xi)-E(s,z^{h}_s,h^\xi)|}\nl
&\leq&|\dot x^k(u^k(s))-\dot x(u(s))||A^k(s,z^{k,h}_s,h^\xi)|\nl
&&+|\dot x(u(s))|\Bigl|A^k(s,z^{k,h}_s,h^\xi)-A(s,z^{h}_s,h^\xi)\Bigr|+|z^{k,h}(u^k(s))-z^{h}(u^k(s))|\nl
&& +|z^{h}(u^k(s))-z^{h}(u(s))|\nl
&\leq& (L+K_4K_0)|h_k|_\Gamma K_6|h|_\Gamma
+Nc_{2,k}|h|_\Gamma+c_{1,k}N_1|h|_\Gamma+N_2|h|_\Gamma K_0|h_k|_\Gamma,
\end{eqnarray*}
which proves \eref{e347}.

\begin{eqnarray*}
 \balra{1.cm}{|F^k(s,z^{h}_s,h^\xi)-F(s,z^{h}_s,h^\xi)|}\nl
&\leq&\Bigl(|\ddot x^k(u^k(s))-\ddot x(u^k(s))|+|\ddot x(u^k(s))-\ddot x(u(s))|\Bigr)|A^k(s,z^h_s,h^\xi)|\nl
&&+|\ddot x(u(s))|\Bigl|A^k(s,z^h_s,h^\xi)-A^k(s,z^h_s,h^\xi)\Bigr|+|\dot z^{h}(u^k(s))-\dot z^{h}(u(s))|.
\end{eqnarray*}
For $t\in(0,\alpha]$ we have by (A2) (viii) that
\begin{eqnarray*}
|\ddot x^k(t)-\ddot x(t)|
&=& \Bigl|\frac d{dt} f(t,x^k_t,x^k(u^k(t)),\theta+h^\theta_k)-\frac d{dt} f(t,x_t,x(u(t)),\theta)\Bigr|\nl
&\leq& L_6 (|x^k_t-x_t|_C+|h^\theta_k|_\Theta+|h^\xi_k|_\Xi)\nl
&\leq& L_6(L+1)|h_k|_\Gamma.
\end{eqnarray*}
For $t\in[-r,0)$ and $h\in\Gamma_2$ we get
$$
|\ddot x^k(t)-\ddot x(t)|=|\ddot h^\phi_k(t)|\leq |h_k|_{\Gamma_2}.
$$
Using that $\ddot x\in\Lw([-r,\alpha],\setR^n)$, similarly to \eref{e400} we can argue that
$$
\lim_{k\to\infty}\int_0^\alpha |\ddot x(u^k(s))-\ddot x(u(s))|\,ds=0.
$$
Then the above relations, $|\ddot x(u(s))|\leq K_4$ for a.e.\ $s\in[0,\alpha]$, \eref{e384}, \eref{e379} and \eref{e371} yield \eref{e399}.
\end{proof}
\medskip

For a.e.\ $s\in[0,\alpha]$, $h,y\in\Gamma$  and $k\in\setN$ we introduce
the bilinear operators by
\begin{eqnarray*}
G^k(s)\langle (h^\phi,h^\xi),(y^\phi,y^\xi)\rangle
&:=&
D_{22}\tau(s,x^k_s,\xi+h^\xi_k)\langle h^\phi, y^\phi\rangle+D_{23}\tau(s,x^k_s,\xi+h^\xi_k)\langle h^\phi, y^\xi\rangle \nl
&&+D_{32}\tau(s,x^k_s,\xi+h^\xi_k)\langle h^\xi, y^\phi\rangle+D_{33}\tau(s,x^k_s,\xi+h^\xi_k)\langle h^\xi, y^\xi\rangle,\nl
H^k(s)\langle(h^\phi,h^\xi),(y^\phi,y^\xi)\rangle
&:=&-A^k(s,h^\phi,h^\xi)F^k(s,y^\phi,y^\xi)\nl
&&-\dot x^k(u^k(s))G^k(s)\langle(h^\phi,h^\xi),(y^\phi,y^\xi)\rangle\nl
&&-\dot h^\phi(-\tau(s,x^k_s,\xi+h^\xi_k))A^k(s,y^\phi,y^\xi),
\end{eqnarray*}
and 
\begin{eqnarray*}
B^k(s)\langle h,y\rangle\!\!
&:=& D_{22}f(\vc v^k(s))\langle h^\phi,y^\phi\rangle 
+D_{23}f(\vc v^k(s))\langle h^\phi,E^k(s,y^\phi,y^\xi)\rangle \nl
&&+D_{24}f(\vc v^k(s))\langle h^\phi,y^\theta\rangle 
+D_{32}f(\vc v^k(s))\langle E^k(s,h^\phi,h^\xi),y^\phi\rangle\nl 
&&+D_{33}f(\vc v^k(s))\langle E^k(s,h^\phi,h^\xi),E^k(s,y^\phi,y^\xi)\rangle \nl
&&+D_{34}f(\vc v^k(s))\langle E^k(s,h^\phi,h^\xi),y^\theta\rangle +D_{42}f(\vc v^k(s))\langle h^\theta,y^\phi\rangle\nl
&& 
+D_{43}f(\vc v^k(s))\langle h^\theta,E^k(s,y^\phi,y^\xi)\rangle 
+D_{44}f(\vc v^k(s))\langle h^\theta,y^\theta\rangle\nl
&&+D_3f(\vc v^k(s))H^k(s)\langle(h^\phi,h^\xi),(y^\phi,y^\xi)\rangle.
\end{eqnarray*}
\medskip

\begin{lemma}\label{l31}
 Assume (A1) (i)--(vi), (A2) (i)--(vii). Then for every $\gamma\in P_2$ 
there exists $K_{12}=K_{12}(\gamma)\geq0$ such that 
\begin{equation}
 |B(s)\langle (z^h_s,h^\theta,h^\xi),(z^y_s,y^\theta,y^\xi)\rangle|
\leq K_{12} |h|_\Gamma|y|_\Gamma,\quad \mbox{a.e. } s\in[-r,\alpha],\ h,y\in\Gamma,\ \gamma\in P_2.\ \label{e393}
\end{equation}
If in addition (A2) (viii) holds, then
 for every $\gamma\in P_2\cap\calP$ there exists a nonnegative sequence $c_{5,k}=c_{5,k}(\gamma)$ 
such that $c_{5,k}\to0$ as $k\to\infty$, and
\begin{equation}
 \int_0^\alpha \Bigl|B^k(s)\langle (z^h_s,h^\theta,h^\xi),(z^y_s,y^\theta,y^\xi)\rangle
-B(s)\langle (z^h_s,h^\theta,h^\xi),(z^y_s,y^\theta,y^\xi)\rangle\Bigr|\,ds
\leq c_{5,k} |h|_{\Gamma_2}|y|_{\Gamma_2},\quad \label{e401}
\end{equation}
for  $h,y\in\Gamma_2$.
\end{lemma}
\begin{proof}
 Let  $L_3=L_3(\alpha,M_1,M_2,M_3)$ and $L_5=L_5(\alpha,M_1,M_4)$ be the Lipschitz constants
from (A1) (v) and (A2) (vi), respectively. Then the definition of $G$, (A2) (vi) and \eref{e45}
yield
\begin{equation}\label{e402}
|G(s)\langle (z^h_s,h^\xi),(z^y_s,y^\xi)\rangle |\leq 4L_5N_1^2|h|_\Gamma|y|_\Gamma, \qquad h,y\in\Gamma,
\quad s\in[0,\alpha].
\end{equation}
Then definition of $H$, \eref{e37}, \eref{e45}, \eref{e382}, \eref{e384}, \eref{e398} and  \eref{e402} imply 
\begin{equation}\label{e403}
|H(s)\langle (z^h_s,h^\xi),(z^y_s,y^\xi)\rangle |\leq K_{13}|h|_\Gamma |y|_\Gamma, \quad h,y\in\Gamma,\quad
\mbox{a.e. } s\in[0,\alpha]
\end{equation}
with $K_{13}=K_{13}(\gamma):=K_6(K_4K_6+1)+N4L_5N_1^2+ K_6$.
Therefore we have by the definition of $B$, \eref{e389a} and  \eref{e403}  
\begin{equation*}
|B(s)\langle h,y\rangle|
\leq L_3(4+ 4K_8+K_8^2+K_{13})|h|_\Gamma|y|_\Gamma,\qquad \mbox{a.e. } s\in[0,\alpha],
\end{equation*}
which, together with \eref{e397}, yields \eref{e393}.

Define the set $M_4^*:=\{\xi\}\cup\{ h^\xi_k\st k\in\setN\}$. It is easy to show that
$M_4^*\subset M_4$ is a compact subset of $\Xi$. 
Define
\begin{eqnarray*}
\Omega_{2,\tau}(\epsilon)
&:=&\max_{i,j=2,3} \sup \Bigl\{|D_{ij}\tau(s,\psi,\eta)-D_{ij}\tau(s,\bar\psi,\bar\eta)|_{\calL^2(X_i\times X_j,\,\setR)}\st\nl
&&\qquad\qquad s\in[0,\alpha],\ \psi,\bar\psi\in M_1, \eta,\bar\eta\in M_4^*,\ |\psi- \bar\psi|_C+|\eta-\bar\eta|_\Xi\leq\epsilon\Bigr\},
\end{eqnarray*}
where $X_2:=C$ and $X_3:=\Xi$. Assumption (A2) (vii) and the compactness of $[0,\alpha]\times M_1\times M_4^*$
yields that $\Omega_{2,\tau}(\epsilon)\to0$ as $\epsilon\to0+$.
Then \eref{e45} and \eref{e396} give
{\arraycolsep=0.5\arraycolsep
\begin{eqnarray}
|[G^k(s)-G(s)]\langle (z^h_s,h^\xi),(z^y_s,y^\xi)\rangle|
&\leq&
|[D_{22}\tau(s,x^k_s,\xi+h^\xi_k)-D_{22}\tau(s,x_s,\xi)]\langle z^h_s, z^y_s\rangle|\nl
&&+|[D_{23}\tau(s,x^k_s,\xi+h^\xi_k)-D_{23}\tau(s,x^k_s,\xi+h^\xi_k)]\langle z^h_s, y^\xi\rangle |\nl
&&+|[D_{32}\tau(s,x^k_s,\xi+h^\xi_k)-D_{32}\tau(s,x^k_s,\xi+h^\xi_k)]\langle h^\xi, z^y_s\rangle|\nl
&&+|[D_{33}\tau(s,x^k_s,\xi+h^\xi_k)-D_{33}\tau(s,x^k_s,\xi+h^\xi_k)]\langle h^\xi, y^\xi\rangle|\nl
&\leq& \Omega_{2,\tau}\Bigr((L+1)|h_k|_\Gamma\Bigr)(N_1+1)^2|h|_\Gamma|y|_\Gamma,\qquad s\in[0,\alpha].\nl
\label{e404}
\end{eqnarray}
}\medskip

Relations \eref{e37}, \eref{e25}, \eref{e88}, \eref{e45}, \eref{e94},  \eref{e384}, 
\eref{e371}, \eref{e409}, \eref{e398},  \eref{e399},  \eref{e402} and \eref{e404}  imply
\begin{eqnarray}
\balra{1cm}{
\int_0^\alpha|[H^k(s)-H(s)]\langle (z^h_s,h^\xi),(z^y_s,y^\xi)\rangle|\,ds}\nl
&\leq& \int_0^\alpha\Bigl(
|[A^k(s,z^h_s,h^\xi)-A(s,z^h_s,h^\xi)]F(s,z^y_s,y^\xi)|\hspace{5cm}\nl
&&\qquad+|A^k(s,z^h_s,h^\xi)[F^k(s,z^y_s,y^\xi)-F(s,z^y_s,y^\xi)]|\nl
&&\qquad+|[\dot x^k(u^k(s))-\dot x(u(s))]G^k(s)\langle(z^h_s,h^\xi),(z^y_s,y^\xi)\rangle|\nl
&&\qquad+|\dot x(u(s))[G^k(s)-G(s)]\langle(z^h_s,h^\xi),(z^y_s,y^\xi)\rangle|\nl
&&\qquad+|[\dot z^h(u^k(s))-\dot z^h(u(s))]A^k(s,z^y_s,y^\xi)|\nl
&&\qquad+|\dot z^h(u(s))[A^k(s,z^y_s,y^\xi)-A(s,z^y_s,y^\xi)]|\Bigr)\,ds\nl
&\leq& \alpha K_{10}|h|_\Gamma|h_k|_\Gamma K_{11}|y|_\Gamma
+K_6|h|_\Gamma c_{4,k}|y|_{\Gamma_2}
+(L+K_4K_0)|h_k|_\Gamma4L_5N_1^2|h|_\Gamma|y|_\Gamma\nl
&&+N\Omega_{2,\tau}\Bigr((L+1)|h_k|_\Gamma\Bigr)(N_1+1)^2|h|_\Gamma|y|_\Gamma\nl
&&+ \int_0^\alpha|\dot z^h(u^k(s))-\dot z^h(u(s))|\,ds\,K_6|y|_\Gamma
+\alpha N_2|h|_{\Gamma_2} K_{10}|h|_\Gamma|y|_\Gamma\nl
&\leq & c_{6,k}|h|_{\Gamma_2}|y|_{\Gamma_2}\label{e407}
\end{eqnarray}
with some appropriate sequence $c_{6,k}=c_{6,k}(\gamma)$ satisfying $c_{6,k}\to0$ as $k\to\infty$, 
where in the last estimate we used \eref{e379}.

Simple manipulations give
{
\begin{eqnarray}
 \balra{1cm}{|[B^k(s)-B(s)]\langle (z^h_s,h^\theta,h^\xi),(z^y_s,y^\theta,y^\xi)\rangle|}\nl
&\leq& |[D_{22}f(\vc v^k(s))-D_{22}f(\vc v(s))]\langle z^h_s,z^y_s\rangle| \nl
&&+|[D_{23}f(\vc v^k(s))-D_{23}f(\vc v(s))]\langle z^h_s,E^k(s,z^y_s,y^\xi)\rangle| \nl
&&+|D_{23}f(\vc v(s))\langle z^h_s,E^k(s,z^y_s,y^\xi)-E(s,z^y_s,y^\xi)\rangle| \nl
&&+|[D_{24}f(\vc v^k(s))-D_{24}f(\vc v(s))]\langle z^h_s,y^\theta\rangle |\nl
&&+|[D_{32}f(\vc v^k(s))-D_{32}f(\vc v(s))]\langle E^k(s,z^h_s,h^\xi),z^y_s\rangle|\nl 
&&+|D_{32}f(\vc v(s))\langle E^k(s,z^h_s,h^\xi)-E(s,z^h_s,h^\xi),z^y_s\rangle|\nl 
&&+|[D_{33}f(\vc v^k(s))-D_{33}f(\vc v(s))]\langle E^k(s,z^h_s,h^\xi),E^k(s,z^y_s,y^\xi)\rangle| \nl
&&+|D_{33}f(\vc v(s))\langle E^k(s,z^h_s,h^\xi)-E(s,z^h_s,h^\xi),E^k(s,z^y_s,y^\xi)\rangle| \nl
&&+|D_{33}f(\vc v(s))\langle E(s,z^h_s,h^\xi),E^k(s,z^y_s,y^\xi)-E(s,z^y_s,y^\xi)\rangle| \nl
&&+|[D_{34}f(\vc v^k(s))-D_{34}f(\vc v(s))]\langle E^k(s,z^h_s,h^\xi),y^\theta\rangle|\nl 
&&+|D_{34}f(\vc v(s))\langle E^k(s,z^h_s,h^\xi)-E(s,z^h_s,h^\xi),y^\theta\rangle|\nl 
&&+|[D_{42}f(\vc v^k(s))-D_{42}f(\vc v(s))]\langle h^\theta,z^y_s\rangle|\nl
&& +|[D_{43}f(\vc v^k(s))-D_{43}f(\vc v(s))]\langle h^\theta,E^k(s,z^y_s,y^\xi)\rangle \nl
&& +|D_{43}f(\vc v(s))\langle h^\theta,E^k(s,z^y_s,y^\xi)-E(s,z^y_s,y^\xi)\rangle \nl
&&+|[D_{44}f(\vc v^k(s))-D_{44}f(\vc v(s))]\langle h^\theta,y^\theta\rangle|\nl
&&+|[D_3f(\vc v^k(s))-D_3f(\vc v(s))]H^k(s)\langle(z^h_s,h^\xi),(z^y_s,y^\xi)\rangle|\nl
&&+|D_3f(\vc v(s))[H^k(s)-H(s)]\langle(z^h_s,h^\xi),(z^y_s,y^\xi)\rangle|.\label{e406}
\end{eqnarray}
}%
Define the set $M_3^*:=\{\theta\}\cup\{h^\theta_k\st k\in\setN\}$. Clearly, 
$M_3^*\subset M_3$ is a compact subset of $\Theta$. 

Define
\begin{eqnarray*}
\Omega_{2,f}(\epsilon)
&:=&\max_{i,j=2,3,4} \sup \Bigl\{|D_{ij}f(s,\psi,v,\eta)-D_{ij} f(s,\bar\psi,\bar v,\bar\eta)|_{\calL^2(Y_i\times Y_j,\,\setR)}\st\nl
&&\qquad\qquad s\in[0,\alpha],\ \psi,\bar\psi\in M_1, v,\bar v\in M_2,\ \eta,\bar\eta\in M_3^*,\nl
&&\qquad\qquad |\psi- \bar\psi|_C+|v-\bar v|+|\eta-\bar\eta|_\Theta
\leq\epsilon\Bigr\},
\end{eqnarray*}
where $Y_2:=C$, $Y_3:=\setR^n$ and $Y_4:=\Theta$. Assumption (A1) (vi) and the compactness of 
$[0,\alpha]\times M_1\times M_2\times M_3^*$
yields that $\Omega_{2,f}(\epsilon)\to0$ as $\epsilon\to0+$.
Then combining \eref{e406} with \eref{e355},
$|D_{ij}f(\vc v^k(s))-D_{ij}f(\vc v(s))|_{\calL^2(Y_i\times Y_j,\,\setR^n)}\leq \Omega_{2,f}\Bigl(K_3|h_k|_\Gamma\Bigr)$
for $i,j=2,3,4$, $|D_if(\vc v^k(s))|_{\calL(Y_i,\setR^n)}\leq L_1$ for $i=2,3,4$, $s\in[0,\alpha]$ and $k\in\setN_0$, 
\eref{e45}, \eref{e389a}, \eref{e347}, \eref{e403}, \eref{e407} and \eref{e406}.
yields \eref{e401}

\end{proof}
\medskip

\begin{lemma}\label{l30}
 Assume (A1) (i)--(vi), (A2) (i)--(vii), $\gamma\in P_2$. Then there exists $N_5=N_5(\gamma)\geq0$ such that
the solution of the \IVP{e330}{e331} satisfies
\begin{equation}\label{e390}
 |w^{h,y}(t)|\leq N_5 |h|_\Gamma|y|_\Gamma,\qquad t\in[-r,\alpha],\quad h,y\in\Gamma.
\end{equation}
\end{lemma}
\begin{proof}
 It follows from \eref{e330} and \eref{e331} that
$$
w^{h,y}(t)=\int_0^t B(s)\langle(z^h_s,h^\theta,h^\xi),
(z^{y}_s,y^\theta,y^\xi)\rangle\,ds+\int_0^t L(s,x)(w^{h,y}_s,0,0)\,ds,\qquad t\in[0,\alpha].
$$
Therefore \eref{e368} and \eref{e393} yield
$$
|w^{h,y}(t)|\leq K_{12}|h|_\Gamma|y|_\Gamma+L_1N_0\int_0^t |w^{h,y}_s|_C\,ds,\qquad t\in[0,\alpha].
$$
Since $w^{h,y}(t)=0$ for $t\in[-r,0]$, Lemma~\ref{l0} gives
\eref{e390} with $N_5:=K_{12}e^{L_1N_0\alpha}$.
\end{proof}
\bigskip

\begin{lemma}\label{l32}
 Assume (A1) (i)--(vi), (A2) (i)--(viii), (H). For $h,y\in\Gamma_2$ and $k\in\setN$ let $w^{h,y}(t):=w(t,\gamma,h,y)$ and
$w^{k,h,y}(t):=w(t,\gamma+h_k,h,y)$ be the solutions
of the \IVP{e330}{e331}. Then there exists a nonnegative sequence $c_{7,k}=c_{7,k}(\gamma)$ such that
\begin{equation}\label{e394}
 |w^{k,h,y}_t-w^{h,y}_t|_C\leq c_{7,k}|h|_{\Gamma_2}|y|_{\Gamma_2}, \qquad t\in[0,\alpha],\quad h,y\in\Gamma_2.
\end{equation}
\end{lemma}
\begin{proof}
It follows from \eref{e368}, \eref{e395}, \eref{e53}, \eref{e330}, \eref{e393}, \eref{e409} and \eref{e390}
{\arraycolsep=0.7\arraycolsep
\begin{eqnarray*}
\balra{0cm}{|w^{k,h,y}(t)-w^{h,y}(t)|}\nl
&\leq& \int_0^t \Bigl(|[L(s,x^k)-L(s,x)](w^{k,h,y}_s,0,0)|+|L(s,x)(w^{k,h,y}_s-w^{h,y}_s,0,0)|\Bigr)ds\nl
&&\!+\int_0^t \Bigl(|B^k(s)\langle (z^{k,h}_s,h^\theta,h^\xi),(z^{k,y}_s-z^y_s,0,0)\rangle|
+|B^k(s)\langle (z^{k,h}_s-z^h_s,0,0),(z^{y}_s,y^\theta,y^\xi)\rangle|\nl
&&\qquad+ |B^k(s)\langle (z^{h}_s,h^\theta,h^\xi),(z^{y}_s,y^\theta,y^\xi)\rangle
-B(s)\langle (z^{h}_s,h^\theta,h^\xi),(z^{y}_s,y^\theta,y^\xi)\rangle|\Bigr)ds
\end{eqnarray*}}%
\begin{eqnarray*}
&\leq& \alpha c_{0,k}N_5|h|_\Gamma|y|_\Gamma+L_1L_2\int_0^\alpha |\dot x(u^k(s))-\dot x(u(s))|\,ds N_5|h|_\Gamma|y|_\Gamma\nl
&&+L_1N_0\int_0^t|w^{k,h,y}_s-w^{h,y}_s|_C\,ds+2\alpha K_{12}c_{1,k}N_1^2|h|_\Gamma |y|_\Gamma+\alpha c_{5,k}|h|_\Gamma|y|_\Gamma\qquad\qquad\qquad\nl
&\leq& c_{8,k}|h|_\Gamma|y|_\Gamma+L_1N_0\int_0^t|w^{k,h,y}_s-w^{h,y}_s|_C\,ds,
\end{eqnarray*}%
where $c_{8,k}=c_{8,k}(\gamma):=\alpha c_{0,k}N_5+L_1L_2(L+K_4K_0)N_5|h_k|_\Gamma +2\alpha K_{12}c_{1,k}N_1^2
+\alpha c_{5,k}$. Then Lemma~\ref{l0} is applicable, since $|w^{k,h,y}_0-w^{h,y}_0|_C=0$, and it yields
\eref{e394} with $c_{7,k}:=c_{8,k}e^{L_1N_0\alpha}$.
\end{proof}
\medskip

We define 
\begin{eqnarray*}
\omega_{D_2f}(\vc v(s),\vc v^k(s),\psi)\!\!%
&:=&D_2f(\vc v^k(s))\psi-D_2f(\vc v(s))\psi-D_{22}f(\vc v(s))\langle \psi,x^k_s-x_s\rangle\nl                                   
&&-D_{23}f(\vc v(s))\langle \psi,x^k(u^k(s))-x(u(s))\rangle 
-D_{24}f(\vc v(s))\langle \psi,h^\theta_k\rangle, \nl
\omega_{D_3f}(\vc v(s),\vc v^k(s),v)\!\!%
&:=& D_3f(\vc v^k(s))v-D_3f(\vc v(s))v-D_{32}f(\vc v(s))\langle v,x^k_s-x_s\rangle\nl
&& -D_{33}f(\vc v(t))\langle v,x^k(u^k(s))-x(u(s))\rangle 
-D_{34}f(\vc v(s))\langle v,h^\theta_k\rangle, \nl
\omega_{D_4f}(\vc v(s),\vc v^k(s),\eta)\!\!%
&:=& D_4f(\vc v^k(s))\eta-D_4f(\vc v(s))\eta-D_{42}f(\vc v(s))\langle \eta,x^k_s-x_s\rangle\nl
&&-D_{43}f(\vc v(s))\langle \eta,x^k(u^k(s))-x(u(s))\rangle 
-D_{44}f(\vc v(s))\langle \eta,h^\theta_k\rangle 
\end{eqnarray*}
for $s\in[0,\alpha]$,  $\psi\in C$, $v\in \setR^n$ and $\eta\in\Theta$.

The proof of the following lemma is similar to that of Lemma~\ref{l17}.

\begin{lemma}\label{l18}
 Assume (A1) (i)--(vi) and (H).
Then
\begin{equation}\label{e358}
\lim_{k\to\infty}\sup_{h\neq0\atop h\in\Gamma}\frac1{|h|_\Gamma|h_k|_\Gamma}
\int_0^\alpha |\omega_{D_2 f}(s,x_s,x(u(s)),\theta,x^k_s,x^k(u^k(s)),\theta+h^\theta_k,z^{k,h}_s)|\,ds=0,
\end{equation}
\begin{equation}\label{e359}
\lim_{k\to\infty}\sup_{h\neq0\atop h\in\Gamma}\frac1{|h|_\Gamma|h_k|_\Gamma}
\int_0^\alpha |\omega_{D_3 f}(s,x_s,x(u(s)),\theta,x^k_s,x^k(u^k(s)),\theta+h^\theta_k,E^k(s,z^{k,h}_s,h^\xi))|\,ds=0,
\end{equation}
and
\begin{equation}\label{e360}
\lim_{k\to\infty}\sup_{h\neq0\atop h\in\Gamma}\frac1{|h|_\Gamma|h_k|_\Gamma}
\int_0^\alpha |\omega_{D_4 f}(s,x_s,x(u(s)),\theta,x^k_s,x^k(u^k(s)),\theta+h^\theta_k,h^\theta_k)|\,ds=0.
\end{equation}
\end{lemma}
\medskip

\begin{lemma}\label{l15}
Assume (A1) (i)--(vi), (A2) (i)--(vii), (H), $\gamma\in\calP$ and $h_k\in\Gamma_2$ for $k\in\setN$.
Then 
\begin{eqnarray}
\balra{1.cm}{L(s,x^k)(z^{k,h}_s,h^\theta,h^\xi)- L(s,x)(z^{h}_s+w^{h,{h_k}}_s,h^\theta,h^\xi)
- B(s)\Bigl\langle (z^{h}_s,h^\theta,h^\xi),(z^{h_k}_s,h^\theta_k,h^\xi_k)\Bigr\rangle }\nl
&=&  L(s,x)(q^{k,h}_s,0,0) +g^{k,h}_{5}(s),\qquad \mbox{a.e. } s\in[0,\alpha],\qquad\qquad\qquad\qquad\qquad\quad
\label{e340}
\end{eqnarray}
\vskip -0.5cm\noindent
where
\begin{eqnarray*}
g^{k,h}_{5}(s)
&:=&D_{22}f(\vc v(s))\langle z^{k,h}_s-z^{h}_s,x^k_s-x_s\rangle+D_{22}f(\vc v(s))\langle z^{h}_s,p^k_s\rangle\nl
&&+D_{23}f(\vc v(s))\langle z^{k,h}_s-z^{h}_s,x^k(u^k(s))-x(u(s))\rangle
-D_{23}f(\vc v(s))\langle z^{h}_s,g^{k}_{1}(s)\rangle\nl
&&+D_{24}f(\vc v(s))\langle z^{k,h}_s-z^{h}_s,h^\theta_k\rangle
+ D_{32}f(\vc v(s))\langle E(s,z^{h}_s,h^\xi),p^k_s\rangle\nl
&&+ D_{32}f(\vc v(s))\langle E^k(s,z^{k,h}_s,h^\xi)-E(s,z^{h}_s,h^\xi),x^k_s-x_s\rangle\nl
&& +D_{33}f(\vc v(s))\langle E^k(s,z^{k,h}_s,h^\xi)-E(s,z^{h}_s,h^\xi),x^k(u^k(s))-x(u(s))\rangle\nl
&& +D_{33}f(\vc v(s))\langle E(s,z^{h}_s,h^\xi), g^{k}_{1}(s)\rangle+D_3f(\vc v(s))g^{k,h}_{4}(s)\nl
&&+D_{34}f(\vc v(s))\langle E^k(s,z^{k,h}_s,h^\xi)-E(s,z^{h}_s,h^\xi), h^\theta_k\rangle\nl
&&+D_{42}f(\vc v(s))\langle h^\theta,p^k_s\rangle 
+D_{43}f(\vc v(s))\langle h^\theta,g^{k}_{1}(s)\rangle+\omega_{D_2f}(\vc v(s),\vc v^k(s),z^{k,h}_s)\nl
&&+\omega_{D_3f}(\vc v(s),\vc v^k(s),E^k(s,z^{k,h}_s,h^\xi))
 +\omega_{D_4f}(\vc v(s),\vc v^k(s),h^\theta_k)
\end{eqnarray*}
satisfies
\begin{equation}\label{e341}
\lim_{k\to\infty} \sup_{h\neq0\atop h\in\Gamma_2}\frac1{|h|_{\Gamma_2}|h_k|_{\Gamma_2}} \int_0^\alpha |g^{k,h}_{5}(s)|\,ds=0.
\end{equation}
\end{lemma}
\begin{proof}
Straightforward manipulations yield for a.e.\ $s\in[0,\alpha]$
\begin{eqnarray*}
\balra{0.cm}{L(s,x^k)(z^{k,h}_s,h^\theta,h^\xi)- L(s,x)(z^{h}_s+w^{h,{h_k}}_s,h^\theta,h^\xi)
- B(s)\Bigl\langle (z^{h}_s,h^\theta,h^\xi),(z^{h_k}_s,h^\theta_k,h^\xi_k)\Bigr\rangle }\nl
&=& D_2f(\vc v^k(s))z^{k,h}_s -D_2f(\vc v(s))z^{k,h}_s +D_2f(\vc v(s))(z^{k,h}_s-z^{h}_s-w^{h,{h_k}}_s)\nl
&& +D_3f(\vc v^k(s))E^k(s,z^{k,h}_s,h^\xi) - D_3f(\vc v(s))E^k(s,z^{k,h}_s,h^\xi)\nl
&&+D_3f(\vc v(s))\Bigl(E^k(s,z^{k,h}_s,h^\xi)-E(s,z^{h}_s,h^\xi)\Bigr)+D_4f(\vc v^k(s))h^\theta -D_4f(\vc v(s))h^\theta\nl
&& -D_3f(\vc v(s))E(s,w^{h,{h_k}}_s,0)
-B(s)\Bigl\langle (z^{h}_s,h^\theta,h^\xi),(z^{h_k}_s,h^\theta_k,h^\xi_k)\Bigr\rangle \nl
\end{eqnarray*}
\begin{eqnarray*}
&=& D_2f(\vc v^k(s))z^{k,h}_s-D_2f(\vc v(s))z^{k,h}_s-D_{22}f(\vc v(s))\langle z^{k,h}_s,x^k_s-x_s\rangle \nl
&&-D_{23}f(\vc v(s))\langle z^{k,h}_s,x^k(u^k(s))-x(u(s))\rangle 
-D_{24}f(\vc v(s))\langle z^{k,h}_s,h^\theta_k\rangle \nl
&&+D_2f(\vc v(s))q^{k,h}_s+D_{22}f(\vc v(s))\langle z^{k,h}_s-z^{h}_s,x^k_s-x_s\rangle +D_{22}f(\vc v(s))\langle z^{h}_s,p^k_s\rangle\nl
&&+D_{23}f(\vc v(s))\langle z^{k,h}_s-z^{h}_s,x^k(u^k(s))-x(u(s))\rangle\nl
&&+D_{23}f(\vc v(s))\langle z^{h}_s,x^k(u^k(s))-x(u(s))-E(s,z^{h_k}_s,h^\xi_k)\rangle 
+D_{24}f(\vc v(s))\langle z^{k,h}_s-z^{h}_s,h^\theta_k\rangle \nl
&&+D_3f(\vc v^k(s))E^k(s,z^{k,h}_s,h^\xi) - D_3f(\vc v(s))E^k(s,z^{k,h}_s,h^\xi)\nl
&&-D_{32}f(\vc v(s))\langle E^k(s,z^{k,h}_s,h^\xi),x^k_s\!-\!x_s\rangle \nl
&&-D_{33}f(\vc v(s))\langle E^k(s,z^{k,h}_s,h^\xi),x^k(u^k(s))-x(u(s))\rangle 
 -D_{34}f(\vc v(s))\langle E^k(s,z^{k,h}_s,h^\xi),h^\theta_k\rangle\nl
&&+ D_{32}f(\vc v(s))\langle E^k(s,z^{k,h}_s,h^\xi)-E(s,z^{h}_s,h^\xi),x^k_s-x_s\rangle 
+ D_{32}f(\vc v(s))\langle E(s,z^{h}_s,h^\xi),p^k_s\rangle \nl
&& +D_{33}f(\vc v(s))\langle E^k(s,z^{k,h}_s,h^\xi)-E(s,z^{h}_s,h^\xi),x^k(u^k(s))-x(u(s))\rangle\nl
&&+D_{33}f(\vc v(s))\langle E(s,z^{h}_s,h^\xi),x^k(u^k(s))-x(u(s))-E(s,z^{h_k}_s,h^\xi_k)\rangle\nl
&&+D_{34}f(\vc v(s))\langle E^k(s,z^{k,h}_s,h^\xi)-E(s,z^{h}_s,h^\xi),h^\theta_k\rangle\nl
&&+D_3f(\vc v(s))[E^k(s,z^{k,h}_s,h^\xi)-E(s,z^{h}_s,h^\xi)-H(s)\langle(z^{h}_s,h^\xi),(z^{k,h}_s,h^\xi_k)\rangle\!-E(s,w^{h,h_k}_s,0)]\nl
&&+D_4(\vc v^k(s))h^\theta-D_4(\vc v(s))h^\theta-D_{42}f(\vc v(s))\langle h^\theta,x^k_s-x_s\rangle\nl
&&-D_{43}f(\vc v(s))\langle h^\theta,x^k(u^k(s))-x(u(s))\rangle -D_{44}f(\vc v(s))\langle h^\theta,h^\theta_k\rangle\nl
&&+D_{42}f(\vc v(s))\langle h^\theta,p^k_s\rangle
+D_{43}f(\vc v(s))\langle h^\theta,x^k(u^k(s))-x(u(s))-E(s,z^{h_k}_s,h^\xi_k)\rangle,
\end{eqnarray*}
which implies \eref{e340}, using \eref{e300} and \eref{e338}.
Let $L_3=L_3(\alpha,M_1,M_2,M_3)$ be defined by (A1) (iv).
Then (A1) (iv), \eref{e25}, \eref{e94}, \eref{e345}, \eref{e344},  
 \eref{e389a} and  \eref{e347} yield
\begin{eqnarray*}
 \balra{0.cm}{\int_0^\alpha |g^{k,h}_{5}(s)|\,ds}\nl
&\leq& \alpha L_3c_{1,k}N_1|h|_\Gamma L|h_k|_\Gamma
+\alpha L_3N_1|h|_\Gamma\max_{s\in[0,\alpha]}|p^k_s|_C
+\alpha L_3c_{1,k}N_1|h|_\Gamma K_2|h_k|_\Gamma\nl
&&+ L_3N_1|h|_\Gamma \int_0^\alpha|g_{1,k}(s)|\,ds+\alpha L_3c_{1,k}N_1|h|_\Gamma|h_k|_\Gamma
+\alpha L_3K_8|h|_\Gamma\max_{s\in[0,\alpha]}|p^k_s|_C\nl
&&+\alpha L_3 c_{3,k} |h|_\Gamma L|h_k|_\Gamma
+\alpha L_3c_{3,k} |h|_\Gamma K_2|h_k|_\Gamma\nl
&&+L_3 K_8|h|_\Gamma  \int_0^\alpha|g_{1,k}(s)|\,ds+L_1 \int_0^\alpha|g^{h}_{3,k}(s)|\,ds
+\alpha L_3c_{3,k} |h|_\Gamma|h_k|_\Gamma\nl
&&+L_3|h|_\Gamma\max_{s\in[0,\alpha]}|p^k_s|_C+L_3|h|_\Gamma \int_0^\alpha|g_{1,k}(s)|\,ds+\int_0^\alpha|\omega_{D_2f}(\vc v(s),\vc v^k(s),z^{k,h}_s)|\,ds\nl
&&+\int_0^\alpha|\omega_{D_3f}(\vc v(s),\vc v^k(s),E^k(s,z^{k,h}_s,h^\xi))|\,ds
+\int_0^\alpha|\omega_{D_4f}(\vc v(s),\vc v^k(s),h^\theta_k)|\,ds.
\end{eqnarray*}
Hence  $c_{1,k}\to0$, $c_{3,k}\to0$ as $k\to\infty$,
\eref{e335}, \eref{e304}, \eref{e337}, \eref{e358}, \eref{e359} and \eref{e360}   imply \eref{e341}.

\end{proof}
\medskip

Now we are ready to prove the main result of this section.
\medskip

\begin{theorem}\label{t3}
 Assume (A1) (i)--(vi), (A2) (i)--(vii). Then for $t\in[0,\alpha]$ the maps 
$$
\Gamma_2\supset (P_2\cap\Gamma_2)\to\setR^n,\quad \gamma\mapsto x(t,\gamma)
$$
and
$$
\Gamma_2\supset (P_2\cap\Gamma_2)\to C,\quad \gamma\mapsto x_t(\cdot,\gamma)
$$
are twice differentiable wrt $\gamma$ for every $\gamma\in P_2\cap \Gamma_2\cap \calP$, and 
$$
D_{22} x(t,\gamma)\langle h,y\rangle=w^{h,y}(t),\qquad h,y\in\Gamma_2,
$$
and
$$
D_{22} x_t(\cdot,\gamma)\langle h,y\rangle=w^{h,y}_t,\qquad h,y\in\Gamma_2,
$$
where $w^{h,y}$ is the solution of the \IVP{e330}{e331}.
Moreover, if in addition, (A2) (viii) holds, then the maps
$$
\setR\times\Gamma_2\supset \Bigl([0,\alpha]\times (P_2\cap\Gamma_2\cap\calP)\Bigr)
\to\calL^2(\Gamma_2\times\Gamma_2,\setR^n),\quad (t,\gamma)\mapsto D_{22}x(t,\gamma)
$$
and
$$
\setR\times\Gamma_2\supset \Bigl([0,\alpha]\times (P_2\cap\Gamma_2\cap\calP)\Bigr)
\to \calL^2(\Gamma_2\times\Gamma_2,C),\quad (t,\gamma)\mapsto D_{22}x_t(\cdot,\gamma)
$$
are continuous.
\end{theorem}
\begin{proof}
It follows from Theorem~\ref{t2} that $D_2 x(t,\gamma)\in\calL(\Gamma,\setR^n)$ exists 
for all $\gamma\in P_2$ and $t\in[0,\alpha]$. Since $|h|_\Gamma\leq |h|_{\Gamma_2}$ for all $h\in\Gamma_2$,
it follows that $D_2 x(t,\gamma)\Bigr|_{\Gamma_2}\in\calL(\Gamma_2,\setR^n)$, and $D_2 x(t,\gamma)\Bigr|_{\Gamma_2}$ is 
the derivtive of the map $\Gamma_2\supset (P_2\cap\Gamma_2)\to\setR^n$, $\gamma\to x(t,\gamma)$. 
For simplicity, the restiction of $D_2 x(t,\gamma)$ to $\Gamma_2$ will be denoted by $D_2 x(t,\gamma)$, as well.
Theorem~\ref{t2} yields that $D_2 x(t,\gamma)h=z(t,\gamma,h)$, where $z(t,\gamma,h)$
is the solution of the \IVP{e3}{e4} for $h\in\Gamma_2$.

Let $\gamma\in P_2\cap \Gamma_2\cap \calP$ be fixed, $h_k=(h^\phi_k,h^\theta_k,h^\xi_k)\in\Gamma_2$ ($k\in\setN$) be
a sequence such that $\gamma+h_k\in P_2$ for $k\in\setN$, $0\neq h=(h^\phi,h^\theta,h^\xi)\in\Gamma_2$.
Let $x(t):=x(t,\gamma)$ and $x^k(t):=x(t,\gamma+h_k)$ be the solutions of the \IVP{e1}{e2}, 
$z^h(t):=D_2x(t,\gamma)h$ and $z^{k,h}(t):=D_2x(t,\gamma+h_k)h$ be the solution of the \IVP{e3}{e4},
and $w^{h,h_k}(t)$ be the solution of the \IVP{e330}{e331} corresponding to parameters $h$ and $h_k$.
Then we have for $t\in[0,\alpha]$
\begin{eqnarray*}
 z^{k,h}(t) &=& h^\phi(0)+\int_0^t L(s,x^k)(z^{k,h}_s,h^\theta,h^\xi)\,ds,\nl
 z^{h}(t) &=& h^\phi(0)+\int_0^t L(s,x)(z^{h}_s,h^\theta,h^\xi)\,ds,\nl
w^{h,{h_k}}(t) &=& \int_0^t \Bigl(L(s,x)(w^{h,{h_k}}_s,0,0)
+B(s)\Bigl\langle (z^{h}_s,h^\theta,h^\xi),(z^{h_k}_s,h^\theta_k,h^\xi_k)\Bigr\rangle\Bigr)ds.
\end{eqnarray*}
Hence Lemma~\ref{l15} and the definition of $q^{k,h}$ give
\begin{eqnarray*}
q^{k,h}(t)
&=& \int_0^t \Bigl(L(s,x^k)(z^{k,h}_s,h^\theta,h^\xi)- L(s,x)(z^{h}_s+w^{h,{h_k}}_s,h^\theta,h^\xi)\nl
&&\quad- B(s)\Bigl\langle (z^{h}_s,h^\theta,h^\xi),(z^{h_k}_s,h^\theta_k,h^\xi_k)\Bigr\rangle\Bigr)\,ds\nl
&=&\int_0^t g^{k,h}_{5}(s)\,ds+\int_0^t L(s,x)(q^{k,h}_s,0,0)\,ds,\qquad t\in[0,\alpha],
\end{eqnarray*}
so \eref{e368} yields
$$
|q^{k,h}(t)|
\leq \int_0^t |g^{k,h}_{5}(s)|\,ds+ \int_0^t |L(s,x)(q^{k,h}_s,0,0)|\,ds
\leq \int_0^\alpha |g^{k,h}_{5}(s)|\,ds+L_1N_0\int_0^t |q^{k,h}_s|_C\,ds,
$$
for $t\in[0,\alpha]$.
Using that $q^{k,h}(t)=0$ for $t\in[-r,0]$, Lemma~\ref{l0} implies
$$
|q^{k,h}(t)|\leq |q^{k,h}_t|_C\leq N_1\int_0^\alpha |g^{k,h}_{5}(s)|\,ds,\qquad t\in[0,\alpha],
$$
where $N_1:=e^{L_1N_0\alpha}$.
Therefore \eref{e341} yields for $t\in[0,\alpha]$
$$
\lim_{k\to\infty}\sup_{h\neq 0\atop h\in\Gamma_2} \frac{|q^{k,h}(t)| }{|h|_{\Gamma_2}|h_k|_{\Gamma_2}}
\leq\lim_{k\to\infty}\sup_{h\neq 0\atop h\in\Gamma_2} \frac{|q^{k,h}_t|_C }{|h|_{\Gamma_2}|h_k|_{\Gamma_2}}
\leq\lim_{k\to\infty}\sup_{h\neq 0\atop h\in\Gamma_2} \frac{ N_1}{|h|_{\Gamma_2}|h_k|_{\Gamma_2}}\int_0^\alpha |g^{k,h}_{5}(s)|\,ds=0,
$$
which completes the proof of the second-order differentiability wrt parameters.
The continuity of $D_{22} x(t,\gamma)$ follows from Lemma~\ref{l32}.
\end{proof}
\bigskip

We note that the method  used in this section to prove the existence of the second order derivative $D_{22} x(t,\gamma)$
can not be used to prove the existence of the third order derivative, since some parts of the proof relied on the
assumption that the parameter $\gamma$ satisfies the compatibility condition $\gamma\in\calP$. 
The key step to show the existence of higher order derivatives is to  get rid of this assumption
in the proof of Theorem~\ref{t3}.

\end{document}